\documentclass[10pt, twoside]{article}
\usepackage{amsmath,amsthm,amssymb}
\usepackage{times}
\usepackage{enumerate}
\def\titlerunning#1{\gdef\titrun{#1}}
\makeatletter
\def\author#1{\gdef\autrun{\def\and{\unskip, }#1}\gdef\@author{#1}}
\def\address#1{{\def\and{\\\hspace*{18pt}}\renewcommand{\thefootnote}{}%
\footnote {#1}}%
\markboth{\autrun}{\titrun}}
\makeatother
\def\email#1{e-mail: #1}
\def\subjclass#1{{\renewcommand{\thefootnote}{}%
\footnote{\emph{Mathematics Subject Classification (2010):} #1}}}
\def\keywords#1{\par\medskip
\noindent\textbf{Keywords.} #1}
\DeclareMathOperator{\lcm}{lcm}

\usepackage{mathtools}

\mathtoolsset{showonlyrefs}

\newtheorem{thm}{Theorem}[section]
\newtheorem{cor}[thm]{Corollary}
\newtheorem{lem}[thm]{Lemma}

\newtheorem{conj}[thm]{Conjecture}
\newtheorem{cram}[thm]{Cram\'er Model}
\newtheorem{prop}[thm]{Proposition}

\theoremstyle{definition}

\newtheorem{rem}[thm]{Remark}

\renewcommand{\vec}[1]{\mathbf{#1}}

\numberwithin{equation}{section}

\begin{document}


\baselineskip=17pt
\setlength\parindent{0pt}


\titlerunning{The least prime number represented by a binary quadratic form}

\title{The least prime number represented by a binary quadratic form}

\author{Naser T. Sardari}

\date{}

\maketitle

\address{N. T. Sardari: University of Wisconsin Madison; \email{ntalebiz@math.wisc.edu}}


\begin{abstract} Let  $D<0$ be a fundamental discriminant and $h(D)$ be the class number of $\mathbb{Q}(\sqrt{D})$. Let $R(X,D)$ be the number of classes of the binary quadratic forms of discriminant $D$ which  represent a prime number in the interval $[X,2X]$. Moreover, assume that $\pi_{D}(X)$ is the number of primes, which split  in $\mathbb{Q}(\sqrt{D})$ with norm in the interval $[X,2X].$ We prove that
$$
\Big(\frac{\pi_D(X)}{\pi(X)}\Big)^2 \ll \frac{R(X,D)}{h(D)}\Big(1+\frac{h(D)}{\pi(X)}\Big),
$$
where $\pi(X)$ is the number of primes in the interval $[X,2X]$ and the implicit  constant in $\ll$ is independent of $D$ and $X$. 
\end{abstract}
\tableofcontents
\section{Introduction}
\subsection{Motivation}
In this paper, we consider the problem of giving the optimal upper bound  on the least prime number represented by a binary quadratic form in terms of its discriminant. Giving a sharp upper bound on the least prime number  represented by a binary quadratic form is crucial in the analysis of the complexity of some algorithms in quantum compiling. In particular, Ross and Selinger's algorithm for the optimal navigation of $z$-axis rotations in $SU(2)$ by quantum gates \cite{RS}, and its $p$-adic analogue for finding the shortest path between two diagonal vertices  of  LPS Ramanujan graphs \cite{complexity}. In \cite{complexity}, we proved that these heuristic algorithms run in polynomial time under a Cram\'er type conjecture on the distribution of the inverse image of  integers representable as a sum of two squares (or primes $p\equiv 1 \mod 4$ ) by a binary quadratic form; see \cite[Conjecture 1.4]{complexity}. Given a discriminant $D<0$, we show that this
Cram\'er type conjecture holds for  binary quadratic forms of discriminant $D$  with a positive probability  that only depends  on the density of the primes that split in $\mathbb{Q}(\sqrt{D})$. 

More precisely, let $\pi(X)$ be the number of primes with norm  in the interval $[X,2X].$   Let $D<0$ be a fundamental discriminant, which means $D$ is squarefree and  $D\equiv 1 \mod 4$.
 Let $\pi_{D}(X)$  be the number of  primes that split in  $\mathbb{Q}(\sqrt{D})$ with norm  in the interval $[X,2X]$. 
In \cite{Minkowski's}, we proved that for a given  fundamental discriminant $D$, by assuming the generalized Riemann hypothesis for the zeta function of the Hilbert class field of the imaginary  quadratic field $\mathbb{Q}(\sqrt{D})$,  $100\%$ of the binary quadratic forms of  discriminant $D$ represent a prime number less than $h(D)\log(|D|)^{2+\epsilon}$ as $|D|\to \infty$, where $h(D)$ is the class number of  $\mathbb{Q}(\sqrt{D})$. In this paper, we remove the GRH assumption and show  that unconditionally with probability at least $\alpha \Big(\frac{\pi_D(X)}{\pi(X)}\Big)^2 $ a binary quadratic forms of discriminant  $D<0$  represent a prime number smaller than any fixed scalar multiple of $h(D)\log(|D|)$, where $\alpha$ is an absolute constant independent of $D$.    
 As a result, we prove that if $\Big(\frac{\pi_D(X)}{\pi(X)}\Big)^2\gg 1 $ for some $X\sim h(D)\log(|D|)$ then a positive proportion of the binary quadratic forms of discriminant  $D<0$  represent a prime number smaller than any fixed scalar multiple of $h(D)\log(|D|).$ Next, we state  a form of our main theorem. Let  $R(X,D)$ be the number of  classes of  binary quadratic forms of discriminant $D$ which  represent a prime number in the interval $[X,2X]$. 
\begin{thm}\label{positive}
We have 
$$\Big(\frac{\pi_D(X)}{\pi(X)}\Big)^2 \ll \frac{R(X,D)}{h(D)}\Big(1+\frac{h(D)}{\pi(X)}\Big),$$
where  the implicit  constant in $\ll$ is independent of $D$ and $X$.
\end{thm}

 \begin{rem}
  Note that by Dirichlet's theorem,  we have  $\frac{\pi_D(X)}{\pi(X)}\sim 1/2 $ as $X\to \infty.$ By assuming the Riemann hypothesis or even a zero-free region of width  $O(\frac{\log\log (|D|)}{\log(|D|)}) )$ for the Dirichlet L-function $L(s,\chi_D)$, we have   $\frac{\pi_D(X)}{\pi(X)}\sim 1/2 $ for  any $X\gg |D|^{\epsilon}$ where $\epsilon>0.$ Since $ h(D)\gg |D|^{1/2-\epsilon},$ under GRH we have $\frac{\pi_D(X)}{\pi(X)}\sim 1/2$ for any $X\sim h(D)\log(|D|)$ and it follows that the above proposed algorithms give a probabilistic polynomial time algorithm for navigating $SU(2)$ and $PSL_2(\mathbb{Z}/q\mathbb{Z})$.
\end{rem}

Next, we show that our result is optimal up to a scalar. Namely, if a positive proportion of the binary quadratic forms of discriminant $D$ represent a prime number less than $X,$ then  
$h(D)\log{|D|}\ll X.$
We give a proof of this claim in what follows.  Let $H(D)$ be the genus class of the binary quadratic form of discriminant $D$ and $r(n,D)$ denote the sum of the representation of $n$ by all the classes of binary quadratic forms of discriminant $D$
$$r(n,D)=\sum_{Q\in H(D)}r(n,Q).$$
By the classical formula due to Dirichlet we have
\begin{equation}\label{dirichform}r(n,D)=w_D\sum_{d|n}\chi_{D}(d),\end{equation}
where,
$$w_D=\begin{cases} 6, \quad \text{ if }D=-3, \\  4, \quad \text{ if }D=-4, \\   2,\quad \text{ if } D<-4.  \end{cases}$$
This means that the multiplicity of representing a prime number $p$ by all the binary quadratic forms of a fixed negative discriminant $D<-4$ is bounded by $4$
\begin{equation}\label{primebound}
r(p,D) \leq 4.
\end{equation}
Assume that a positive proportion of binary quadratic forms represent a prime number smaller than $X$. Let $N(X,D)$ denote the number of pairs $(p,Q)$ such that $X<p<2X$ is a prime number  represented by $Q\in H(D)$. We proceed by giving a double counting formula for $N(X,D)$. By our assumption a positive proportion of binary quadratic forms of discriminant $D$ represent a prime number in the interval $[X,2X]$, then
\begin{equation}\label{low}h(D)\ll N(X,D).\end{equation}
On the other hand,
$$N(X,D)=\sum_{p<X} r(p,D).$$
By inequality \eqref{primebound},
$$N(X,D)\leq 4 \pi_D(X).$$ 
By the above inequality and inequality \eqref{low}, we obtain
$$h(D)\ll \pi_D(X).$$
By Siegel's lower bound $|D|^{1/2-\varepsilon}\ll h(D)$, it follows that  
$$h(D)\log(|D|)\ll X.$$
This completes the proof of our claim. 
\subsection{The generalized Minkowski's bound for the prime ideals}
It follows from our result that a positive proportion of the ideal classes of $\mathbb{Q}(\sqrt{D})$ contain a prime ideal with norm less than  the optimal bound   $h(D)\log(|D|).$
More precisely, let $H_{D}$ denote the ideal class group of $\mathbb{Q}(\sqrt{D})$ and $N_{\mathbb{Q}(\sqrt{D})} (x+y\sqrt{D})=x^2-Dy^2$ be the norm of the imaginary quadratic field $\mathbb{Q}(\sqrt{D}).$ Given an integral ideal $I\subset \mathcal{O}_{\mathbb{Q}(\sqrt{D})}$, let $q_{I}(x,y)$ be the following class of the integral binary quadratic form defined up to the action of $SL_2(\mathbb{Z})$ 
\begin{equation}\label{classis}
q_{I}(x,y):=\frac{N_{\mathbb{Q}(\sqrt{D})}(x\alpha+y\beta)}{N_{\mathbb{Q}(\sqrt{D})}(I)}\in \mathbb{Z},
\end{equation}
where $x,y\in \mathbb{Z}$, and $ I \cong \langle \alpha,\beta \rangle_{\mathbb{Z}}$ identifies the integral ideal $ I $ with $ \mathbb{Z}^2$.   It follows that $q_{I}$ only depends on the ideal class $[I]\in H_{D}.$ This gives an isomorphism between $H_{D}$ and the orbits of the integral binary quadratic forms of discriminant $D$ under the action of $SL_2(\mathbb{Z})$. Note that if $q_{I}$ represents the prime number $p$ then  $q_{I}(x,y)=p$ for some $x,y\in \mathbb{Z}$.  Then, the principal ideal $(x\alpha +y\beta)=IJ$ factors into the product of $I$ and $J$ where $N_{\mathbb{Q}(\sqrt{D})}(J)=p$ and  $J$ belongs to the inverse of the ideal class $[I]\in H_D$. Let $h_D(X)$ be the number of  ideal classes of $H_D$ that contain a prime ideal with norm  in the interval $[X,2X]$.   Hence, we have the following Corollary of Theorem~\ref{positive}.
\begin{cor}\label{positivecor}
We have 
$$\Big(\frac{\pi_D(X)}{\pi(X)}\Big)^2 \ll \frac{h_D(X)}{h(D)}\Big(1+\frac{h(D)}{\pi(X)}\Big),$$
 where  the implicit  constant in $\ll$ is independent of $D$ and $X$.
\end{cor}
More generally, let $K$ be a number field of bounded degree $n$ over $\mathbb{Q}$  with   discriminant $D_K$ and  class number $h_K$. Then we have the following conjecture  which generalizes Minkowski's bound for  prime ideals.  
 \begin{conj}
A positive proportion (depending  only on  $n$) of  ideal classes in the ideal class group of $K$ contain a prime ideal with norm  less than any fixed scalar multiple of $h_K\log(|D_K|)$. \end{conj}
Next, we show that these bounds are compatible with the random model for  prime numbers known as  Cram\'er's model. We cite the following formulation of the  Cram\'er model from~\cite{Soundist}. 
\begin{cram}
The primes behave like independent random variables $X(n)$ $ (n \geq 3)$ with $X(n) = 1$ (the number $n$ is `prime') with probability $1/\log n$, and $X(n) = 0$ (the number $n$ is `composite') with probability $1-1/\log n.$
\end{cram}
Note that each class of the integral binary quadratic forms is associated to  a Heegner point in $SL_2(\mathbb{Z}) \backslash \mathbb{H}$. By the equidistribution of Heegner points in $SL_2(\mathbb{Z}) \backslash \mathbb{H}$, it follows that almost all classes of the integral quadratic forms have a representative $Q(x,y):=Ax^2+Bxy+Cy^2$ such that the coefficients of $Q(x,y)$ are bounded by any function growing faster than $\sqrt{|D|}$:
$$\max(|A|,|B|,|C|)<\sqrt{|D|}\psi(D),$$
for any function $\psi(D)$ defined on integers such that $\psi(D)\to \infty$ as $|D| \to \infty$. We show this claim in what follows. We consider the set of representative of the Heegner points inside the Gauss fundamental domain of $SL_2(\mathbb{Z}) \backslash \mathbb{H}$ and denote them by $z_{\alpha}$ for $\alpha \in H(D)$. They are associated to  roots of representatives of binary quadratic forms in the ideal class group. By the equidistribution of Heegner points in $SL_2(\mathbb{Z}) \backslash \mathbb{H}$ and the fact that the volume of the  Gauss fundamental domain decays with rate $y^{-1}$ near the cusp, it follows that for almost all $\alpha \in H(D)$ if $z_{\alpha}=a+ib$ is the Heegner point inside the Gauss fundamental domain  associated to $\alpha$  then
\begin{equation}
\label{bineq}
|a| \leq 1/2, \text{ and }\sqrt{3}/2 \leq  b\leq \psi(D),
\end{equation}
where $\psi(D)$ is any function such that $\psi(D)\to \infty$ as $|D|\to \infty.$ 
Let $Q_{\alpha}(x,y):=Ax^2+Bxy+Cy^2$ be the quadratic form associated to $\alpha \in H(D)$ that has $z_{\alpha}$ as its root. Then 
$$z_{\alpha}=\frac{-B\pm i\sqrt{|D|}}{2A},$$
where $a=\frac{-B}{2A}$ and $b=\frac{\sqrt{|D|}}{2A}$.
By inequality~\eqref{bineq}, we have
\begin{equation}
|B|\leq |A|, \text{ and }
\frac{\sqrt{|D|}}{2\psi(D)} \leq A < \sqrt{|D|}.
\end{equation}
By the above inequalities and $D=B^2-4AC$, it follows that
\begin{equation}\label{hbd}\max(|A|,|B|,|C|)<\sqrt{|D|}\psi(D).\end{equation}
This concludes our claim. Next, we give a heuristic upper bound on the size of the smallest prime number represented by a binary quadratic forms of discriminant $D$ that satisfies \eqref{hbd}.  Since $D$ is squarefree, there is no local restriction for representing prime numbers.   So, by  Cram\'er's model and consideration of the Hardy-Littlewood local measures,  we expect that for a positive proportion of the classes of  binary quadratic forms $Q$ there exists an integral point $(a,b) \in \mathbb{Z}^2$ such that $|(a,b)|^2<L(1,\chi_{D})\log(|D|)$ and  $Q(a,b)$ is a prime number. We have 
\begin{equation}\label{Cramer}
\begin{split}
Q(a,b)&=Aa^2+Bab+Cb^2
\leq \max(|A|,|B|,|C|)|(a,b)|^2
\ll\sqrt{|D|}L(1,\chi_{D})\psi(D)\log(|D|).
\end{split}
\end{equation}
We may take $\psi(D)$ to be any constant in the above estimate. Therefore, we expect that a positive proportion of  quadratic forms of discriminant $D$ represent a prime number less than $h(D)\log(|D|).$ By a similar analysis, we expect that  almost all binary quadratic forms of discriminant $D$ represent a prime number less than 
$h(D)\log(|D|)^{2+\epsilon}.$ In other words, almost all ideal classes of $\mathbb{Q}(\sqrt{D})$  contain a prime ideal with norm less than $h(D)\log(|D|)^{2+\epsilon}.$ In \cite{Minkowski's}, we proved this result by assuming the generalized Riemann hypothesis for the zeta function of the Hilbert class field of the imaginary  quadratic field $\mathbb{Q}(\sqrt{D}).$ We  conjectured that  this type of generalized Minkowski bound holds for every number field. 
\begin{conj}\label{genmisk}
Almost all ideal classes in  the ideal class group of $K$ contain a prime ideal with norm  less than $h_{K}\log(|D_K|)^{A}$ for some $A>0$. Note that by the Brauer-Siegel Theorem and GRH,  we have $h_{K} \ll \sqrt{|D_K|}\log(|D_K|)^{\epsilon}.$  \end{conj}
\subsection{Repulsion of the prime ideals near the cusp }As we noted above, based on  Cram\'er's model we expect that the split prime numbers are randomly  distributed among the ideal classes of $\mathbb{Q}(\sqrt{D})$, and hence with a positive probability that is independent of $D$, a quadratic form of discriminant $D$ represent a prime number less than a fixed scalar multiple of $h(D)\log(|D|).$ We may hope that every ideal class contain a prime ideal of size $h(D)|D|^{\epsilon}$. Note that Cram\'er conjecture states that every short interval of size $\log(X)^{2+\epsilon}$ contains a prime number. By Linnik's conjecture, every congruence class modulo $q$ contains a prime number less than $q^{1+\epsilon}.$ This shows that small prime numbers cover all  short intervals and congruence classes. However, we note that the family of binary quadratic forms of discriminant $D<0$ is different from the family of short intervals and its $p$-adic analogue. 
Small primes do not cover all  the classes of binary quadratic forms. For example, the principal ideal class that is associated to the binary quadratic form $Q(x,y)=Dx^2+y^2$ repels prime numbers, which means the least prime number represented by this form is bigger than $|D|$ compared to $\sqrt{|D|}\log(|D|)^{2+\varepsilon}$ that is the upper bound for almost all binary quadratic forms  under GRH. This feature is different from the analogous conjectures for the size of the least prime number in a given congruence classes modulo an integer (Linnik's conjecture) and the distribution of prime numbers in short intervals (Cram\'er's conjecture). We call this new feature the repulsion of small primes by the cusp. In fact, the binary quadratic forms with the associated Heegner point near the cusp repels prime numbers. This can be seen in equation \eqref{Cramer}, where $\max(|A|,|B|,|C|)|$ could be as large as $|D|$ near the cusp but for a typical binary quadratic form it is bounded by $|D|^{1/2+\epsilon}.$ This shows that the bound in  Conjecture~\ref{genmisk} does not hold for every ideal class.

\subsection{Method of the proof}
Our method is based on our recent work on the distribution of  prime numbers  in the short intervals. In  \cite{cramer}, we proved that a positive proportion of the intervals of length equal to any fixed scalar multiple of $\log(X)$ in the dyadic interval $[X,2X]$ contain a prime number. We also showed  that a positive proportion of the congruence classes modulo $q$ contain a prime number smaller than any fixed scalar multiple of  $\varphi(q)\log(q).$   These results are compatible with Cram\'er's Model. 
 
We briefly describe our method here.  We proceed by introducing some new notations and follow the previous ones. Let $w(u)$ be a positive smooth weight function that is supported on $[1,2]$ and $\int w(u)du=1$. Let $w_X(u):=w(u/X)$ that is derived from $w(u)$ by scaling with $X$. 
  Let $\pi(Q,w,X)$ denote the number of primes weighted by  $w_X$ that are representable by the binary quadratic form $Q$. By the Cauchy-Schwarz inequality, we obtain
\begin{equation}\label{cauchybinary}
\Big(\sum_{Q\in H(D)} \pi(Q,w,X)     \Big)^2 \leq R(X,D)\Big( \sum_{Q\in H(D)}  \pi(Q,w,X)^2   \Big).
\end{equation}
By Dirichlet's formula in  \eqref{dirichform}, $\sum_{Q\in H(D)} \pi(Q,w,X) $ is the weighted number of  prime numbers inside the interval $[X,2X]$ that split in the quadratic field  $\mathbb{Q}(\sqrt{D})$. So, we have  
\begin{equation}\label{splitsim}
\pi_D(X)\sim \sum_{Q\in H(D)} \pi(Q,w,X).
\end{equation} 
Next, we give a double counting formula for the sum $ \sum_{Q\in H(D)}  \pi(Q,w,X)^2  $. This sum counts pairs of primes $(p_1,p_2)$ weighted by $w_X(p_1)w_X(p_2)$  such that $p_1$ and $p_2$ are represented by the same binary quadratic form class $[Q]\in H(D)$. Assume that $Q$ is a representative of that class that represents two prime numbers $p_1$ and $p_2$. Without loss of generality, we assume that $Q(x,y)=p_1x^2+\alpha xy +\beta y^2$ for some integers  $\alpha$ and $\beta$ such that 
\begin{equation}\label{disci}
D=\alpha^2-4p_1\beta,
\end{equation}
since by the action of $SL_2(\mathbb{Z})$ on the space of the integral binary quadratic forms we can find a representative of $Q$ with the above form. Since $Q$ represents $p_2,$ 
$ p_2= p_1u^2+\alpha uv +\beta v^2 $
for some integers $u$ and $v$. We multiply both side of the previous identity by $4p_1$ and obtain 
$$ 4p_1p_2= 4p_1^2u^2+4p_1\alpha uv +4p_1\beta v^2.  $$
We use identity~\eqref{disci}, and substitute $\alpha^2-D=4p_1\beta$ in the above identity and obtain 
$$ 4p_1p_2= 4p_1^2u^2+4p_1\alpha uv +(\alpha^2-D) v^2.$$
Hence, 
$$4p_1p_2=(2p_1u+\alpha v)^2-Dv^2.$$
We change variables to $s:=2p_1u+\alpha v$, and obtain 
\begin{equation}\label{prod}
4p_1p_2=s^2-Dv^2.
\end{equation}
%
%
%
 On the other hand if   $(p_1,p_2)$ is a solution to  the equation~(\ref{prod}) for prime numbers $X<p_1<2X$  and $X<p_2<2X$, then $p_1$ and $p_2$ are represented by the same binary quadratic form class in $H(D)$. Heuristically, this number is about $\frac{\pi_D(X)^2}{h(D)}+\pi_D(X)$, that is the number of  distinct pairs of split primes inside the interval $[X,2X]$ divided by the number of classes of binary quadratic forms of discriminant  $D$ plus the contribution of diagonal terms where $p_1=p_2$. Therefore, we expect  
\begin{equation}\label{eqqq}
\sum_{Q\in H(D)}  \pi(Q,w,X)^2   \approx \frac{\pi_D(X)^2}{h(D)} +\pi_D(X) . 
\end{equation}
In fact, by applying the Selberg upper bound sieve on the number of prime solutions $(p_1,p_2)$ to the equation~\eqref{prod}, we show that 
\begin{equation}\label{selup}
 \sum_{Q\in H(D)}  \pi(Q,w,X)^2   \ll \frac{\pi(X)^2}{h(D)}+ \pi(X).
\end{equation}
Therefore, by  inequality~\eqref{cauchybinary}, equation \eqref{splitsim} and the above inequality, it follows that 
$$
\Big(\frac{\pi_D(X)}{\pi(X)}\Big)^2 \ll \frac{R(X,D)}{h(D)}\Big(1+\frac{h(D)}{\pi(X)}\Big).
$$
This gives a proof of Theorem~\ref{positive}. Next, we briefly explain how we prove inequality~\eqref{selup}.  We begin by counting the number of  solutions $(p_1,p_2,s,v)$  to the equation~\eqref{prod} weighted by the smooth weight function $w_X$ where $v=0$. We call them the diagonal solutions.  If $v=0,$ then 
$4p_1p_2=s^2.$
Hence, $p_1=p_2=p$ for some prime number $p$ and $s=\pm2p$. Therefore, the number of diagonal solutions  to the equation~\eqref{prod} is the number of prime numbers weighted by  $w_X$ that is 
$\pi(w_X)\approx \pi(X).$
 Next, we give an upper bound on the number of non-diagonal terms  $v\neq 0$ weighted by $w_X(p_1)w_X(p_2)$. Since $D<0$ and $w_X(p_1)w_X(p_2)\neq 0$ only if $X<p_1,p_2<2X$  then
\begin{equation}\label{bound}
|s|\leq 4X, \text{ and }
|v|\leq \frac{4X}{\sqrt{|D|}}.
\end{equation}
We fix $v=v_0$ for some $0 \leq v_0 \leq   \frac{4X}{\sqrt{|D|}},$ and let  $m:=Dv_0^2<0.$ Let $\mathcal{P}(m,w_X)$ be the number of  prime solutions $(p_1,p_2)$  to $s^2-4p_1p_2=m$ weighted by $w_X(p_1)w_X(p_2).$ Let 
 \begin{equation}\label{nsol}
V_{m}:=\{(x,y,z): q(x,y,z)=m\},
\end{equation} 
 where $q(x,y,z):=z^2-4xy.$ Let  $Y:= |D|^{\delta},$ where $\delta>0$ is a fixed small power; e.g. $\delta<1/620$. 
 Let $S(m,Y,w_X)$ denote the number of integral solutions $(x,y,z)$ to the equation~\eqref{nsol} weighted by $w_X(x)w_X(y),$  where $x$ and $y$ do not have any  prime divisor smaller than $Y.$ If $X<\sqrt{|m|}/2,$ then $\mathcal{P}(m,w_X)=S(m,Y,w_X)=0.$ Otherwise, $X\geq \sqrt{|m|}/2 > Y,$ and we have $\mathcal{P}(m,w_X)\leq S(m,Y,w_X).$ We apply the Selberg upper bound  sieve to give  a sharp upper bound up to a constant on $S(m,Y,w_X).$ We briefly discuss the Selberg upper bound sieve in what follows.   Assume that $d_1,d_2$ and $d$ are squarefree integers. Let $\#_{w_X}A_{d_1,d_2}$ denote the number of  integral solutions $(x,y,z)$ to the equation~\eqref{nsol} weighted by $w_X(x)w_X(y),$
where  $d_1|x$ and  $d_2|y$. Similarly, let $\#_{w_X}A_{d}$ be the same number where $d|xy$. We write  $\#_{w_X}A$ for $\#_{w_X}A_{d}$ when $d=1.$  It follows from the  inclusion-exclusion principal that; see \cite[Lemma~8,~Page 79]{Fouvry} 
\begin{equation}\label{incex}
\#_{w_X}A_{d}=\mu(d)\sum_{\lcm[d_1,d_2]=d} \mu(d_1)\mu(d_2)\#_{w_X}A_{d_1,d_2}.
\end{equation}
Let $\chi_{Y}(.)$ be the indicator function of integers with no prime divisor less than $Y$. Let $\{\lambda_d\}$ be any sequence of real numbers for $d\geq 1$, where  $\lambda_1=1$. We have the following upper bound on $\chi_{Y}(n)$
\begin{equation}
\chi_{Y}(n)\leq \Big( \sum_{d|\gcd(n,\prod_{p<Y}p)} \lambda_{d}\Big)^2.
\end{equation} 
Hence,
\begin{equation}\label{form11}
\begin{split}
S(m,Y,w_X) = \sum_{z^2-4xy=m} \chi_{Y}(xy)w_{X}(x)w_{X}(y)
&\leq \sum_{z^2-4xy=m}  \Big( \sum_{d|\gcd(xy,\prod_{p<Y}p)} \lambda_{d}\Big)^2w_{X}(x)w_{X}(y)
\\
&=\sum_{d}\mu^{+}(d)\#_{w_X}A_{d},
\end{split}
\end{equation}
where
\begin{equation}\label{Selwei}
\mu^{+}(d):=\sum_{\lcm[d_1,d_2]=d}\lambda_{d_1}\lambda_{d_2}.
\end{equation}
%
%
In Theorem~\ref{quantitative}, we give an asymptotic formula for    $\#_{w_X}A_{d_1,d_2}$ with a power saving error term if $d_1d_2\leq |D|^{1/308}.$  The proof of this theorem is the main technical part of our work. We apply the Siegel Mass formula on the ternary quadratic  form $V_{m,d_1d_2}:=\left\{(x,y,z): z^2-d_1d_2xy=0\right\}$ in order to give the main term of $\#_{w_X}A_{d_1,d_2}$ as the product of Hardy-Littlewood local densities. For giving a power saving upper bound on the error term we  use the spectral theory of modular forms and  Duke's sub-convex upper bounds on the Fourier coefficients of weight $1/2$ Maass forms and Eisenstein series and our upper bound on the $L^2$ norm of the  theta lift of weight $1/2$ Maass forms. We give the outline of the proof of Theorem~\ref{quantitative} in the next section. By assuming these results
%
the main term of the weighted  number of integral points comes from the product of the local densities with a power saving error term Er
\begin{equation}\label{dukeforrr}\#_{w_X}A_{d_1,d_2}(m)=\sigma_{\infty,w_X}\prod_{p}\sigma_p(V_{m,d_1d_2})+ \text{Er},\end{equation}
where 
\begin{equation}\label{siinf} 
\begin{split}
&\sigma_{\infty,w_X}=\lim_{\epsilon \to 0} \frac{\int_{ m<z^2-4d_1d_2xy<m+\varepsilon} w_{X/d_1}(x)w_{X/d_2}(y) dx dy dz}{\epsilon},
\\
&\sigma_p(V_{m,d_1d_2}):=\lim_{k\to \infty}\frac{|V_{m,d_1d_2}(\mathbb{Z}/p^k\mathbb{Z})|}{p^{2k}}.
\end{split}
\end{equation}
We explicitly compute these local densities in terms of the quadratic character $\chi_{D}$ and as a result we have 
\begin{equation}\label{sd}
\#_{w_X}A_{d}=\#_{w_X}A \frac{\omega(d)}{d} +\text{ Er},
\end{equation}
where $\omega(.)$ is  explicit and is called  the  sieve density.   For a squarefree integer $l$, define  
 \begin{equation}\label{seldens}
 g (l):=\frac{\omega(l)}{l}\prod_{p|l}\Big(1-\frac{\omega(p)}{p}   \Big)^{-1},
 \end{equation}
 and let 
$
 G(Y):= \sum_{l=1}^{Y} g(l),
$
 where the sum is over squarefree integers $l.$ By the fundamental theorem for the  Selberg sieve \cite[Theorem~7.1]{Opera}, there exists a sequence  $\lambda_d\in \mathbb{R}$ with $\lambda_1=1$ such that 
  \begin{equation*}
 S(m,Y,w_X)\leq \sum_{d}\mu^{+}(d)\#_{w_X}A_{d} \leq \frac{\#_{w_X}A}{G(Y)}+\text{ Er}.
 \end{equation*}
In Lemma~\ref{selower}, we show that 
\begin{equation*}
L(1,\chi_D)^2 \log(|D|)^2 \frac{\varphi(v_0)}{v_0} \ll G(Y).
\end{equation*}
Finally, by summing over  $|v_0|\ll \frac{X}{\sqrt{|D|}}$ and proving the analogue of
Gallagher's result  on the average size of the Hardy-Littlewood singular series  \cite[equation (3)]{Gallagher}, we prove inequality~\eqref{selup} and hence Theorem~\ref{positive}.

\subsection{Outline of the paper} 
 In Section \ref{secpos},  we give the proof of Theorem~\ref{positive} by assuming Theorem~\ref{quantitative}.  In Lemma~\ref{sigmainf}, we compute $\sigma_{\infty,w_X}$ the Hardy-Littlewood measure at the archimedean place.  In Lemma~\ref{localde}, we give an explicit formula for $\sigma_p$ in terms of the quadratic character $\chi_{D}.$  In Lemma~\ref{uppermm}, we give  an explicit formula for $\#_{w_X}A_{d}$  involving $L(1,\chi_{D})$. In Lemma~\ref{wexpl}, we compute the sieve densities $\omega(d)$ defined in equation~\eqref{sd}. In Lemma~\ref{selower}, we give a sharp upper bound on the main term of the Selberg sieve. Finally, we prove the average size of these singular series is bounded (the analogue of Gallagher's theorem), and prove Theorem~\ref{positive}.

  In Section \ref{pduke}, we prove Theorem~\ref{quantitative} which implies equation \eqref{dukeforrr}. Let $q_k(x,y,z)=z^2-4kxy$, $V_{m,k}:=\{(x,y,z):q_k(x,y,z)=m \}$ and $\Gamma_k:=SO_{q_k}(\mathbb{Z})$ be the integral points of the orthogonal group of $q_k$. Then  $\Gamma_k$ is a lattice and   $\Gamma_k \backslash V_{m,k}$ has a natural  hyperbolic structure with finitely possible elliptic and  cusp points. We construct an automorphic function $W$ defined on $\Gamma_k \backslash V_{m,k}$  from the smooth function $w_X$. We spectrally expand  $W$   in the basis of eigenfunctions of the Laplace-Beltrami operator on $\Gamma_k \backslash V_{m,k}.$ We denote the contribution of  the constant function by the main term and  the contribution of the non-trivial eigenfunctions (Maass cusp forms and Eisenstein series of  $\Gamma_k \backslash V_{m,k}$) by  Er. 
     
 In Section~\ref{classnumber}, we prove a generalized class number formula in Proposition~\ref{pdt}.  This theorem gives the main term of  $\#_{w_X}A_{d_1,d_2}$ defined in equation \eqref{dukeforrr}.  
%
We briefly describe the proof of  Proposition~\ref{pdt}. The proof uses the Siegel Mass formula that gives a product formula for the sum of the representation number of an integer $n$ by a quadratic form averaged over the genus class of $q_k$. In  Lemma~\ref{genuslem}, we show that the genus class of $q_k(x,y,z)=z^2-4kxy$ contains only one element for every $k\in \mathbb{Z}$ and  Proposition~\ref{pdt} follows from  the Siegel Mass formula.%
  
  Our goal for the rest of Section \ref{pduke} is to give a power saving  upper bound on Er. This power  saving in the error term is crucial for the  application of the Selberg   sieve in Section~\ref{secpos}.  We write Er  as the sum of the low and the high frequency eigenfunctions in  the spectrum 
$$\text{Er}= \text{Er}_{\text{high}}+\text{Er}_{\text{low}}+\text{Er}_{\text{cts,low}},$$
where they are defined in \eqref{deflow}.

In Section~\ref{highsec}, we give an upper bound on the contribution of $\text{Er}_{\text{high}}.$ The upper bound follows from integration by parts. We show that    $\text{Er}_{\text{high}}=O(1)$. In Section~\ref{masidsec},  we prove an explicit form of the Maass identity that relates the Weyl sums to the Fourier  coefficients of the associated half-integral  weight Maass form obtained by the theta transfer using  the Siegel theta kernel.  In Section~\ref{lowfrq}, we  give an upper bound on $\text{Er}_{\text{low}}$. We apply Duke's non-trivial upper bound on the Fourier coefficients of the weight $1/2$  Maass form and the upper bound on the $L^2$ norm of the theta transfer of a Maass form  that we prove in Section~\ref{thetatrans} to bound  $\text{Er}_{\text{low}}$.
  
   In Section~\ref{cts,low}, we  give an upper bound on $\text{Er}_{\text{cts,low}}.$ We need to regularize the theta integral for bounding $\text{Er}_{\text{cts,low}}.$ We  use the center of the enveloping algebra (Casimir operator) for regularizing this theta integral. This method has been used in the work of Maass~\cite{Hans}, Deitmar and Krieg \cite{Deitmar} and Kudla and Rallis~\cite[Section 5]{Kudla}. Finally, we use  Duke's non-trivial upper bound on the Fourier coefficients of the weight $1/2$ Eisenstein series.

 There is a technical issue in using  Duke's result. The bound is exponentially growing in the eigenvalue aspect with the term $ \cosh(\pi t/2)$ for the half-integral weight eigenfunctions $\psi_{\lambda}$ with norm 1 and eigenvalue $1/4+t^2$.  We show that this term cancels with the exponentially decaying factor $\cosh(-\pi t/2)$ that appears  in $|\Theta*\psi_{\lambda}|^2$, the $L^2$ norm of the theta transfer of $\psi_{\lambda}$. This is the content of Section~\ref{thetatrans}.

  In Section \ref{thetatrans},  we  give an upper bound  on  the $L^2$ norm of the theta transfer of a weight $1/2$ Maass form $f$ in the eigenvalue and the level aspect up to a polynomial in these parameters. In Lemma~\ref{seesaw}, we compute the Mellin transform of the theta lift of $f$  by a see-saw identity that is originally due Niwa~\cite{Niwa} and used by Katok and Sarnak \cite{katok}. The see-saw idenity in this case identifies the Mellin transform  with the inner product of an Eisenstein series against the product of the weight 1/2 modular form $f$ and the complex conjugate of the Jacobi theta series $\bar{\theta}$. The last integral against the Eisenstein series is explicitly computable by  unfolding the Eisenstein series.   Hence, we obtain the Fourier coefficients of the theta transfer  at the cusp at infinity.  Finally, we bound the $L^2$ norm of a modular form by bounding the truncated sum of the squares of its Fourier coefficients; see \cite[Page 110, equation 8.17]{Iwaniec2}.  Note that the $L^2$ norm of the theta transfer of a new form is given by the Rallis-Inner product formula. Since we also deal with old forms, we rather use a more direct approach.

\section{Deducing Theorem~\ref{positive} from Theorem~\ref{quantitative}  via the Selberg sieve}\label{secpos}
Recall \eqref{dukeforrr}
\[
\#_{w_X}A_{d_1,d_2}(m)=\sigma_{\infty,w_X}\prod_{p}\sigma_p(V_{m,d_1d_2})+ \text{Er},
\]
where  $\sigma_{\infty,w_X}$ and $\sigma_p(V_{m,d_1d_2})$ were defined  in equation \eqref{siinf}. 
In this section,  we give the proof of Theorem~\ref{positive} by assuming the above formula and a power saving upper bound on Er.

\subsection{Local densities}\label{sievesec}
We proceed by computing the local densities $\sigma_{\infty,w_X}$ and $\sigma_p.$ Let 
 \begin{equation}
I(a):=\int_{1}^2\int_{1}^2   \frac{1}{2\sqrt{4x_1x_2+a}^{+}} w(x_1)w(x_2) dx_1 dx_2,
\end{equation} 
where $\sqrt{y}^{+}:=\begin{cases}\sqrt{y}, \quad &\text {  if } y>0,
\\
0, \quad  &\text{ otherwise.}
 \end{cases} $
\begin{lem}\label{sigmainf} 
We have 
\begin{equation}\sigma_{\infty,w_X}= \frac{X}{d_1d_2}I(\frac{m}{X^2}).\end{equation}
\end{lem}
\begin{proof}
We change the variables to $u:=d_1x$ and $v:=d_2y$ then
\begin{equation*}
\begin{split}
 \sigma_{\infty,w_X}&=\lim_{\epsilon \to 0} \frac{\int_{ m<z^2-4d_1d_2xy<m+\varepsilon} w_{X/d_1}(x)w_{X/d_2}(y) dx dy dz}{\epsilon}
 \\
 &=\frac{1}{d_1d_2}\lim_{\epsilon \to 0} \frac{\int_{ m<z^2-4uv<m+\varepsilon} w_{X}(u)w_{X}(v) du dv dz}{\epsilon}.
 \end{split}
\end{equation*}
Next, we scale the coordinates by $1/X$ and define $x_1=u/X$, $x_2=v/X$ and $x_3=z/X$. Hence,
\begin{equation*}
\begin{split}
 \sigma_{\infty,w_X}&=\frac{1}{d_1d_2}\lim_{\epsilon \to 0} \frac{\int_{ m<z^2-4uv<m+\varepsilon} w_{X}(u)w_{X}(v) du dv dz}{\epsilon}
 \\
 & =\frac{X}{d_1d_2}\lim_{\epsilon^{\prime} \to 0} \frac{\int_{ \frac{m}{X^2}<x_3^2-4x_1x_2<\frac{m}{X^2}+\varepsilon^{\prime}} w(x_1)w(x_2) dx_1 dx_2 dx_3}{\epsilon^{\prime}}
 \\
  & =\frac{X}{d_1d_2}\int_{1}^2\int_{1}^2   \frac{1}{2\sqrt{4x_1x_2+\frac{m}{X^2}}^{+}} w(x_1)w(x_2) dx_1 dx_2,
 \end{split}
\end{equation*}
where $\epsilon^{\prime}:=\frac{\epsilon}{X^2}.$ Then,
$\sigma_{\infty,w_X}= \frac{X}{d_1d_2}I(\frac{m}{X^2}).$
 It follows that $I$ is smooth  and is bounded  by a constant that only depends on the smooth function $w.$ 
 \end{proof}
Next, we compute explicitly, the local density $\sigma_p$ at each odd prime $p$. We have
 $$\sigma_p=\sum_{t=0}^{\infty} S(p^t),$$
where $S(1):=1$ and
$$S(p^t):=\frac{1}{p^{3t}} {\sum_{a}}^{\ast} \sum_{b}e\Big(\frac{a(q_{d_1d_2}(b)-n)}{p^t}\Big),$$
where $a$ runs over integers modulo $p^{t}$ with $\gcd(a,p)=1$, and  $b$ runs over vectors in $\mathbb{Z}^3$ modulo $p^t$. 
Since $p$ is an odd prime number, we can diagonalize our quadratic form $q_{d_1d_2}(X)$ over the local ring $\mathbb{Z}_{p}$ by changing the variables to $x_1=z$ , $x_2=x-y$ and $x_3=x+y$ and obtain
$$q_{d_1d_2}(x_1,x_2,x_3)=x_1^2+d_1d_2x_2^2-d_1d_2x_3^2.$$
 We apply the following lemma for the computation of local densities. For another versions for this lemma see; \cite[Lemma~3.1]{T.Sardari2017} and Blomer~\cite[(1.8)]{VB}. 
\begin{lem}\label{localdensss}
Let 
$$Q(x_1,x_2,x_3)=x_1^2+p^{\alpha}dx_2^2-p^{\alpha}dx_3^2,$$
where $\alpha \in \mathbb{Z}$ with $\alpha\geq 0$ and $d\in \mathbb{Z}_p$ with $\gcd(d,p)=1$.  Assume that $n=p^{\beta}n^{\prime}$ where $n^{\prime}\in \mathbb{Z}_p$ with $\gcd(n^{\prime},p)=1$. Let $V_n$ be the following quadric
$$V_{n}:=\left\{(x_1,x_2,x_3): Q(x_1,x_2,x_3)=n\right\},$$
defined over $\mathbb{Z}_p.$
Then
\begin{equation}\label{localdens}
\sigma_p(V_n):= \lim_{t\to \infty} \frac{V_n(\mathbb{Z}/p^t\mathbb{Z})}{p^{2t}}=1+\sum_{t=1}^{\infty} S(p^t),
\end{equation}
where $$S(p^t):=\frac{1}{p^{3t}} {\sum_{a}}^{\ast} \sum_{b}e\Big(\frac{a(Q(b)-n)}{p^t}\Big).$$
Moreover  if  $t$ is odd, then
\begin{equation}\label{odd}
S(p^t)= \begin{cases} \Big(\frac{n^{\prime}}{p} \Big)\frac{ p^{\min(\alpha+t,2t)}p^{t/2}}{p^{3t}}p^{t-\frac{1}{2}},\quad &\text{ if } \beta=t-1  ,\\
0, \quad &\text{ otherwise.}
\end{cases}
%
\end{equation}
\\
where $\Big(\dfrac{n^{\prime}}{p}  \Big)$ denote the Legendre symbol of $n^{\prime}$ modulo $p$, and if $t$ is even then
\begin{equation}\label{even}
S(p^t)=\begin{cases} 0,\quad   &\text{ if } \beta< t-1, \\ 
-\frac{p^{\min(\alpha+t,2t)}p^{t/2}}{p^{3t}}p^{t-1}, \quad &\text{  if  } \beta=t-1, \\
\frac{p^{\min(\alpha+t,2t)}p^{t/2}}{p^{3t}} \phi(p^t), \quad &\text{  if  } \beta\geq t.
\end{cases}
\end{equation}
\end{lem}
\begin{proof}
We compute 
\begin{equation*}
\begin{split}
S(p^t)&:=\frac{1}{p^{3t}}{\sum_{a}}^{\ast}\sum_{b\in{(\frac{\mathbb{Z}}{p^t\mathbb{Z}}})^3}e\Big(\frac{a(Q(b)-n)}{p^t}\Big)
\\
&=\frac{1}{p^{3t}}{\sum_{a}}^{\ast}\sum_{b\in{(\frac{\mathbb{Z}}{p^t\mathbb{Z}}})^3}e\Big(\frac{a(b_1^2+p^{\alpha}db_2^2-p^{\alpha}db_3^2-n)}{p^t}\Big)
\\
&=\frac{1}{p^{3t}}{\sum_{a}}^{\ast}e\Big(\frac{-an}{p^t}\Big)\prod_{i=1}^{3}\sum_{b \text{ mod } p^t}e\Big(\frac{aa_{i}p^{\alpha_i}b^2}{p^t}\Big),
\end{split}
\end{equation*}
where $a_1:=1$, $\alpha_1:=0$, $a_2:=d$, $\alpha_2:=\alpha$, $a_3:=-d$ and $\alpha_3=\alpha$. 
We note that the last summation is a Gauss sum. Let $G(h,m):=\sum_{x \text{ mod } m} e(\frac{hx^2}{m})$ be the Gauss sum, and let $\varepsilon_{m}=1$ if $m\equiv1 \text{ modulo } 4 $ and $\varepsilon_{m}=i$ if $m\equiv 3 \text{ modulo } 4 $. Then if $\gcd(h,m)=1$, we have

\begin{equation*}
G(h,m):=\begin{cases} 
\varepsilon_m \Big(\frac{h}{m} \Big) m^{1/2}, \quad & \text{ if } m \text{ is odd },\\
(1+\chi_{-4}(h))m^{1/2}, \quad  &\text{ if } m=4^{\alpha},\\
(\chi_8(h)+i\chi_{-8}(h)) m^{1/2}, \quad  &\text{ if } m=2.4^{\alpha}, \alpha\geq 1,
\end{cases}
\end{equation*}
%
where $\Big(\dfrac{h}{m} \Big)$ is the Jacobi symbol. We define $G(h,p^{t-\alpha_i}):=1$ when $t <\alpha_i $. We have
$$S(p^t)=\frac{1}{p^{3t}}{\sum_{a}}^{\ast}e\Big(\frac{-an}{p^t}\Big)\prod_{i=1}^{3}p^{\min(\alpha_i,t)}G(aa_i,p^{t-\alpha_i}).$$
We substitute the values of $G$ and obtain
\begin{equation*}
\begin{split}
S(p^t)=\frac{\prod_{i=1}^3 p^{\min(\frac{\alpha_i+t}{2},t)}\varepsilon_{p^{t-\alpha_i}}}{p^{3t}}{\sum_{a}}^{\ast}e\Big(\frac{-an}{p^t}\Big)\Big(\frac{a}{p} \Big)^{t}\Big(\frac{-1}{p} \Big)^{t-\alpha},
\end{split}
\end{equation*}
By our assumption we have $n=p^{\beta}n^{\prime}$, where $\gcd(n^{\prime},p)=1$. If $t$ is an odd number, then the inner sum is a Gauss sum, and we obtain
\begin{equation*}
  \mathop{{\sum}^\ast}_{a\; \mathrm{mod}\;p^t} e\Big(\frac{-ap^{\beta}n^{\prime}}{p^t}\Big)\Big(\frac{a}{p} \Big)=
 \begin{cases}
\varepsilon_{p} \Big(\frac{-n^{\prime}}{p} \Big) p^{t-\frac{1}{2}},\quad   &\text{  if  } \beta=t-1,\\
 0, \quad    &\text{  otherwise  }.
 \end{cases} 
\end{equation*}
Note $\varepsilon_{p}^2\Big(\frac{-1}{p} \Big)=1 $ and $\varepsilon_{p^{t-\alpha}}^2\Big(\frac{-1}{p} \Big)^{t-\alpha}=1 $.  Hence if $t$ is odd, we deduce that
\begin{equation}\label{odd}
S(p^t)= \begin{cases} \Big(\frac{n^{\prime}}{p} \Big)\frac{p^{\min(\alpha+t,2t) } p^{t/2}}{p^{3t}}p^{t-\frac{1}{2}}, \quad &\text{ if } \beta=t-1  ,\\
0, \quad &\text{ otherwise.}
\end{cases}
%
\end{equation}
On the other hand, if $t$ is even then the inner sum is a Ramanujan sum $c_{p^t}(n)$:
$$c_{p^t}(n)={\sum_{a}}^{\ast}e\Big(\frac{-an}{p^t}\Big)=\begin{cases} 0, \quad   &\text{ if } \beta< t-1, \\ 
-p^{t-1}, \quad &\text{  if  } \beta=t-1, \\
\phi(p^t), \quad &\text{  if  } \beta\geq t.
    \end{cases}$$
Hence if $t$ is even, it follows that
 
\begin{equation}\label{even}
S(p^t)=\begin{cases} 0,\quad   &\text{ if } \beta< t-1, \\ 
-\frac{\prod_{i=1}^3 p^{\min(\frac{\alpha_i+t}{2},t)}}{p^{tk}}p^{t-1}, \quad &\text{  if  } \beta=t-1, \\
\phi(p^t)\frac{\prod_{i=1}^3 p^{\min(\frac{\alpha_i+t}{2},t)}}{p^{3t}}, \quad &\text{  if  } \beta\geq t.
\end{cases}
\end{equation}
\end{proof}
In the following lemma, we apply Lemma~\ref{localdensss}  and give the explicit formula for the local densities $\sigma_p(V_{m,d_1d_2})$. 
\begin{lem}\label{localde}
Let 
$\alpha(d_1d_2):=\text{Ord}_p(d_1d_2)$, and $\beta(m):=\text{Ord}_p(m)$. Then, 
we have 
\begin{equation}
\sigma_p(V_{m,d_1d_2})=\begin{cases}
 1+\frac{1}{p}+\frac{\chi_D(p)}{p^{k+1}}-\frac{1}{p^{k+1}}, \quad &\text { if } \alpha(d_1d_2)=0 \text{ and } \beta(m)=2k,
 \\
2+\frac{\chi_D(p)}{p^k}-\frac{1}{p^k}, \quad   &\text { if } \alpha(d_1d_2)=1 \text{ and } \beta(m)=2k,
\\
p+1+\frac{\chi_D(p)}{p^{k-1}}-\frac{1}{p^{k-1}}, \quad  &\text { if } \alpha(d_1d_2)=2 \text{ and } \beta(m)=2k,
\\
1+\frac{1}{p}-\frac{1}{p^{k+1}}-\frac{1}{p^{k+2}}, \quad  &\text { if } \alpha(d_1d_2)=0 \text{ and } \beta(m)=2k+1,
\\
2-\frac{1}{p^{k}}-\frac{1}{p^{k+1}}, \quad &\text { if } \alpha(d_1d_2)=1 \text{ and } \beta(m)=2k+1,
\\
p+1-\frac{1}{p^{k-1}} -\frac{1}{p^{k}}, \quad &\text { if } \alpha(d_1d_2)=2 \text{ and } \beta(m)=2k+1.
\end{cases}
\end{equation}

\end{lem}

\begin{proof}
By Lemma~\ref{localdensss}, we have 
$$\sigma_p(V_{m,d_1d_2})=\sigma_p(\alpha(d_1d_2),\beta(m)).$$
 If $\alpha=0$ and $\beta=0$,  it follows that   $$\sigma_{p}(0,0)=1+\frac{\chi_{D}(p)}{p}.$$
 More generally, we have
\begin{equation}
\sigma(0,2k)=1+\frac{1}{p}+\frac{\chi_D(p)-1}{p^{k+1}}.
\end{equation}
Moreover, if  $\alpha=1$ or 2 and $\beta=0$   then $$\sigma_{p}(1,0)=\sigma_{p}(2,0)=1+\chi_{D}(p).$$
More generally, 
\begin{equation}
\sigma_p(1,2k)=2+\frac{\chi_D(p)}{p^k}-\frac{1}{p^k}.
\end{equation}
We also have for $k\geq 1$
\begin{equation}
\sigma_p(2,2k)=p+1+\frac{\chi_D(p)}{p^{k-1}}-\frac{1}{p^{k-1}}.
\end{equation}
Next, we compute the local densities for $\beta=2k+1$ and $\alpha=0, 1, 2.$ We have
\begin{equation*}
\begin{split}
\sigma(0,1)&=1-1/p^2,
\\
\sigma(1,1)&=1-1/p,
\\
\sigma(2,1)&=0.
\end{split}
\end{equation*}
In general, we have 
\begin{equation}
\begin{split}
\sigma(0,2k+1)&=1+\frac{1}{p}-\frac{1}{p^{k+1}}-\frac{1}{p^{k+2}},
\\
\sigma(1,2k+1)&=2-1/p^{k}-1/p^{k+1},
\\
\sigma(2,2k+1)&=1+p-\frac{1}{p^{k-1}} -\frac{1}{p^{k}}.
\end{split}
\end{equation}
\end{proof}
In the following lemma, we give an asymptotic formula for $\#_{w_X}A=\#_{w_X}A_{d_1,d_2}$ where $d_1=d_2=1$.
\\
\begin{lem}\label{uppermm}
We have
\begin{equation}\label{asysel}
\#_{w_X}A= XW(\frac{m}{X^2})  L(1,\chi_{D})\frac{6}{\pi^2} \prod_{\beta(p)\geq2}\Big(1-\frac{1}{p^2}    \Big)^{-1} \Big(1-\frac{\chi_D(p)}{p}\Big)\sigma_p+ \text{Er},
\end{equation}
 where $m=Dv_0^2.$ As a result,
 \begin{equation}\label{chom}
 \#_{w_X}A \ll XW(\frac{m}{X^2})  L(1,\chi_{D}) \big(\frac{v_0}{\varphi(v_0)}\big)^2.
 \end{equation}
\end{lem}
\begin{proof}
 By formula~\eqref{dukeforrr}, we have  
 $$\#_{w_X}A=\sigma_{\infty,w_X}\prod_{p}\sigma_p+ \text{Er}, $$
 where $\sigma_p=\sigma_p(\alpha,\beta)$ for $\alpha(p)=0$ and $\beta(p)=\text{Ord}_p(Dv_0^2)$. By Lemma~\ref{sigmainf} and~\ref{localde}, we have 
 \begin{equation*}
 \begin{split}
 \sigma_{\infty,w_X}&= XW(\frac{m}{X^2}),
 \\
 \sigma(0,0)&=\Big(1+\frac{\chi_{D}(p)}{p}    \Big),
 \\
  \sigma(0,1)&=\Big(1-\frac{1}{p^2}    \Big).
 \end{split}
 \end{equation*}
By substituting the above values in the product formula, we obtain 
\begin{equation*}
\#_{w_X}A=XW(\frac{m}{X^2})\prod_{\beta(p)=0} \Big(1+\frac{\chi_{D}(p)}{p}    \Big)\prod_{\beta(p)=1}\Big(1-\frac{1}{p^2}    \Big) \prod_{\beta(p)\geq 2}\sigma_p + \text{Er}.
\end{equation*}
We simplify the above product formula by applying the following  Euler product identities
\begin{equation*}
 L(1,\chi_{D})=\prod_{p} \Big(1-\frac{\chi_D(p)}{p}\Big)^{-1},
 \text{ and } 
 \prod_p \Big(1-\frac{1}{p^2}   \Big)=\frac{6}{\pi^2}.
\end{equation*}
Hence, we have 
\begin{equation*}
\begin{split}
\#_{w_X}A&=XW(\frac{m}{X^2})  L(1,\chi_{D})\prod_{p} \Big(1-\frac{\chi_D(p)}{p}\Big)
\\
&\times \prod_{\beta(p)=0} \Big(1+\frac{\chi_{D}(p)}{p}    \Big) \prod_{\beta(p)=1}\Big(1-\frac{1}{p^2}    \Big) \prod_{\beta(p)\geq 2}\sigma_p + \text{Er}
\\
&= XW(\frac{m}{X^2})  L(1,\chi_{D})\frac{6}{\pi^2} \prod_{\beta(p)\geq2}\Big(1-\frac{1}{p^2}    \Big)^{-1} \Big(1-\frac{\chi_D(p)}{p}\Big)\sigma_p+ \text{Er}.
\end{split}
\end{equation*}
This completes the proof of the identity~\eqref{asysel}.  By Lemma~\ref{localde} if $\beta(p)\geq 2$, then 
$$
\sigma_p=1+1/p+O(1/p^2).
$$
Hence,
\begin{equation}
\begin{split}
\#_{w_X}A \ll XW(\frac{m}{X^2})  L(1,\chi_{D}) \prod_{\beta(p)\geq2}\Big(1+\frac{2}{p}    \Big)
 \ll XW(\frac{m}{X^2})  L(1,\chi_{D}) \big(\frac{v_0}{\varphi(v_0)}\big)^2.
\end{split}
\end{equation}
This completes the proof of our lemma.
\end{proof}
Recall that from identity \eqref{incex}, we have 
 \begin{equation*}
\#_{w_X}A_{d}=\mu(d)\sum_{\lcm[d_1,d_2]=d} \mu(d_1)\mu(d_2)\#_{w_X}A_{d_1,d_2}.
\end{equation*}
In the following lemma, we give an asymptotic formula for $\#_{w_X}A_{d}$.
\begin{lem}\label{leveldiss}
We have 
\begin{equation}
\#_{w_X}A_{d}=\#_{w_X}A \frac{\omega(d)}{d} +\text{ Er, }
\end{equation}
where 
\begin{equation}\label{sieveden}
\omega(d)=\prod_{p|d} \frac{2\sigma_p(1,\beta)-\sigma_p(2,\beta)/p}{\sigma_p(0,\beta)}.
\end{equation}
\end{lem}
\begin{proof}
Let $d_1$ and $d_2$ be two squarefree integers. By product formula \eqref{dukeforrr}, we have  
 \begin{equation*}\#_{w_X}A_{d_1,d_2}(m)=\sigma_{\infty,w_X}\prod_{p}\sigma_p(\alpha,\beta)+ \text{Er},\end{equation*}
where $\alpha(p)=\text{Ord}_p(d_1d_2)$ and $\beta(p)=\text{Ord}_p(Dv_0^2). $ Hence, 
 \begin{equation*}\#_{w_X}A_{d_1,d_2}(m)=\frac{\#_{w_X}A}{d_1d_2} \prod_{p|d_1d_2} \frac{\sigma_p(\alpha,\beta)}{\sigma_p(0,\beta)}+ \text{Er}.\end{equation*}
We substitute the above product formula in  \eqref{incex} and  obtain 
\begin{equation*}
\begin{split}
\#_{w_X}A_{d}&=\mu(d)\sum_{\lcm[d_1,d_2]=d} \mu(d_1)\mu(d_2)\#_{w_X}A_{d_1,d_2}+Er
\\
&=\mu(d) \#_{w_X}A\sum_{\lcm[d_1,d_2]=d} \frac{\mu(d_1)\mu(d_2)}{d_1d_2}  \prod_{p|d_1d_2} \frac{\sigma_p(\alpha,\beta)}{\sigma_p(0,\beta)}+Er
\\
&=\frac{ \#_{w_X}A}{d}\prod_{p|d} \frac{2\sigma(1,\beta)-\sigma_p(2,\beta)/p}{\sigma_p(0,\beta)}+Er.
\end{split}
\end{equation*}
This completes the proof of our lemma. 
\end{proof}
In the following lemma, we give an explicit formula for $\omega(p)$ that is defined in \eqref{sieveden}. 
\begin{lem}\label{wexpl}
We have 
\begin{equation}
\omega(p)=\begin{cases}
\frac{2+2\chi_D(p)-1/p-\chi_D(p)/p}{1+\chi_D(p)/p}, \quad   &\text{  if } \beta(p)=0,
\\
\\
\frac{2}{1+1/p}, \quad   &\text{  if } \beta(p)=1,
\\
\\
\frac{3-1/p+\chi_D(p)/p^k-1/p^k}{1+1/p+\chi_D(p)/p^{k+1}-1/p^{k+1}}, \quad      &\text{  if } \beta(p)=2k  \text{ for } k\geq1,
\\
\\
\frac{3-1/p-3/p^k+2/p^{k+1}}{1+ 1/p - 1/p^{k+1}-1/p^{k+2}}, \quad     &\text{  if } \beta(p)=2k+1  \text{ for } k\geq1. 
\end{cases}   
\end{equation}
\end{lem}
\begin{proof}
By definition of $\omega(p)$ given in equation~\eqref{sieveden}, we have 
$$ \omega(p)=\frac{2\sigma_p(1,\beta)-\sigma_p(2,\beta)/p}{\sigma_p(0,\beta)}.$$
We substitute the explicit values of $\sigma_p(\alpha,\beta)$ from Lemma~\eqref{localde} and obtain the explicit values of $\omega(p).$
\end{proof}
Finally, we give an upper bound on the main term of the sieve. For a squarefree integer $l$, define  
 \begin{equation}\label{seldens}
 g (l):=\frac{\omega(l)}{l}\prod_{p|l}\Big(1-\frac{\omega(p)}{p}   \Big)^{-1},
 \end{equation}
 and let $
 G(Y):= \sum_{l=1}^{Y} g(l),
$
 where the sum is over square free variables $l.$
 In the following lemma, we give an asymptotic formula for $G(Y).$
\begin{lem}\label{selower}
Let $Y=|D|^{\delta}$ for some fixed $\delta>0$ and $G(Y)$ be as above. Then  
\begin{equation}
L(1,\chi_D)^2 \log(|D|)^2 \frac{\varphi(v_0)}{v_0} \ll_{\delta} G(Y).
\end{equation}
\end{lem}
\begin{proof} First, we estimate  the value of $g(p)$ at primes $p$. By equation \eqref{seldens}, we have
$$
g(p)=\frac{\omega(p)}{p-\omega(p)}\geq 0.
$$
By Lemma~\ref{wexpl}, we have 
\begin{equation}\label{gloc}
g(p)=\begin{cases}
\frac{2(1+\chi_D(p))}{p} + O(1/p^2), \quad   &\text{  if } \beta(p)=0,
\\
\\
\frac{2}{p}+O(1/p^2), \quad   &\text{  if } \beta(p)=1,
\\
\\
\frac{3}{p}+O(1/p^2), \quad      &\text{  if } \beta(p)=2k  \text{ for } k\geq1,
\\
\\
\frac{3}{p}+O(1/p^2), \quad     &\text{  if } \beta(p)=2k+1  \text{ for } k\geq1,
\end{cases}   
\end{equation}
where the implicit constant involved  in $O(1/p^2)$ is independent of all variables. Next, we apply  Rankin's trick and relate the 
truncated sum $G(Y)$ to an  Euler product. Note that 
$$
G(Y)\geq \sum_{\substack {n\\ 
p|n \implies p\leq Y^{1/10} }} \mu(n)^2 g(n) \big(\frac{1}{n^{10/\log(Y)}} -e^{-10}  \big).
$$
Then
$$
G(Y)\geq \prod_{p\leq Y^{1/10} } \big(1+ \frac{g(p)}{p^{10/\log(Y)}} \big) -e^{-10} \prod_{p\leq Y^{1/10} } \big(1+ g(p) \big).
$$
Since $\frac{\exp(x)}{1+x}$ is a monotone increasing function in $x\geq 0$, then we have 
\begin{equation*}
\begin{split}
\prod_{p\leq Y^{1/10} } \big(1+ g(p) \big)\big(1+ \frac{g(p)}{p^{10/\log(Y)}} \big)^{-1} &\leq \exp\big( \sum_{p\leq Y^{1/10}}  g(p)(1-\frac{1}{p^{10/\log(Y)}}) \big)
\\
&\leq \prod_{p\leq Y^{1/10} } \big( 4\sum_{p\leq Y^{1/10}}\frac{1}{p} (\frac{10 \log(p)}{\log(Y)}) \big) \sim e^{4},
\end{split}
\end{equation*}
where we used the prime number theorem and the fact that $g(p)\leq \frac{4}{p}.$
Hence, 
$$
G(Y)\geq 1/2 \prod_{p\leq Y^{1/10} } \big(1+ \frac{g(p)}{p^{10/\log(Y)}} \big). 
$$
Next, we complete the above Euler product by extending the product over primes $Y^{1/10}<p$ . Note that 
\[
 \prod_{ Y^{1/10}< p } \big(1+ \frac{g(p)}{p^{10/\log(Y)}} \big)  \leq \exp \big( \sum_{ Y^{1/10}< p } \frac{g(p)}{p^{10/\log(Y)}}   \big) 
 \leq  \exp \big( \sum_{ Y^{1/10}< p } \frac{4}{p^{1+10/\log(Y)}}   \big) 
 \leq 2\log(2),
\]
where we used the fact that $\sum_{p<N} \frac{1}{p}=\log\log(n)+O(1)$ and  $g(p)\leq \frac{4}{p}.$ Therefore, we have
\begin{equation}\label{sumdir}
G(Y)\gg \prod_{ p } \big(1+ \frac{g(p)}{p^{10/\log(Y)}} \big).
\end{equation}
Next, we complexify this Euler product and consider $G(s)$, the Dirichlet series associated to the multiplicative function $g$  
\begin{equation*}
G(s):=\sum_{l}\mu(l)^2g(l)l^{-s} =\prod_{p}\big(1+\frac{g(p)}{p^s}\big).
\end{equation*}
We write 
\begin{equation}\label{pff}
G(s)=\zeta(s+1)^2L(s+1,\chi_{-D})^2\eta(s)\tilde{G}(s),
\end{equation} 
where 
\[\eta(s)=\prod_{\beta(p)\geq 2}(1+\frac{g(p)}{p^s})(1-\frac{1}{p^{s+1}})^2(1-\frac{\chi_{-D}(p)}{p^{s+1}})^2,
\]
and 
\[
\tilde{G}(s)= \prod_{\beta(p)\leq 1}(1+\frac{g(p)}{p^s})(1-\frac{1}{p^{s+1}})^2(1-\frac{\chi_{-D}(p)}{p^{s+1}})^2.
\]
We analyze the Dirichlet series $\eta(s)$ and $\tilde{G}(s)$. First, we give an upper bound on the $|\eta(s)|.$ Recall that $\beta(p)=\text{Ord}_p(Dv_0^2)$ and $D$ is squarefree.
Let $p$ be a prime number 
such that $\beta(p)\geq 2.$  Hence, $p|v_0^2$ and by equation~\eqref{gloc}, we have
\begin{equation*}
\begin{split}
\eta(s)&=\prod_{p|v_0}(1+\frac{g(p)}{p^s})(1-\frac{1}{p^{s+1}})^2(1-\frac{\chi_{-D}(p)}{p^{s+1}})^2=\prod_{p|v_0}\big( 1+\frac{1-2\chi_{-D}(p)}{p^{s+1}}+O(\frac{1}{p^{s+2}}) \big).
\end{split}
\end{equation*}
Hence, for $\sigma>0$ we have 
$$
\eta(\sigma+it)\gg \prod_{p|v_0}\big( 1-\frac{1}{p} \big) =\frac{\varphi(v_0)}{v_0}.
$$
In particular,
\begin{equation}\label{etalo}
\eta(10/\log(Y))\gg \frac{\varphi(v_0)}{v_0}.
\end{equation}
 Next, we analyze $\tilde{G}(s)$. Assume that $p$ is a prime number such that $\beta(p)\leq 1$. By equation~\eqref{gloc}, it follows that 
\begin{equation}
(1+\frac{g(p)}{p^s})(1-\frac{1}{p^{s+1}})^2(1-\frac{\chi_{-D}(p)}{p^{s+1}})^2=1+O(\frac{1}{p^{s+2}}).
\end{equation}
Hence, 
\begin{equation}\label{gresidue}
\tilde{G}(s)\ll 1 \text{ and }
 \tilde{G}(s)^{-1}\ll1
\end{equation}
for $\Re(s)> -1+\epsilon,$ where the implicit constants depend only on $\epsilon>0.$ In particular, we have
$$
\tilde{G}(\frac{10}{\log(Y)})\ll 1.
$$
By \eqref{sumdir}, \eqref{pff}, \eqref{etalo} and the above inequality, it follows that 
$$
G(Y)\gg \zeta(1+\frac{10}{\log(Y)})^2L(1+\frac{10}{\log(Y)},\chi_{D})^2 \frac{\varphi(v_0)}{v_0}.
$$
Since  $Y=|D|^{\delta}$ then  $ \zeta(1+\frac{10}{\log(Y)})^2 \gg \big(\frac{\delta \log(|D|)}{10}\big)^2$ and  it follows that 
\begin{equation}\label{lastine}
G(Y)\gg_{\delta} \log(|D|)^2L(1+\frac{10}{\log(Y)},\chi_{D})^2 \frac{\varphi(v_0)}{v_0}.
\end{equation}
Finally, we make the observation that any completed $L$-function is monotone increasing in $\sigma\geq 1$. This is a consequence of the fact that all zero are to the left of 1. More precisely, for $D$ a negative discriminant one looks at 
$$
\Lambda (s,\chi_D):=\frac{|D|}{\pi}^{s/2}\Gamma(\frac{s+1}{2})L(s,\chi_D),
$$
then $\Lambda (\sigma,\chi_D)$ is monotone increasing in $\sigma\geq 1$. The proof is an application of the Hadamard factorization formula, which shows that 
$
 \Lambda (\sigma,\chi_D)=\prod_{\rho}|1-\frac{\sigma}{\rho}|,
$
and since all the zeros have real part in $(0,1)$  then  each term $|1-\sigma/\rho|$ is monotone increasing in $\sigma\geq 1.$ Therefore, 
$$
\Lambda(1,\chi_D) \leq \Lambda (1+\frac{10}{\log(Y)},\chi_{D}).
$$
In other words,
$$
L(1,\chi_D)\ll |D|^{5/\log(Y)} L(1+\frac{10}{\log(Y)},\chi_{D}).
$$
Since $Y=|D|^{\delta}$ then $|D|^{5/\log(Y)}=e^{5\delta}$. By the above inequality and \eqref{lastine}, we have 
$$
L(1,\chi_D)^2 \log(|D|)^2 \frac{\varphi(v_0)}{v_0} \ll_{\delta} G(Y).
$$
This completes the proof of our lemma. 
\end{proof}

%

\subsection{Proof of Theorem~\ref{positive} }
 \begin{proof} Recall that $S(m,Y,w_X)$ is the weighted number of integral solutions  $(x,y, z)$ to 
\(
z^2-4xy=m,
\)
 where $x$ and $y$ do not have a  prime divisor smaller than $Y$ and $m=Dv_0^2$. By inequality~\eqref{form11}, we have 
 \begin{equation}
\begin{split}
S(m,Y,w_X) \leq \sum_{d}\mu^{+}(d)\#_{w_X}A_{d}.
\end{split}
\end{equation}
 By the fundamental theorem for Selberg sieve \cite[Theorem~7.1]{Opera}, we have 
  \begin{equation*}
 S(m,Y,w_X) \leq \frac{\#_{w_X}A}{G(Y)}+O_{\epsilon}(X^{1-\epsilon}),
 \end{equation*}
 for some $\epsilon>0.$
By Lemma~\ref{uppermm}  and Lemma~\ref{selower}, we have 
 \begin{equation*}
 \begin{split}
 \#_{w_X}A \ll XW(\frac{m}{X^2})  L(1,\chi_{D}) \big(\frac{v_0}{\varphi(v_0)}\big)^2,
 \\
 L(1,\chi_D)^2 \log(|D|)^2 \frac{\varphi(v_0)}{v_0} \ll_{\delta} G(Y).
 \end{split}
 \end{equation*}
Therefore,
$$ S(m,Y,w_X) \ll \frac{XW(\frac{m}{X^2})}{\log(|D|)^2L(1,\chi_{D})}\big(\frac{v_0}{\varphi(v_0)}\big)^3,$$
where $m=Dv_0^2$.  By inequality~\eqref{bound}, we have $v_0\leq 4X/\sqrt{|D|}.$  We sum the above inequality for  $0\leq v_0\leq 4X/\sqrt{|D|},$ and obtain 
\begin{equation}
\begin{split}
 \sum_{Q\in H(D)}  \pi(Q,w,X)^2   &\ll \pi(X) + \sum_{1\leq v_0\leq 4X/\sqrt{|D|}}\frac{XW(\frac{Dv_0^2}{X^2})}{\log(|D|)^2L(1,\chi_{D})}\big(\frac{v_0}{\varphi(v_0)}\big)^3
\\
&\ll  \pi(X)+ \frac{X}{\log(|D|)^2L(1,\chi_D)}  \sum_{1\leq v_0\leq 4X/\sqrt{|D|}}W(\frac{Dv_0^2}{X^2})\big(\frac{v_0}{\varphi(v_0)}\big)^3.
\end{split}
\end{equation}
By lemma~\ref{sigmainf},  
$
W(\frac{Dv_0^2}{X^2})=O(1).
$
It is easy to check that 
$$
  \sum_{1\leq v_0\leq 4X/\sqrt{|D|}}\big(\frac{v_0}{\varphi(v_0)}\big)^3=O(X/\sqrt{|D|}).
$$
Therefore,  we obtain  
\begin{equation}
\begin{split}
\sum_{Q\in H(D)}  \pi(Q,w,X)^2   \ll \pi(X)+ \frac{X}{\log(|D|)^2L(1,\chi_D)}\frac{X}{\sqrt{|D|}}
\ll \pi(X)+\frac{\pi(X)^2}{h(D)}.
\end{split}
\end{equation}
This proves inequality~\eqref{selup} and concludes Theorem~\ref{positive}.
 \end{proof}
 
 \section{Quantitative  equidistribution of integral points on hyperboloids}\label{pduke}
Recall that $q(\vec{v}):=z^2-4xy$, where $\vec{v}:=(x,y,z),$ and  that  $V_{m}:=\{\vec{v}\in \mathbb{R}^3:  q(\vec{v})=m \},$ where $m:=Dv_0^2,$ and $D<0$ is a fundamental discriminant  and $v_0\leq \log(|D|)^{A}$ for some $A>0$. Assume that  $d_1$ and $d_2$ are integers.    Recall that $w(u)$ is a positive smooth weight function that is supported on $[1,2]$ and $\int_{u}w(u)du=1$. Let  $X\gg \sqrt{|m|}$ and $w_X(u):=w(u/X)$. Recall that  $\#_{w_X}A_{d_1,d_2}(m)$ is the number of integral points lying on 
$V_{m}(\mathbb{R})$
which are weighted by $w_X(x)w_X(y)$ such that $x$ and $y$ are divisible by $d_1$ and $d_2$, respectively. 
In this section,   we show  that
$$\#_{w_X}A_{d_1,d_2}(m)=\sigma_{\infty,w_X}\prod_{p} \sigma_p(V_m)+ \text{Er},$$
where $\sigma_{\infty,w_X}$ and $ \sigma_p(V_m)$  were  defined in~\eqref{siinf}
and Er is the error term that we bound in this section.  We briefly explain our method for bounding Er.  Recall that $q_k(x,y,z):=z^2-4kxy$, and $V_{m,k}:=\{ \vec{v}\in \mathbb{R}^3:  q_k(\vec{v})=m \}$ and $V_{m,k}(\mathbb{Z})$ is the set of integral points of $V_{m,k},$  where $k:=d_1d_2.$ Note that  $\#_{w_X}A_{d_1,d_2}(m)$ is  the number of integral points lying on $ V_{m,k}$ which are weighted by $w_{X/d_1}(x)w_{X/d_2}(y)$.  Let $\Gamma_k:=SO_{q_k}(\mathbb{Z})$ and consider the  surface   $\Gamma_k \backslash V_{m,k}.$ We equipped  $\Gamma_k \backslash V_{m,k}$ with the  hyperbolic metric. Let  $d\mu$ be the Haar measure induced from the hyperbolic metric, and let 
$\langle f, g\rangle :=\int_{\Gamma_k \backslash V_{m,k}} fg d\mu$ be the Petersson inner product.  
Let $\Delta$ be the Laplace-Beltrami  operator.  
 We assume that the reader is familiar with the spectral theory of $\Delta;$  see~\cite{Iwaniec1,Selberg}. Let $\mathcal{S}_k:=\{f_{\lambda}\in L^{2}(\Gamma_k \backslash V_{m,k}):  \Delta f_{\lambda}=\lambda f \}$ be an orthonormal basis of Maass cusp forms. Let $\mathcal{E}_k:=\{\mathfrak{a}: \mathfrak{a} \text{ ranges over all inequivalent cusps of }  \Gamma_k \backslash V_{m,k}\}$.
For $\mathfrak{a}\in \mathcal{E}_k$, let $\sigma_{\mathfrak{a}}$ be a scaling matrix associated to $\mathfrak{a}$, which is an isometry between $V_{m,k}$ and the upper half-plane $H$ such that $\sigma_{\mathfrak{a}}( \infty )= \mathfrak{a},$ and
$$
\sigma_{\mathfrak{a}}^{-1} \Gamma_{\mathfrak{a}}  \sigma_{\mathfrak{a}}=\left\{\begin{bmatrix} 1 & n \\ 0 &1 \end{bmatrix}: n\in \mathbb{Z}   \right\},
$$
where $\Gamma_{\mathfrak{a}}$ is the stabilizer of $\mathfrak{a}.$
For  $\mathfrak{a}\in \mathcal{E}_k$, we define the height function $y_{\mathfrak{a}}: V_{m,k} \to \mathbb{R}^{+}$ as:
$$y_{\mathfrak{a}} (\vec{v}):= \Im (\sigma_{\mathfrak{a}}^{-1} (\vec{v})).$$
For $\vec{v}\in \Gamma_k \backslash V_{m,k} $ and $s\in \mathbb{C}$, let $E_{\mathfrak{a}}(\vec{v},s)$ be the Eisenstein series such that its constant Fourier coefficient at cusp $\mathfrak{b}$ is $\delta_{\mathfrak{a}\mathfrak{b}} y_{\mathfrak{b}}^s+ \varphi_{\mathfrak{a}\mathfrak{b}}(s)y_{\mathfrak{b}}^{1-s},$ where   $\delta_{\mathfrak{a}\mathfrak{b}}=1$ if $\mathfrak{a}=\mathfrak{b}$ and $\delta_{\mathfrak{a}\mathfrak{b}}=0$ otherwise. We define the $\Gamma_k$ periodic function  $W$ on $\Gamma_k \backslash V_{m,k}$ by averaging the smooth weight function $w$ on  $\Gamma_k$ orbits
\begin{equation}\label{periodic}
W\big(\Gamma_k \vec{h}\big):=\sum_{\gamma\in \Gamma_k} w\big(\gamma \vec{h}\big).
\end{equation}
By Proposition~\ref{pdt}, the action of $\Gamma_k$ on $V_{m,k}(\mathbb{Z})$ has finitely many orbits. Let  $H_k(m) \subset \Gamma_k \backslash V_{m,k}(\mathbb{Z})$ be  the equivalence class of these orbits.   We have\begin{equation}\label{Siegelsum}
\begin{split}
\#_{w_X}A_{d_1,d_2}(m)&= \sum_{\vec{h}\in V_{m,k}(\mathbb{Z})} w(\vec{h})=\sum_{\Gamma_k \vec{h}\in H_k(m)} \frac{1}{|\Gamma_{k,\vec{h}}|} W(\Gamma_k \vec{h}),
\end{split}
\end{equation}
where $|\Gamma_{k,\vec{h}}|$ is the order of the stabilizer  of $\vec{h}.$  Define  the  $m$-th Weyl sum associated to a $\Gamma_k$ periodic function $f$ to be 
\begin{equation}\label{weylss}R(m,f):=\sum_{\Gamma_k \vec{h}\in H_k(m)} \frac{1}{|\Gamma_{k,\vec{h}}|} f(\Gamma_k \vec{h}).\end{equation}
Hence,
$\#_{w_X}A_{d_1,d_2}(m)=R(m,W).$ By the spectral theory of Maass forms developed by Selberg~\cite{Selberg}, we write $W$ 
in terms of Maass cusp forms,  Eisenstein series and the constant function, and obtain 
\begin{equation}\label{spectral}W(\vec{v})=\frac{\int_{\Gamma_k \backslash V_{m,k}} W d\mu }{\text{vol}(\Gamma_k \backslash V_{m,k}) } + \sum_{f_{\lambda}}\langle W,f_{\lambda}  \rangle f_{\lambda}(\vec{v})+W_{\text{cts}}(\vec{v}),   \end{equation}
where  the first term  comes from the contribution of the constant function, and
 $$W_\text{cts}(\vec{v}):= \sum_{\mathfrak{a}\in \mathcal{E}_k} \int_{-\infty}^{\infty} \langle W,E_{\mathfrak{a}}(.,1/2+it) \rangle E_{\mathfrak{a}}(\vec{v},1/2+it) dt. $$
By~\eqref{spectral}, we have  
\begin{equation}
R(m,W)=\frac{\int_{\Gamma_k \backslash V_{m,k}} W d\mu }{\text{vol}(\Gamma_k \backslash V_{m,k}) } \sum_{\Gamma_k \vec{h}\in H_k(m)} \frac{1}{|\Gamma_{k,\vec{h}}|}+\sum_{f_{\lambda}}\langle W,f_{\lambda}  \rangle R(m,f_{\lambda})+R(m,W_\text{cts}).
\end{equation}
Note that  $\sum_{\Gamma_k \vec{h}\in H(m)} \frac{1}{|\Gamma_{k,\vec{h}}|}$ is the  class number associated to the action of $\Gamma_k$ on $V_{m,k}(\mathbb{Z}).$  By Proposition~\ref{pdt}, the first term can be written as the product of  local densities, and we obtain 
$$\frac{\int_{\Gamma_k \backslash V_{m,k}} W d\mu }{\text{vol}(\Gamma_k \backslash V_{m,k}) } \sum_{\Gamma_k \vec{h}\in H(m)} \frac{1}{|\Gamma_{k,\vec{h}}|}= \frac{\int_{\Gamma_k \backslash V_{m,k}} W d\mu }{\text{vol}(\Gamma_k \backslash V_{m,k}) } \sigma_{\infty} \prod_{p} \sigma_p(V_{m,k}),$$
where $\sigma_p(V_{m,k}):=\lim_{l\to \infty}\frac{|V_{m,k}(\mathbb{Z}/p^l\mathbb{Z})|}{p^{2l}}$ and $\sigma_{\infty}:=  \int_{\Gamma \backslash V_{m,k}}  d\sigma_{\infty}  .$  Therefore, 

\begin{equation}\label{dukefor}\#_{w_X}A_{d_1,d_2}(m)=\sigma_{\infty,w_X}  \prod_{p} \sigma_p(V_{m,k})+ \text{Er},\end{equation}
where 
$
\text{Er}:=\sum_{f_{\lambda}}\langle f_{\lambda},W  \rangle R(m,f_{\lambda})+R(m,W_\text{cts}).
$
Our goal in this section is to give an upper bound on Er. Let $T:=|m|^{\delta}$ for some $\delta>0$.  We write
$\text{Er}=  \text{Er}_{\text{high}}+\text{Er}_{\text{low}}+\text{Er}_{\text{cts,low}},$
where 
\begin{equation}\label{deflow}
\begin{split}
\text{Er}_{\text{high}}&:= \sum_{\lambda\geq T}\langle W,f_{\lambda}  \rangle R(m,f_{\lambda})+  \sum_{\mathfrak{a}\in \mathcal{E}_k} \int_{|1/4+t^2|>T} \langle W, E_{\mathfrak{a}}(.,1/2+it) \rangle R(m, E_{\mathfrak{a}}(\vec{v},1/2+it)) dt,
\\
 \text{Er}_{\text{low}}&:=\sum_{\lambda<T}\langle W, f_{\lambda}  \rangle R(m,f_{\lambda}),
 \\
\text{Er}_{\text{cts, low}} &:=  \sum_{\mathfrak{a}\in \mathcal{E}_k} \int_{|1/4+t^2|<T} \langle W,E_{\mathfrak{a}}(.,1/2+it) \rangle R(m, E_{\mathfrak{a}}(\vec{v},1/2+it)) dt.
\end{split}
\end{equation}
\begin{thm}\label{quantitative}
Let $D$ be a fundamental discriminant and $m=Dv_0^2$ where $v_0<\log(|D|)^A$ for some fixed $A>0.$  Let $\#_{w_X}A_{d_1,d_2}(m)$ be as above. Then, for every $\epsilon>0$ we have 
\begin{equation}\label{quanform}
\#_{w_X}A_{d_1,d_2}(m)=\sigma_{\infty,w_X}  \prod_{p} \sigma_p(V_{m,k})+ O\big(1+ |m|^{-1/28} k^{10} X^{1+\epsilon} |D|^{\epsilon}\big).
\end{equation}
As a result, for every $0<\delta$ there exists an $0<\epsilon$ such that if $k^{308+\delta} \leq D$ and $ X\leq |D|^{1/2} \log(|D|)^B$ for some $B>0$ then 
\begin{equation}\label{sieveeror}
\#_{w_X}A_{d_1,d_2}(m)=\sigma_{\infty,w_X} \times \prod_{p} \sigma_p(V_{m,k})+ O\big(1+ \frac{X}{d_1d_2}|D|^{-\epsilon}\big),
\end{equation}
 where $k=d_1d_2$ and the implicit constant in $O$ depends only  on $\epsilon$ and  $w$.
\end{thm}
\begin{proof}
By equation \eqref{dukefor}, we have 
$$
\#_{w_X}A_{d_1,d_2}(m)=\sigma_{\infty,w_X}  \prod_{p} \sigma_p(V_{m,k})+ \text{Er},
$$
where 
$\text{Er}=\text{Er}_{\text{high}}+ \text{Er}_{\text{low}}+\text{Er}_{\text{cts,low}} .$ By Proposition~\ref{highlem}, $\text{Er}_{\text{high}}=O(1), $ 
 where the implicit constant in $O$ depends only on  $\epsilon$ and $w.$ 
By Proposition~\ref{Erlowlem}, we have 
$$ \text{Er}_{\text{low}} \ll  |m|^{-1/28} k^{10} X^{1+\epsilon} T^{7}.$$
Let $T=|D|^{\epsilon/7}$, then 
$$
 \text{Er}_{\text{low}}=O\big( |m|^{-1/28} k^{10} X^{1+\epsilon}  |D|^{\epsilon}\big).
$$
By Proposition~\ref{Erctslow}
\[
\text{Er}_{\text{cts,low}}\ll  k^{6.5} T^{7/4}|m|^{-1/28+\epsilon}|X|^{1+\epsilon}=O\big( |m|^{-1/28} k^{10} X^{1+\epsilon}  |D|^{\epsilon}\big).
\]
Therefore,
 $$
 \#_{w_X}A_{d_1,d_2}(m)=\sigma_{\infty,w_X}  \prod_{p} \sigma_p(V_{m,k})+ \text{Er}+ O\big(1+|m|^{-1/28} k^{10} X^{1+\epsilon}  |D|^{\epsilon}\big).
 $$
 This completes the proof of equation \eqref{quanform}. If $k^{308+\delta} \leq D$ then 
 $$
 m^{-1/28} k^{10}=O\left(\frac{|D|^{-\delta/28}}{k}\right).
 $$
 Moreover if $ X\leq |D|^{1/2} \log(|D|)^B$, then 
\(
X^{1+\epsilon}=O(X|D|^{\epsilon}).
\)
By the above inequalities and  choosing $\epsilon$ small enough comparing to $\delta$, we conclude inequality~\eqref{sieveeror} and our Theorem. 
\end{proof}

 \subsection{Main term}\label{classnumber}
We define the generalized class number $h(k,m)$ to be  the number of  $\Gamma_k$ orbits of $V_{m,k}(\mathbb{Z})$ weighted by their representation number
\begin{equation}\label{genclassnum}
h(k,m):= \sum_{\Gamma_k \vec{h}\in H(m)} \frac{1}{|\Gamma_{k,\vec{h}}|} .
\end{equation}
We cite the following theorem  from~\cite[Chapter~15, Theorem~19]{Conway}.
 \begin{thm}[Due to Kneser, Earnest and Hsia]
If $Q$ is an indefinite integral quadratic form with at least 3 variables  and the genus of $Q$ contains more than one class, then for some prime number $p$, $Q$ can be $p$-adically diagonalized and the diagonal entries all involve distinct powers of $p$. 
\end{thm}
\begin{lem}\label{genuslem}
The genus of $q_k(x,y,z)$ contains only one class for  every $k\in\mathbb{Z}$.
\end{lem} 
 \begin{proof}
   We show this by computing  the local spinor norms; see \cite[Chapter 15]{Conway}. By the work of Kneser \cite{Kneser} on the computation of the local spinor norms for odd primes $p$ and its improvement by Earnest and Hsia \cite{Earnest, Earnest1} for prime 2, we have the following theorem that implies the genus of an indefinite quadratic forms contains only one class.
 We can diagonalize the quadratic form $q_k(x,y,z)$ over every the local ring $\mathbb{Z}_{p}$ where $p\neq 2$ by changing the variables to $x_1=z$ , $x_2=x-y$ and $x_3=x+y$ and obtain
$$q_k(x_1,x_2,x_3)=x_1^2+kx_2^2-kx_3^2.$$
 It is easy to check that check that $q_k$ does not satisfy the conditions of the above theorem and as a result the genus class of $q_k$  contains only one element. This completes the proof of our lemma.
 \end{proof}

 \begin{prop}\label{pdt}
We have 
 \begin{equation}\label{classnum}
 h(k,m):=\sigma_{\infty}\prod_{p} \sigma_p(V_{m,k}),
 \end{equation}
where $\sigma_p(V_{m,k})$ was defined in \eqref{siinf} and  
 $$\sigma_\infty:= \lim_{\epsilon \to 0} \frac{|\text{Vol} \left(\Gamma_k\backslash (|q_k(x,y,z)-m|<\epsilon)\right)}{ 2\epsilon}=\int_{\Gamma_k \backslash V_{m,k}} d\sigma_{\infty}.$$
 \end{prop}
\begin{proof} By Lemma~\ref{genuslem}, the genus class of $q_k$ contains only one class.  Next, we apply the Siegel Mass formula to the indefinite ternary quadratic from $q_k,$ and obtain 
\begin{equation}\label{Siegel}
\frac{1}{\int_{\Gamma_k \backslash V_{m,k}}  d\sigma_{\infty}}  \sum_{\Gamma_k \vec{h}\in H(m)} \frac{1}{|\Gamma_{k,\vec{h}}|}=\prod_{p} \sigma_p(V_{m,k}).
\end{equation} 
This completes the proof of our Proposition.
\end{proof}

%
%
%
%
%
%

\subsection{Bounding  $\text{Er}_{\text{high}}$}\label{highsec}

In this section we give an upper bound on $\text{Er}_{\text{high}}$. Note that $q_k(\vec{v})=\vec{v}^{\intercal}A_k\vec{v},$ where  $$A_k:=\begin{bmatrix}
0 & -2k& 0
\\
-2k& 0&0
\\
0&0&1
\end{bmatrix}.$$
Let
$C_k:=\begin{bmatrix}
1/2\sqrt{k}& 1/2\sqrt{k} & 0
\\ 
1/2\sqrt{k}&-1/2\sqrt{k}&0
\\
0&0&1
\end{bmatrix},
$
then 
$C_{k}^{\intercal}A_{k}C_{k}=\begin{bmatrix}-1&0&0 \\ 0&1&0\\ 0&0&1 \end{bmatrix}.$
We proceed by defining  the  Casimir operator of the orthogonal group $SO_{q_k}$ which induces $\Delta$ on $V_{m,k}.$ Let $X_1:=\begin{bmatrix}0&1&0\\1&0&0\\0&0&0    \end{bmatrix}$, $X_2:=\begin{bmatrix} 0&0&1\\0&0&0\\ 1 &0&0  \end{bmatrix}$ and $X_3:=\begin{bmatrix}0&0&0\\ 0&0&1\\0 & -1&0   \end{bmatrix}$. By the definition of the Casimir operator of the orthogonal group $SO_{q_k}$:
\begin{equation}
\Omega:=Y_1^2+Y_2^2-Y_3^3,
\end{equation}  
where $Y_1:=C_kX_1C^{-1}_k$, $Y_2:=C_kX_2C^{-1}_k$ and $Y_3:=C_kX_3C^{-1}_k$. 
In the following lemma, we give a formula for the Casimir operator $\Omega$ in terms of the $(x,y,z)$ coordinates of the the quartic $V_{m,k}$. 
\begin{lem}
The restriction of $\Omega$ to $V_{m,k}$ is given by 
\begin{equation}\label{casimir}
\begin{split}
\Omega&=x^2\frac{\partial^2}{\partial x^2}+2x\frac{\partial}{\partial x}+\frac{4kxy+2m}{2k}\frac{\partial^2}{\partial x\partial y}+2xz\frac{\partial^2}{\partial x\partial z}
\\
&+y^2\frac{\partial^2}{ \partial y^2}+ 2y\frac{\partial}{\partial y}+ 2yz \frac{\partial^2}{\partial y \partial z }+(z^2-m) \frac{\partial^2}{\partial z^2} + 2z\frac{\partial}{ \partial z}.
\end{split}
\end{equation}
\end{lem}
\begin{proof}
We compute the induced  first order differential operators associated to $Y_1$, $Y_2$ and $Y_3$. Note that
$$Y_1:=C_kX_1C^{-1}_k=\begin{bmatrix}1&0&0 \\0&-1&0\\0&0&0 \end{bmatrix}.$$
This vector is associated to the following first order differential operator 
\begin{equation}\label{1d}Z_1:=x\frac{\partial}{\partial x} -y\frac{\partial}{\partial y}.\end{equation}
Similarly 
$Y_2:=C_kX_2C^{-1}_k=\begin{bmatrix}0&0&1/2\sqrt{k} \\ 0&0& -1/2\sqrt{k} \\ \sqrt{k}& \sqrt{k}&0   \end{bmatrix}$
is associated to
\begin{equation}
Z_2:=z/(2\sqrt{k}) \frac{\partial}{\partial x} -y/(2\sqrt{k})\frac{\partial}{\partial y}+\sqrt{k}(x+y)\frac{\partial}{\partial z}, 
\end{equation}
and 
$Y_3:=C_kX_3C^{-1}_k=\begin{bmatrix} 0 & 0&1/2\sqrt{k} \\ 0 & 0& -1/2\sqrt{k}\\ -\sqrt{k}& \sqrt{k}&0     \end{bmatrix}$
is associated to 
$$Z_3:=z/(2\sqrt{k})\frac{\partial}{\partial x} -z/(2\sqrt{k}) \frac{\partial}{\partial y} +(y-x)\sqrt{k} \frac{\partial}{\partial z}.$$
The induced Casimir operator is given by 
$$Z_1^2+Z_2^2-Z_3^2.$$ 
We have 
\begin{equation}\label{Z1}
\begin{split}
Z_1^2&=\big(x\frac{\partial}{\partial x} -y\frac{\partial}{\partial y}\big)^2=x^2 \frac{\partial^2}{\partial x^2}+ x\frac{\partial}{\partial x}-2xy\frac{\partial^2}{\partial x\partial y}+y^2\frac{\partial^2}{\partial y^2}
+y\frac{\partial}{\partial y},
\end{split}
\end{equation}
\begin{equation}\label{Z2}
\begin{split}
Z_2^2&=\big(z/(2\sqrt{k}) \frac{\partial}{\partial x} +z/(2\sqrt{k})\frac{\partial}{\partial y}+\sqrt{k}(x+y)\frac{\partial}{\partial z}\big)^2
\\
&=z^2/4k \frac{\partial^2}{\partial x^2}+z^2/2k \frac{\partial^2}{\partial x \partial y}+z(x+y)\frac{\partial}{\partial x \partial z}+z/2\frac{\partial}{\partial z}+ (x+y)/2 \frac{\partial}{\partial x}
\\ &+ z^2/4k \frac{\partial^2}{\partial y^2}+z(x+y)\frac{\partial^2}{\partial y \partial z}+z/2\frac{\partial}{\partial z}+(x+y)/2
\frac{\partial}{\partial y}+k(x+y)^2\frac{\partial^2}{\partial z^2},
\end{split}
\end{equation}
and 
\begin{equation}\label{Z3}
\begin{split}
Z_3^2&= \big(z/2\sqrt{k} \frac{\partial}{\partial x}-z/2\sqrt{k}\frac{\partial}{\partial y}+(y-x)\sqrt{k}\frac{\partial}{\partial z} \big)^2
\\
&=z^2/4k \frac{\partial^2}{\partial x^2} -z^2/2k \frac{\partial^2}{\partial x \partial y}+z(y-x)\frac{\partial^2}{\partial x \partial z} -z/2\frac{\partial}{\partial z}+(y-x)/2\frac{\partial}{\partial x}
\\
&+z^2/4k \frac{\partial^2}{\partial y^2}-z(y-x) \frac{\partial^2}{\partial y \partial z}-z/2 \frac{\partial}{\partial z}-(y-x)/2\frac{\partial}{\partial y} +(y-x)^2k\frac{\partial^2}{\partial z^2}.
\end{split}
\end{equation}
By using the formulas in \eqref{Z1}, \eqref{Z2} and \eqref{Z3}, have the following formula for the induced Casimir operator on $V_{m,k}$
\begin{equation*}
\begin{split}
\Omega&=x^2\frac{\partial^2}{\partial x^2}+2x\frac{\partial}{\partial x}+\frac{4kxy+2m}{2k}\frac{\partial^2}{\partial x\partial y}+2xz\frac{\partial^2}{\partial x\partial z}
\\
&+y^2\frac{\partial^2}{ \partial y^2}+ 2y\frac{\partial}{\partial y}+ 2yz \frac{\partial^2}{\partial y \partial z }+(z^2-m) \frac{\partial^2}{\partial z^2} + 2z\frac{\partial}{ \partial z}.
\end{split}
\end{equation*}
\end{proof}
\noindent In the following lemma, we prove an upper bound on the $L^2$ norm of $W$.
\begin{lem}\label{l2w}
Let $W$, $X$ and $k$ be as above.  Then
\begin{equation}
|W|_2 \ll \frac{X^{1+\epsilon}}{\sqrt{d_1m}}.
\end{equation}

\end{lem}
\begin{proof} We have
\begin{equation}\label{l2l1}
\begin{split}
|W|_2^2 &=\int_{\Gamma_k \backslash V_{m,k}} |W|^2 d\mu \leq \sup  | W| \int_{\Gamma_k \backslash V_{m,k}} |W| d\mu.
\end{split}
\end{equation} First, we give an upper bound on $ \int_{\Gamma_k \backslash V_{m,k}} |W| d\mu.$  Recall that 
$
W\big(\Gamma (x,y,z)\big):=\sum_{\gamma\in \Gamma_k} w\big(\gamma_k (x,y,z)\big),
$
where
$
w(x,y,x):=w_{X_1}(x)w_{X_2}(y),
$
and  $X_1=\frac{X}{d_1}$, $X_2=\frac{X}{d_2}$, and $w_{X}(u):=w(u/X)$. Note that  $d\mu=\frac{1}{\sqrt{m}}d\sigma_{\infty}$. Hence, by Lemma \ref{sigmainf}, we have
\begin{equation}\label{inint1}
\begin{split} \int_{\Gamma_k \backslash V_{m,k}} |W| d\mu \leq  \int_{ V_{m,k}} | w| d\mu
 \ll   \frac{X}{\sqrt{m}d_1d_2}=\frac{X}{k\sqrt{m}}.
\end{split}
\end{equation}
Next, we give an upper bound on $\sup  |W|$. Let 
$$
B(X_1,X_2):=\{(x,y,z)\in V_{m,k}(\mathbb{R}): X_1 \leq x \leq 2X_1 \text{ and } X_2 \leq y \leq 2X_2   \}.
$$
For $\vec{h}\in V_{m,k}$, define
$
N(X_1,X_2,\vec{h}):=\#\{\gamma\in \Gamma_k: \Gamma_k \vec{h}\in B(X_1,X_2)  \}.
$
Then,  
\begin{equation}\label{recinq}
 W(\vec{h})=\sum_{\gamma\in \Gamma_k}   w\big(\gamma \vec{h}\big)\ll N(X_1,X_2,\vec{h}).
\end{equation}
We give an upper bound on $N(X_1,X_2,\vec{h})$ by applying  results in hyperbolic geometry. Let $ \text{diam} ( B(X_1,X_2))$ be the diameter of $B(X_1,X_2)$ with respect to the hyperbolic metric on $V_{m,k}$. For $\vec{h} \in V_{m,k}$ define the invariant height of $\vec{h}$ by 
 $
 y_{\Gamma}(\vec{h})=\max_{\mathfrak{a}} ( y_{\mathfrak{a}}( \vec{h}) ).
 $
\begin{lem}\label{inheight}
We have $$\text{diam} ( B(X_1,X_2)) \ll 1+\log( \frac{X}{\sqrt{m}}),$$ and $$\sup_{\vec{h}\in  B(X_1,X_2)} y_{\Gamma}(\vec{h}) \ll d_2 \frac{X}{\sqrt{m}}. $$
\end{lem}
\begin{proof}
 Let $\mathfrak{c}:= \begin{bmatrix}0\\1\\0\end{bmatrix},$ which is a cusp for $\Gamma_k \backslash V_{m,k}$. 
Consider the following change of coordinates  
\[
 u_1:=\frac{d_1x_1}{\sqrt{|m|}},
 u_2:=\frac{d_2x_2}{\sqrt{|m|}}, \text{ and }u_3:=\frac{x_3}{\sqrt{m}}.
\]
 Then $V_{m,k}$ maps to $u_3^2-4u_1u_2=-1,$ and $B(X_1,X_2)$ maps to 
 $$
 B(X,m):=\left\{(u_1,u_2,u_3): u_3^2-4u_1u_2=-1,  \frac{X}{\sqrt{m}} \leq u_1 \leq  2\frac{X}{\sqrt{m}} \text{ and } \frac{X}{\sqrt{m}} \leq u_2 \leq 2\frac{X}{\sqrt{m}}  \right\}.
 $$  
  The quartic $u_3^2-4u_1u_2=-1$ with its induced metric $(du_3)^2-4du_1du_2$ is isomorphic to the hyperbolic plane by 
  $
  \beta: (u_1,u_2,u_3) \to \frac{-u_3+i}{2u_1}.
  $ This maps $\mathfrak{c}$ to $\infty.$
 Since  $X\gg \sqrt{m}$, we have   $ |\frac{u_3}{u_1}|\ll1$. It follows that 
   \begin{equation*}
\text{diam} ( B(X_1,X_2))= \text{diam} ( B(X,m)) \ll 1+\log(|\frac{X}{\sqrt{m}}|).
\end{equation*}
It follows that that the stabilizer of $\mathfrak{c}$ in $\Gamma_k$ is
 $$
 \Gamma_{\mathfrak{c}}:= \left\{ \begin{bmatrix}
1 &0 &0 
\\
n^2k &1& n
\\
2kn&0&1 
\end{bmatrix}: n\in \mathbb{Z} \right\}.
 $$
For  $a\in \mathbb{R},$ let
\(
\alpha_a:=\begin{bmatrix}
1 &0 &0 
\\
a^2\frac{d_1}{d_2} &1& a/d_2
\\
2ad_1&0&1 
\end{bmatrix}.
 \)
 $\beta$ indentifies $V_{m,k}$ with the upper-half plane and maps $\alpha_a$ to $\begin{bmatrix} 1 & a \\0 &1 \end{bmatrix}.$ Hence, it identifies $\Gamma_{\mathfrak{c}}$ with
 \(
 \beta\left( \Gamma_{\mathfrak{c}} \right) = \left\{ \begin{bmatrix}
1 &d_2n 
\\
0&1 
\end{bmatrix}: n\in \mathbb{Z} \right\}.\)
Hence, $y_{\mathfrak{c}} (\vec{h})= \frac{1}{d_2} \Im\left( \beta(\vec{h}) \right)=\frac{1}{2d_2u_1}.$ Since $\frac{X}{\sqrt{m}} \leq u_1 \leq  2\frac{X}{\sqrt{m}},$
\(
  \frac{\sqrt{m}}{2X d_2} \leq y_{\mathfrak{c}} (\vec{h}) \leq  \frac{\sqrt{m}}{X d_2}
\)  for every $\vec{h}\in B(X_1,X_2).$  By Margulis' lemma and decomposing $\Gamma_k \backslash V_{m,k}$ into  thin and thick parts, it follows that $ y_{\mathfrak{c}} (\vec{h})  y_{\mathfrak{a}} (\vec{h}) \ll 1$ for every $\mathfrak{a}\neq \mathfrak{c}.$ Therefore, 
\(
 y_{\Gamma}(\vec{h}) \ll d_2 \frac{X}{\sqrt{m}}.
\)
This completes the proof of Lemma~\ref{inheight}.
\end{proof}
By \cite[Corollary 2.12~Page 52]{Iwaniec2}, we have 
$N(X_1,X_2,\vec{h})\ll  \text{diam} (  B(X_1,X_2)) \sup_{\vec{h}\in  B(X_1,X_2)} y_{\Gamma}(\vec{h}).$
Therefore, by Lemma~\ref{inheight}, we have 
\begin{equation}\label{boundN}N(X_1,X_2,\vec{h})\ll  d_2\frac{X^{1+\epsilon}}{\sqrt{m}}.   \end{equation}
By the above inequality and inequalities \eqref{recinq}, \eqref{inint1} and  \eqref{l2l1}, we obtain  
$$
|W|_2^2 \ll  \frac{X}{k\sqrt{m}}   d_2\frac{X^{1+\epsilon}}{\sqrt{m}} \ll  \frac{X^{2+\epsilon}}{d_1m}.
$$
This concludes Lemma~\ref{l2w}. 
\end{proof}

Next, by applying the integration by parts, we give an upper bound on the inner product of $W$ with the Maass forms and also the Eisenstein series.

\begin{lem}\label{fdecay}
Let $A>0$ be any positive integer. We have
\begin{equation}
\begin{split}
 \left[\sum_{\lambda\geq T}|\langle W, f_{\lambda}  \rangle|^2+  \sum_{\mathfrak{a}\in \mathcal{E}_k} \int_{|1/4+t^2|>T} |\langle W, E_{\mathfrak{a}}(.,1/2+it) \rangle|^2 dt\right]^{1/2}= O_{A}\left( \frac{X^{1+\epsilon}}{\sqrt{md_1}T^{A}}\right),
 \end{split}
\end{equation}
where the implicit constant in $O$ depends only on  $\sup_{1\leq n\leq A} d^{(n)}w.$
\end{lem}
\begin{proof}
By Plancherel theorem and integration by parts we have 
\begin{equation}\label{inner}
\begin{split}
&\sum_{\lambda\geq T}|\langle W,f_{\lambda}  \rangle|^2+  \sum_{\mathfrak{a}\in \mathcal{E}_k} \int_{|1/4+t^2|>T} |\langle W, E_{\mathfrak{a}}(.,1/2+it) \rangle|^2 dt
 \\&=\sum_{\lambda\geq T} \frac{1}{\lambda ^{2A}} |\langle W,  \Omega^A f_{\lambda}  \rangle|^2+  \sum_{\mathfrak{a}\in \mathcal{E}_k} \int_{|1/4+t^2|>T} \frac{1}{(1/4+t^2)^{2A}} \langle W, \Omega^A E_{\mathfrak{a}}(.,1/2+it) \rangle^2 dt
 \\
&\leq \frac{1}{T ^{2A}} \left[\sum_{\lambda\geq T}|\langle  \Omega^A W, f_{\lambda}  \rangle|^2+  \sum_{\mathfrak{a}\in \mathcal{E}_k} \int_{|1/4+t^2|>T} |\langle  \Omega^A W ,E_{\mathfrak{a}}(.,1/2+it) \rangle|^2 dt\right]
 \leq \frac{1}{T^{2A}} \int_{\Gamma_k \backslash V_{m,k}} | \Omega^n W|^2 d\mu.
\end{split}
\end{equation}
By a similar argument as in  Lemma~\ref{l2w}, we give an upper bound on $\int_{\Gamma_k \backslash V_{m,k}} | \Omega^n W|^2 d\mu$.
We have 
$$
\int_{\Gamma_k \backslash V_{m,k}} | \Omega^n W|^2 d\mu \leq \frac{\sup  | \Omega^n W|}{\lambda ^n} \int_{\Gamma_k \backslash V_{m,k}} | \Omega^n W| d\mu.
$$
We have
\begin{equation}\label{inint1}
\begin{split} \int_{\Gamma_k \backslash V_{m,k}} | \Omega^n W| d\mu &\leq  \int_{ V_{m,k}} | \Omega^n w| d\mu \leq \sup| \Omega^n w|  \int_{X_1\leq x\leq 2X_1} \int_{X_2\leq y\leq 2X_2} d\mu \ll \sup| \Omega^n w|  \frac{X}{d_1d_2\sqrt{m}}.
\end{split}
\end{equation}
We show that $\sup \Omega^n  w=O_n(1).$ Note that $w(x,y,z):=w_{X_1}(x)w_{X_2}(y)$ is independent of the $z$ variable. Therefore, all the partial derivatives that include $\frac{\partial}{\partial z}$ in formula \eqref{casimir} vanishes on $w$ and we obtain:
\begin{equation*}
\begin{split}
\Omega^n  w(x,y,z)&=\big(x^2\frac{\partial^2}{\partial x^2}+2x\frac{\partial}{\partial x}+(2xy+m/D)\frac{\partial^2}{\partial x\partial y}+y^2\frac{\partial^2}{ \partial y^2}+ 2y\frac{\partial}{\partial y}\big)^nw_{X_1}(x)w_{X_2}(y).
\end{split}
\end{equation*}
For $n=1$, we check that  $\Omega w$ is bounded by a constant. We have 
\begin{equation}\label{omw}
\begin{split}
\Omega w&=\frac{x^2}{X_1^2}w^{\prime\prime}(\frac{x}{X_1})w(\frac{y}{X_2})+\frac{2x}{X_1}w^{\prime}(\frac{x}{X_1})w(\frac{y}{X_2})+2\frac{x}{X_1}\frac{y}{X_2}w^{\prime}(\frac{x}{X_1})w^{\prime}(\frac{y}{X_2})
\\
&+\frac{m}{DX_1X_2}w^{\prime}(\frac{x}{X_1})w^{\prime}(\frac{y}{X_2})+\frac{y^2}{X_2^2}w(\frac{x}{X_1})w^{\prime\prime}(\frac{y}{X_2})+2\frac{y}{X_2}w(\frac{x}{X_1})w^{\prime}(\frac{y}{X_2}).
\end{split}
\end{equation} 
We assume that for every  $0 \leq n$ all the derivatives $\frac{d^k w}{dt^k}$ for $ 0 \leq k \leq n$ are bounded by a constant $|w|_{\infty,n}$. Since $w$ is supported inside $[1,2]$ then $ 1 \leq \frac{x}{X_1}, \frac{y}{X_2}\leq 2$, otherwise $\Omega w=0$. Since, $m<0$ and $z^2-4kxy=m$ then $m\leq 4kX_1X_2$ otherwise $V_{m,k}$ does not have any point where $|x|<2X_1$ and $|y|<2X_2$. By these assumptions we can bound each term in equation \eqref{omw} and obtain
$
|\Omega w|  \ll 1.
$
Similarly, for every $n$, it follows that
$\sup|\Omega^n w| \ll 1.$
We have
\begin{equation*}
\begin{split}
 \Omega^n W(\vec{h})&=\sum_{\gamma\in \Gamma} \Omega^n  w\big(\Gamma_k \vec{h}\big)\leq N(X_1,X_2,\vec{h})\sup \Omega^n  w.
 \end{split}
\end{equation*}
By inequality~\eqref{boundN}, we have 
$
N(X_1,X_2,\vec{h})\ll d_2\frac{X^{1+\epsilon}}{\sqrt{m}},
$
and hence 
$$
\sup|\Omega^n W(\vec{h})| \ll d_2\frac{X^{1+\epsilon}}{\sqrt{m}}.
$$
Therefore, by the above and inequality   \eqref{inint1}, we obtain 
\begin{equation}\label{l2upper}
\int_{\Gamma_k \backslash V_{m,k}} | \Omega^n W|^2 d\mu \ll\frac{X^{2+\epsilon}}{d_1m}.
\end{equation}
Finally, by the above and inequality  \eqref{inner}, we conclude our lemma. 
\end{proof}

Finally, we show that the contribution of the high frequency spectrum is bounded. Recall that
$$ \text{Er}_{\text{high}}:= \sum_{\lambda\geq T}\langle W, f_{\lambda}  \rangle R(m,f_{\lambda})+  \sum_{\mathfrak{a}\in \mathcal{E}_k} \int_{|1/4+t^2|>T} \langle W , E_{\mathfrak{a}}(.,1/2+it) \rangle R(m, E_{\mathfrak{a}}(\vec{v},1/2+it)) dt.$$
\begin{prop}\label{highlem}
Suppose that $T=D^{\delta}$ for some $\delta>0.$  Then
$$\text{Er}_{\text{high}}=O(1),$$ 
 where the implicit constant in $O$ depends on $\sup_{1\leq n\leq 10/\delta} d^{(n)}w$ and $\delta>0.$
\end{prop}
\begin{proof} First, we give an upper bound on  Weyl sums $R(m,f_{\lambda})$ and $R(m, E_{\mathfrak{a}}(\vec{v},1/2+it))$.  
We have 
\[
R(m,f_{\lambda})= \sum_{\Gamma_k \vec{h}\in H(m)} \frac{1}{|\Gamma_{k,\vec{h}}|} f_{\lambda}(\Gamma_k \vec{h})
| f_{\lambda}|_{\infty}  \sum_{\Gamma_k \vec{h}\in H(m)} \frac{1}{|\Gamma_{k,\vec{h}}|}
 \leq  |f_{\lambda}|_{\infty}h(k,m),
\]
where $k=d_1d_2.$ By the Weyl law we have the following trivial upper bound on the $L^{\infty}$ norm of an eigenfunction; see the recent work of Templier for a sharper upper bound \cite{Templier2015}
\[
 |f_{\lambda}|_{\infty} \ll \lambda^{1/4}k^{1/2}.
\]
By Proposition~\ref{pdt}, Lemma~\ref{sigmainf} and Lemma~\ref{localdensss}, we have 
\(
h(k,m) \ll \frac{X^{1+\epsilon}}{k}.
\)
Therefore, 
\[
|R(m,f_{\lambda})| \ll \frac{\lambda^{1/4}X^{1+\epsilon}}{k^{1/2}}.
\]
Similarly, we have 
\[
|R(m, E_{\mathfrak{a}}(\vec{v},1/2+it))| \ll \frac{(1/4+t^2)^{1/4}X^{1+\epsilon}}{k^{1/2}}.
\]
By Lemma~\ref{fdecay} and the above inequalities, we have
\[
 \text{Er}_{\text{high}}  \ll \sum_{\lambda>T}\frac{X^{1+\epsilon}}{\sqrt{md_1}\lambda^{A}}\frac{\lambda^{1/4}X^{1+\epsilon}}{k^{1/2}}
 \ll \frac{X^{2+\epsilon}}{k^{1/2}} \sum_{\lambda>T} \lambda^{1/4-A}.
\]
By Weyl law for $\Gamma_k\backslash V_{m,k}$, we have 
\[
\sum_{\lambda>T} \lambda^{1/4-A} \ll k|D|^{\delta(1+1/4-A)}.
\]
Recall that $X\ll |D|^{1/2+\epsilon} $, $k=d_1d_2\leq |D|^{1/10} $. Therefore, by choosing $A$ large enough, we obtain 
$
 \text{Er}_{\text{high}}  =O(1).
$
This completes the proof of our Proposition. 
%
\end{proof}

\subsection{The Maass identity and  the Siegel theta kernel}\label{masidsec}
Let $L \in C(\Gamma_k \backslash V_{m,k})$ be any continuous function which decays with an exponential rate at cusps, e.g. $L$ is a truncated Eisenstein series, a cusp form or any function with a compact support. 
 In this section, we write $R(m,L)$ in terms of the asymptotic of the  $m$-th Fourier coefficient of the theta transfer of  $L$.  We begin by introducing Siegel's  theta kernel associated to the indefinite quadratic form $q_k:=z^2-kxy$. Let $H_{A_k}$ denote the majorant space of the symmetric matrix $A_k$ (see \cite{Siegel}):
$$H_{A_k}:=\{P: P^{\intercal}=P, P>0 \text{ and } PA_k^{-1}P=A_k     \}.$$
For $P\in H_{A_k}$ and $z=x+iy\in \mathbb{C}$ with $y>0$, define $ R(z):=xA+iyP.$
The Siegel theta function is defined for $\boldsymbol{\alpha}\in \mathbb{Q}^3$ with $2A_k\boldsymbol{\alpha} \in \mathbb{Z}^3$ by 
\begin{equation}\label{theta}
\Theta_{\boldsymbol{\alpha}}(z,P):=y^{3/4}\sum_{\vec{h}\in \mathbb{Z}^3} e(R(z)[\vec{h}+\boldsymbol{\alpha}]),
\end{equation}
where $R(z)[\vec{h}+\boldsymbol{\alpha}]:=(\vec{h}+\boldsymbol{\alpha})^{\intercal}R(z)(\vec{h}+\boldsymbol{\alpha}).$ We write $\Theta(z,P)$ when $\boldsymbol{\alpha}=0.$ More generally,  let $A[B]:=B^{\intercal}AB$ for matrices $A$ and $B$. $\Theta_{\boldsymbol{\alpha}}(z,P)$ is absolutely convergent, since $y>0$ and $P>0$. We note that the orthogonal group $SO_{q_k}$ a\text{cts} transitively on the majorant space $H_{A_k}$ by sending $P\in H_{A_k}$ to $P[g]:=g^{\intercal}Pg$ for $g\in SO_{q_k}$. We define  $P_0\in H_{A_k}$ to be:
$$P_0:=\begin{bmatrix}2k & 0 &0 \\ 0 & 2k&0 \\ 0 & 0& 1    \end{bmatrix}.$$
 We extend the definition of the theta kernel from $H_{A_k}$ to $SO_{q_k}$ by defining:
\begin{equation}\label{thetaa}\Theta_{\boldsymbol{\alpha}}(z,g):=\Theta_{\boldsymbol{\alpha}}(z,P_0[g^{-1}]).\end{equation}
Note that we used $g^{-1}$ for transforming $P_0$. Next, we cite a theorem that gives the transformation properties of the theta kernel $\Theta_{\alpha}(z,P_0[g^{-1}])$ in $z$ variable. This theorem  is essentially due to Siegel \cite{Siegel2} and is stated in this form in \cite[Theorem 3]{Duke}. It is a consequence of the  properties of the Weil representation; see \cite[Proposition~2.2]{katok}.

\begin{thm}[\cite{Duke}, \cite{katok}]\label{dtrans}
For $\begin{bmatrix}a & b \\ c & d   \end{bmatrix}=\gamma\in \Gamma_0(4k)$ we have 
\begin{equation}\label{tkernelt}
\begin{split}
\Theta(\gamma z,g)=j(\gamma,z)\Theta( z,g),
\\
\Omega \Theta(z,g)=4\Delta_{z,1/2}\Theta(z,g)+\frac{3}{4}\Theta(z,g),
\end{split}
\end{equation}
where $j(\gamma,z)=\frac{\theta(\gamma z)}{\theta(z)}$ is the theta multiplier for $\theta(z)=y^{1/4}\sum_{n\in \mathbb{Z}}e(n^2z)$, and $\Delta_{z,1/2}$ is the laplacian operator defined on weight $1/2$ modular forms and  $\Omega$ is the Casimir operator.
\end{thm}
\begin{rem}\label{remtheta}
 By the above theorem it follows that if $f_{\lambda}$ is a cusp form with eigenvalues  $\lambda=1/4+(2r)^2$, then  $\Theta*f_{\lambda}(z):=\int \Theta(z,g) \bar{f}(g) f\mu(g)$ is a weight $1/2$ modular form defined on $\Gamma_{0}(4k)\backslash H$ with eigenvalues $\lambda^{\prime}=1/4+r^2$. 
\end{rem}

Note that $SO_{q_k}$ also a\text{cts} transitively on $V_{m,k}$. We define 
$\vec{x}_0:=\begin{bmatrix}1/2 \sqrt{|m|/k} \\ 1/2\sqrt{|m|/k} \\ 0  \end{bmatrix},$ and    extend the definition of  $L$ from $\Gamma_k \backslash V_{m,k}$ to $\Gamma_K \backslash SO_{q_k}$ by
$L(g):=L(g\vec{x}_0).$
It is easy to check that the stabilizer of $\vec{x}_0\in V_{m,k}$ is the same as $P_0\in H_A,$ and it is a maximal compact subgroup of $SO_{q_k}.$ We denote this maximal compact subgroup by $K$. Let 
$$F(z):=\Theta*L=\int_{\Gamma_K \backslash SO_{q_k}}\Theta(z,g)\bar{L}(g)d\mu.  $$
 Theorem~\ref{dtrans} implies that  $F(z)$ is a weight 1/2  modular forms of level $4k$ and has moderate growth. Let
$$
F(u+iv)=c_{F,\infty}(v)+\sum_{n\neq0}\rho_{F,\infty}(n,v)e(nu)
$$
be the Fourier expansion of $F$ at $\infty.$ Define the $m$-th Fourier coefficient of $F$ to be
\begin{equation}\label{deffour}
\rho_{F,\infty}(m):=\lim_{v\to \infty}\rho_{F,\infty}(m,v)e^{2\pi|m| v}(4\pi|m|v)^{-\text{sgn}(m)/4 }. 
\end{equation}
Next, we prove an identity that relates $\rho_{F,\infty}(m)$ to  $R(m,L)$. Originally, Maass~\cite{Hans} proved a version of this identity for the eigenfunctions of co-campact lattices. As noted by Duke~\cite[Theorem~6]{Duke} Maass' proof extends easily to  cusp forms, since the theta integral is convergent for cusp forms. However,  the theta integral is not absolutely convergent for Eisenstein series. For our application, we need to extend this identity for the Eisenstein series. In the next section, we prove the analogue of Maass' identity for the  Eisenstein series   by using the center of the enveloping algebra.
\begin{lem}\label{maasslem}
We have 
$
\rho_{F,\infty}(m)= \frac{\pi^{1/4}}{\sqrt{2}} |m|^{-3/4} \overline{R(m,L)}.
$
\end{lem}
\begin{proof}
Note that 
$
\rho_{F,\infty}(m,v)=\int_{0}^{1} F(u+iv)e(-mu)du.
$
We have
\begin{equation*}
\begin{split}
\rho_{F,\infty}(m,v)&=\int_{0}^{1} \int_{\Gamma_K \backslash SO_{q_k} } \Theta(u+iv,P_0[g^{-1}])\bar{L}(g) e(-mu) d\mu(g) du
\\
&=  v^{3/4}\int_{\Gamma_K \backslash SO_{q_k} }\int_{0}^{1} \sum_{\vec{h}\in \mathbb{Z}^3} e(uq_k(\vec{h})+ivP_0[g^{-1}\vec{h}]   )\bar{L}(g) e(-mu)du d\mu(g)
\\
&= v^{3/4}\int_{\Gamma_K \backslash SO_{q_k} }  \sum_{\substack {\vec{h}\in \mathbb{Z}^3\\ q_k(\vec{h})=m } } e(ivP_0[g^{-1}\vec{h}])\bar{L}(g) d\mu(g).
\end{split}
\end{equation*}
We unfold the above integral and write it as a finite sum over the integral orbits. Then
$$\rho_{F,\infty}(m,v)=\sum_{\vec{l}\in H(m)}  \frac{v^{3/4}}{|\Gamma_{k,\vec{l}}|} \int_{SO_{q_k}} e(ivP_0[g^{-1}\vec{l}])\bar{L}(g) d\mu(g). $$
Next,  we use Fubini's theorem and write the above integral over the ternary quadric $V_{m,k}$ with its invariant measure induced from the transitive action of $SO_{q_k}$ on $V_{m,k}$. Recall that $$\vec{x}_0:=\begin{bmatrix} \sqrt{|m|/4k} \\ \sqrt{|m|/4k} \\ 0  \end{bmatrix}.$$
Since $SO_{q_k}$ a\text{cts} transitively on $V_{m,k}$, for any  $\vec{l}\in V_{m,k}$ there exist $l \in SO_{q_k}$ such that
$l \vec{x}_0=\vec{l}.$
In fact if $l \vec{x}_0=\vec{l}$ then $l k \vec{x}_0=\vec{l}$ for any $k\in K$.  We write every element $g \in SO_{q_k}$ as $g=l k t$ for some $t\in  K \backslash SO_{q_k}$ and $k\in K$.
Since $d\mu $ is a Haar measure then $d\mu(lg)=d\mu(g)$. Note that $K$ is a compact group, so we normalize the Haar measure so that $\int_{K} dk=1$.  We have 
\begin{equation}
\begin{split}
\rho_{F,\infty}(m,v)&=\sum_{\vec{l}\in H(m)}  \frac{v^{3/4}}{|\Gamma_{k,\vec{l}}|} \int_{SO_{q_k}} e(ivP_0[g^{-1}\vec{l}])\bar{L}(g) d\mu(g)
\\
&=\sum_{\vec{l}\in H(m)}  \frac{v^{3/4}}{|\Gamma_{k,\vec{l}}|} \int_{K \backslash SO_{q_k}} \int_{K} e(ivP_0[(l k t)^{-1}\vec{l}])\bar{L}( l kt ) dk dt 
\\
&=\sum_{\vec{l}\in H(m)}  \frac{v^{3/4}}{|\Gamma_{k,\vec{l}}|} \int_{K \backslash SO_{q_k}} \int_{K} e(ivP_0[t^{-1}k^{-1}l^{-1}\vec{l}])\bar{L}(l kt) dk dt 
\\
&=\sum_{\vec{l}\in H(m)}  \frac{v^{3/4}}{|\Gamma_{k,\vec{l}}|} \int_{K \backslash SO_{q_k}} \int_{K} e(ivP_0[t^{-1}\vec{x}_0])\bar{L}(l kt) dk dt
\\
&=\sum_{\vec{l}\in H(m)}  \frac{v^{3/4}}{|\Gamma_{k,\vec{l}}|} \int_{K \backslash SO_{q_k}} e(ivP_0[t^{-1}\vec{x}_0]) \int_{K} \bar{L}(l kt) dk dt. 
\end{split}
\end{equation}
Recall that $ L(l kt)= L(l kt\vec{x}_0)$. We take the integral over the compact group $K$ and obtain
\begin{equation}\label{MG0}\rho_{F,\infty}(m,v)= \sum_{\vec{l}\in H(m)}  \frac{v^{3/4}}{|\Gamma_{k,\vec{l}}|} \int_{K \backslash SO_{q_k}} e(ivP_0[t^{-1}\vec{x}_0]) V_{l}(t) dt, \end{equation}
where $V_{l}(t):=  \int_{K}  \bar{L}(l ktx_0) dk.$ By our normalization of the Haar measure of $K$ we obtain 
$$\sup_{t\in K \backslash SO_{q_k}} |V(t)| \leq \sup_{\vec{x}\in V_{m,k}} |L(\vec{x})|.$$
So $V_l$ is a bounded function on $K \backslash SO_{q_k}$.    We note that the quotient space $K \backslash SO_{q_k}$ is identified with 
$V_{m,k}$ by sending $t\in K \backslash SO_{q_k} $ to $\vec{h}:=t^{-1}\vec{x}_0 \in V_m$ and we write 
$$ \vec{h}:=\begin{bmatrix}h_1 \\ h_2 \\h_3    \end{bmatrix}.$$
The measure $dt$ is identified with the invariant measure defined over $V_{m,k}$ that is the hyperbolic measure. We denote this measure by $d\mu$. Next,  we change the variables and  write the integral \eqref{MG0} that is over the quotient space  $K \backslash SO_{q_k} $ in terms of an integral over $V_{m,k}$ and its hyperbolic measure.  We also consider the smooth weight function $V_l(t)$ as a function on $V_{m,k}$ by our identification $t\to t^{-1}\vec{x}_0 \in V_{m,k}$.  Hence, we obtain
$$\rho_{F,\infty}(m,v)= \sum_{\vec{l}\in H(m)}  \frac{v^{3/4}}{|\Gamma_{k,\vec{l}}|} \int_{V_{m,k}} e(ivP_0[h]) V_l(\vec{h}) d\mu. $$
Let 
$I(l,v):=v^{3/4} \int_{V_{m,k}} e(ivP_0[h]) V_l(\vec{h}) d\mu.$
Then 
\begin{equation}\label{Iv}
 \rho_{F,\infty}(m,v)= \sum_{\vec{l}\in H(m)}  \frac{1}{|\Gamma_{k,\vec{l}}|}I(l,v).
\end{equation}
Next, we give an asymptotic formula for $I(l,v)$ as $v\to \infty$. We note that 
$P_0[\vec{h}]=2kh_1^2+2kh_2^2+h_3^2.$
Then, 
$$ I(l,v)= v^{3/4}\int_{V_{m,k}} \exp(-2\pi v(2kh_1^2+2kh_2^2+h_3^2)) V_l(\vec{h}) d\mu.$$
Since $\vec{h}\in V_m,$  $h_3^2-4kh_1h_2=m$, and we obtain
\begin{equation*} 
I(l,v)=  \exp(-2\pi v|m|)v^{3/4}\int_{V_{m,k}} \exp(-2\pi v(2k(h_1-h_2)^2+2h_3^2)) V_l(\vec{h}) d\mu.
\end{equation*}
We change the variables to $u_1:=\frac{h_1\sqrt{2k}}{\sqrt{|m|}}$,  $u_2:=\frac{h_2\sqrt{2k}}{\sqrt{|m|}}$ and $u_3:=\frac{h_3}{\sqrt{|m|}}.$ Hence, we obtain
$$
I(l,v)= \exp(-2\pi v|m|)v^{3/4}\int_{u_3^2-u_1u_2=-1} \exp\big(-2\pi vm((u_1-u_2)^2+2u_3^2)\big) V_l(\vec{u})d\mu.
$$
We note that as $v\to \infty$ the above integral localizes around $\vec{u}_{0}=(1,1,0)$. By stationary phase theorem, it follows that 
\begin{equation*}
\lim_{v\to \infty} \int_{u_3^2-u_1u_2=-1} \exp\big(-2\pi vm((u_1-u_2)^2+2u_3^2)\big) V_l(\vec{u})d\mu=(1/2+O(\frac{1}{\sqrt{v}}))\frac{V_l(\vec{x}_0)}{v|m|}.
\end{equation*}
where  $\vec{x}_0=\begin{bmatrix}1/2 \sqrt{|m|/k} \\ 1/2\sqrt{|m|/k} \\ 0  \end{bmatrix}$
is the minimum of the quadratic form $2k(h_1-h_2)^2+2h_3^2$  on $V_{m,k}$.  Note that 
\begin{equation*}
\begin{split}
V_{l}(\vec{x}_0):=  \int_{K}  \bar{L}(l k\vec{x}_0) dk=\bar{L}(l\vec{x}_0)=\bar{L}(\vec{l}).
\end{split}
\end{equation*}
Therefore, 
$$I(l,v)=\exp(-2\pi v|m|) (4\pi|m| v)^{-1/4} \bar{L}(\vec{l})\frac{|m|^{-3/4}\pi^{1/4}}{\sqrt{2}}(1+O(1/\sqrt{v})).$$
We use the above identity in the  equation \ref{Iv} and obtain
\begin{equation}\label{Maass}
\begin{split}
\rho_{F,\infty}(m,v)&= \exp(-2\pi v|m|) (4\pi|m| v)^{-1/4} \frac{|m|^{-3/4}\pi^{1/4}}{\sqrt{2}} \sum_{\vec{l}\in H(m)}  \frac{1}{|\Gamma_{k,\vec{l}}|}I(l,v)\bar{L}(\vec{l})
\\
&=\exp(-2\pi v|m|) (4\pi|m| v)^{-1/4} \frac{|m|^{-3/4}\pi^{1/4}}{\sqrt{2}} \overline{R(m,L)}.
\end{split}
\end{equation}
By  \eqref{deffour}, we have
$$ \rho_{F,\infty}(m)= \frac{|m|^{-3/4}\pi^{1/4}}{\sqrt{2}} \overline{R(m,L)}.$$
This completes the proof of the Maass identity. 
\end{proof}

\subsection{Bounding $ \text{Er}_{\text{low}}$}\label{lowfrq}
Recall that 
$
 \text{Er}_{\text{low}}:=\sum_{\lambda<T}\langle W,f_{\lambda} \rangle R(m,f_{\lambda}),
$
where $T=|D|^{\delta}$ for some fixed power $\delta>0.$ In this section, we give an upper bound on $ \text{Er}_{\text{low}}$. Let $$B_{T}:=\{ \psi_{\lambda^{\prime}}\in L^2(\Gamma_{0}(4k)\backslash H): \Delta_{1/2}\psi_{\lambda^{\prime}}=\lambda^{\prime}  \psi_{\lambda^{\prime}} \text{ and } \lambda^{\prime}< T/4+3/16 \}$$ be an orthonormal basis of  weight 1/2 cusp  forms of level $4k$ and eigenvalue less than $T/4+3/16$. 
 It is known that $\psi_{\lambda^{\prime}}(z)$ has a Fourier development at $\infty$ of the form
$$\psi_{\lambda^{\prime}}(u+iv)=c_{\psi_{\lambda^{\prime}},\infty}(v)+\sum_{n\neq0}\rho_{\psi_{\lambda^{\prime}},\infty}(n)W_{1/4 \text{sgn}(n),it}(4\pi|n|v)e(nu), $$
where $1/4+t^2=\lambda^{\prime}$,  $c_{\psi_{\lambda^{\prime}},\infty}(v)$ is a linear combination of $v^{1/2+it}$ and $v^{1/2-it}$ and $W_{\beta,\mu}(v)$ is the Whittaker function normalized so that
$W_{\beta,\mu}(v)\approx e^{-v/2}v^{\beta}  \text{ as } v\to \infty.$
We note that the asymptotic of the Whittaker function is independent of the spectral parameter $\lambda.$ In the following lemma, we apply the Maass identity proved in Lemma~\ref{maasslem} and write $ \text{Er}_{\text{low}}$ in terms of the Fourier coefficients of $\psi_{\lambda^{\prime}}\in B_{T}$. 
\begin{lem}\label{Erlowl} We have 
\begin{equation}
 \overline{\text{Er}_{\text{low}}}= |m|^{3/4}\pi^{-1/4}\sqrt{2} \sum_{ \psi_{\lambda^{\prime}} \in B_{T}}\langle \Theta*W,\psi_{\lambda^{\prime}}  \rangle  \rho_{\psi_{\lambda^{\prime}},\infty}(m).
\end{equation}

\end{lem}
\begin{proof}
Let $W_{T}=\sum_{0\leq \lambda\leq T} \langle W,f_{\lambda} \rangle f_{\lambda},$
where $\{f_{\lambda} \}$ is an orthonormal basis of the cusp forms  with the $\Omega$ eigenvalue less than $T.$ Since $W_{T}$ is a finite linear combination of Maass cusp forms, it decays rapidly at cusps. By  Lemma~\ref{maasslem}, we have 
\begin{equation}\label{ERL}
 \overline{\text{Er}_{\text{low}}}=\overline{R(m,W_T)}=|m|^{3/4}\pi^{-1/4}\sqrt{2}\rho_{\Theta*W_T,\infty}(m).
\end{equation}
 By Theorem~\ref{dtrans}; see Remark~\ref{remtheta},  $\Theta*W_T$ is spanned by the orthonormal basis $B_T.$ Hence,
 $$\Theta*W_T=\sum_{ \psi_{\lambda^{\prime}} \in B_{T} }\langle \Theta*W,\psi_{\lambda^{\prime}}  \rangle  \psi_{\lambda^{\prime}}.$$
 By computing the $m$-th Fourier coefficient of  both sides of the above identity and using the asymptotic of the Whittaker function, we have 
$$
\rho_{\Theta*W_T,\infty}(m)=\sum_{ \psi_{\lambda^{\prime}} \in B_{T}}\langle \Theta*W,\psi_{\lambda^{\prime}}  \rangle  \rho_{\psi_{\lambda^{\prime}},\infty}(m).
$$
By the above and equation~\eqref{ERL}, it follows that 
$$
 \overline{\text{Er}_{\text{low}}}= |m|^{3/4}\pi^{-1/4}\sqrt{2} \sum_{ \psi_{\lambda^{\prime}} \in B_{T}}\langle \Theta*W,\psi_{\lambda^{\prime}}  \rangle  \rho_{\psi_{\lambda^{\prime}},\infty}(m).
$$
This completes the proof of the lemma.
\end{proof}
Finally, we bound the contribution of $ \text{Er}_{\text{low}}.$ 
\begin{prop}\label{Erlowlem}
We have 
$$ |\text{Er}_{\text{low}}| \ll |m|^{-1/28} k^{10} X^{1+\epsilon} T^{7}. $$
\end{prop}
\begin{proof}
By Lemma~\ref{Erlowl}, we have
\begin{equation}\label{Erloww}
 |\text{Er}_{\text{low}}|\leq  |m|^{3/4}\pi^{-1/4}\sqrt{2} \sum_{  \lambda^{\prime}< T/4+3/16}|\langle \Theta*W,\psi_{\lambda^{\prime}}  \rangle|  |\rho_{\psi_{\lambda^{\prime}},\infty}(m)|.
\end{equation}
Recall that $m=Dv_0^2$ where $D$ is a fundamental discriminant  and $v_0\leq |D|^{\epsilon}.$ By  Duke's upper bound \cite[Theorem 5]{Duke} on the  Fourier coefficients of the  weight half-integral weight forms, we have 
\begin{equation}\label{dukeu}
|\rho_{\psi_{\lambda^{\prime},\infty}}(m)|\ll_{\varepsilon} |\lambda|^{3/2} \cosh(\pi t/2)|m|^{-2/7+\varepsilon}.
\end{equation}
Next, we give an upper bound on $\langle \Theta*W,\psi_{\lambda^{\prime}}  \rangle$.  We have
\begin{equation}\label{innerthet}
\begin{split}
\langle \Theta*W,\psi_{\lambda^{\prime}}  \rangle&= \int_{\Gamma_{0}(4k)\backslash H} \overline{ \psi_{\lambda^{\prime}}(z)}\int_{\Gamma_k \backslash V_{m,k}} \Theta(z,\vec{h})w(\vec{h}) d\mu(\vec{h}) d\eta(z)
\\
&= \int_{\Gamma_k \backslash V_{m,k}}w(\vec{h}) \int_{\Gamma_{0}(4k) \backslash H}  \overline{ \psi_{\lambda^{\prime}}(z)} \Theta(z,\vec{h})d\eta(z)d\mu(\vec{h}).
\end{split}
\end{equation}
where $d\eta(z)$ and $d\mu(\vec{h})$ are invariant measures on $\Gamma_{0}(4k) \backslash H$ and $\Gamma_k \backslash V_{m,k}$, respectively.  Let 
\begin{equation}\label{thetalift}
\varphi_{\lambda}(h):=  \int_{\Gamma_{0}(4k) \backslash H} \Theta(z,\vec{h})   \overline{\psi_{\lambda^{\prime}}(z)}d\eta(z).
\end{equation}
It follows from Theorem~\ref{dtrans} that  $\varphi_{\lambda}$ is a Maass form of weight zero and eigenvalue $\lambda=4\lambda^{\prime}-3/4.$ We say $\varphi_{\lambda}$ is the theta lift of the weight $1/2$ modular form $\psi_{\lambda^{\prime}}$.
 By equation \eqref{innerthet}, we have
\begin{equation}\label{innerp}
\begin{split}
\langle  \Theta*W,\psi_{\lambda^{\prime}} \rangle&=  \int_{\Gamma_k \backslash V_{m,k}}  \varphi_{\lambda}(\vec{h})  W(\vec{h})d\mu(\vec{h})
=\langle \varphi_{\lambda},W \rangle.
\end{split}
\end{equation}
By the Cauchy-schwarz inequality 
$$|\langle  \Theta*W,\psi_{\lambda^{\prime}} \rangle| \leq |W|_2 |\varphi_{\lambda}|_2,$$ 
where $|W|_2$ and $|\varphi_{\lambda}|_2$ are the $L^2$ norm of $W$ and $\varphi_{\lambda}$. By Lemma~\ref{l2w}, we have 
$
|W|_2 \ll \frac{X^{1+\epsilon}}{\sqrt{d_1m}}.
$
By Theorem~\ref{Rallisthm}, we have 
$|\varphi_{\lambda}|_2\ll  \cosh(-\pi r/2) k^{9} \lambda^{9/2}.$
Therefore,
$$
|\langle  \Theta*W,\psi_{\lambda^{\prime}} \rangle| \leq \cosh(-\pi r/2) k^{9} \lambda^{9/2} \frac{X^{1+\epsilon}}{\sqrt{m}}.
$$
By applying the above and the  inequality \eqref{dukeu} in equation \eqref{Erloww}, we obtain 
\begin{equation}
| \text{Er}_{\text{low}} |\ll  |m|^{3/4} \Big(\sum_{  \lambda^{\prime}< T/4+3/16} |\lambda|^{3/2} \cosh(\pi t/2)|m|^{-2/7+\varepsilon}\cosh(-\pi r/2) k^{9} \lambda^{9/2} \frac{X^{1+\epsilon}}{\sqrt{m}}\Big).
\end{equation}
By the Weyl law the number of eigenvalues $\lambda^{\prime} \leq T $ is bounded by $kT.$  Therefore,
$$
| \text{Er}_{\text{low}}| \ll |m|^{-1/28} k^{10} X^{1+\epsilon} T^{7}.
$$ 
We choose $T=|D|^{\delta}$ for a small fixed $\delta>0.$
\end{proof}

\subsection{Bounding $ \text{Er}_{\text{cts,low}}$} \label{cts,low}
We briefly explain our method for bounding $ \text{Er}_{\text{cts,low}}$.  Recall that 
\[
 \text{Er}_{\text{cts,low}}:=\sum_{\mathfrak{a}\in \mathcal{E}_k} \int_{|1/4+t^2|\leq T} \langle W, E_{\mathfrak{a}}(.,1/2+it) \rangle R(m, E_{\mathfrak{a}}(\vec{v},1/2+it)) dt.
\]
 We wish to apply Lemma~\ref{maasslem} to $R(m, E_{\mathfrak{a}}(\vec{v},1/2+it))$ which relates $R(m, E_{\mathfrak{a}}(\vec{v},1/2+it))$ to the $m$-th Fourier coefficient of $\Theta*E_{\mathfrak{a}}(\vec{v},1/2+it).$ However, we note that the theta integral $\Theta*E_{\mathfrak{a}}(\vec{v},1/2+it)$ is not absolutely convergent and we need to regularize this integral.  We  use the center of the enveloping algebra (Casimir operator) for regularizing this theta integral. This method has been used in the work of Maass~\cite{Hans}, Deitmar and Krieg \cite{Deitmar} and Kudla and Rallis~\cite[Section 5]{Kudla}. We begin by proving an auxiliary lemma. 
 \begin{lem}\label{heightlem}
 We have 
 \[
\frac{1}{k^{2.5}}\ll y_{\mathfrak{a}}(\vec{x}_0) \ll 1
 \]
 for every $\mathfrak{a}\in\mathcal{E}_k.$
 \end{lem} 
 \begin{proof}
 Let $\vec{v}_0:=\begin{bmatrix}1 \\1 \\0    \end{bmatrix} \in V_{-4k,k}.$ By scaling $\Gamma_k \backslash V_{-4k,k}$ maps isometrically to $\Gamma_k\backslash V_{m,k}.$ This maps $\vec{v}_0$  to $\vec{x}_0.$ Hence 
$y_{\Gamma_k}(\vec{x}_0)= y_{\Gamma_k}(\vec{v}_0).$ It follows that  
\[
y_{\Gamma_k}(\vec{v}_0) \ll \max_{n_{\mathfrak{a}}}\left( 1,\frac{1}{\text{dist}(\vec{v}_0, n_{\mathfrak{a}}\vec{v}_0)}\right),
\]
 where $n_{\mathfrak{a}}$ is a parabolic element of $\Gamma_k.$ Since, $n_{\mathfrak{a}}$ is parabolic, $\vec{v}_0\neq n_{\mathfrak{a}}\vec{v}_0.$ Let $n_{\mathfrak{a}}\vec{v}_0=\begin{bmatrix}n_1\\n_2\\n_3\end{bmatrix}$ for some $n_1,n_2,n_3\in \mathbb{Z},$ where $n_3^2-4kn_1n_2=-4k.$ By integrality of $n_i$ and $\vec{v}_0\neq n_{\mathfrak{a}}\vec{v}_0,$ we have  $|n_1n_2|\geq 2.$ Hence, $|n_3|\geq 2\sqrt{k}.$ We define the following isometry from $V_{-4k,k}$ to the upper half-plane $H$
 \begin{equation}\label{eqisom}
 \begin{bmatrix} a_1\\a_2\\a_3\end{bmatrix}\in V_{-4k,k} \to \begin{bmatrix} ka_1/2\sqrt{k} \\ a_2/2\sqrt{k} \\a_3/2\sqrt{k}  \end{bmatrix}\in V_{-1,1} \to z_a=\frac{-a_3+i2\sqrt{k}}{2ka_1} \in H.
 \end{equation}
 We note that $\vec{v}_0$ maps to $\frac{i}{\sqrt{k}}$ and $n_{\mathfrak{a}}\vec{v}_0$ maps to $\frac{-n_3}{{2kn_1}}+\frac{i}{n_1\sqrt{k}}.$ Hence,
 \[
 \text{dist}(\vec{v}_0, n_{\mathfrak{a}}\vec{v}_0)=\text{dist}(i, \frac{-n_3}{{2{\sqrt{k}}n_1}}+\frac{i}{n_1})\gg 1.
 \]
 This completes the proof of our upper bound.  For proving the lower bound, we identify $SO_{q_k}$  with $PSL_2(\mathbb{R}),$ so that  $\Gamma_k$  is identified with $\Gamma^{\prime}\subset SO_{q_k},$ where  $\Gamma^{\prime}$ contains the congruence subgroup $\Gamma_0(k).$ Then we parametrize the cusps of $\Gamma_0(k)$ with $1/w$ for $1\leq w\leq k$ and show $\frac{1}{k^{2.5}} \leq y_{1/w}(i/\sqrt{k}).$  Since $ \Gamma^{\prime} \backslash H$ is a covering of  $ \Gamma_{0}(k) \backslash H$ and $\vec{x}_0$ maps to $i/\sqrt{k}$, we have   \( \frac{1}{k^{2.5}} \leq y_{1/w}(i/\sqrt{k}) \leq y_{\mathfrak{a}}(\vec{x}_0).\)

 We give the details of our argument. $PSL_2(\mathbb{R})$ a\text{cts} on the space of binary quadratic forms $Q:=\{ax^2+bxy+cy^2: a,b,c \in \mathbb{R} \}$ by linear change of variables and it preserves the discriminant of the  binary quadratic forms
 $$
 \begin{bmatrix}a & b\\c&d  \end{bmatrix}.F(x,y)= F(ax+by,cx+dy).
 $$
 This identifies $PSL_2(\mathbb{R})$ with $SO_{q_1}$ where $q_1(x,y,z)=z^2-4xy$ through the map
 \begin{equation}\label{map}
\gamma=\begin{bmatrix} a & b \\ c & d  \end{bmatrix} \to g_{\gamma}=\begin{bmatrix}a^2& c^2 &ac 
\\b^2 & d^2 & bd \\ 2ab & 2cd& ad+bc \end{bmatrix}.
\end{equation}
 Let 
$B_k:=\begin{bmatrix}
k& 0 & 0
\\ 
0&1&0
\\
0&0&1
\end{bmatrix},
$
then
$
B_k^{\intercal} \begin{bmatrix}
0& -2 & 0
\\ 
-2&0&0
\\
0&0&1
\end{bmatrix}B_k =\begin{bmatrix}
0& -2k & 0
\\ 
-2k&0&0
\\
0&0&1
\end{bmatrix}.
$
We note that if $g\in SO_{q_1}$ then  $B_k^{-1}gB_k \in SO_{q_k}$. This identifies $PSL_2(\mathbb{R})$ with $SO_{q_k}$
\begin{equation}\label{isom}
\gamma \in PSL_2(\mathbb{R}) \to g_{\gamma}\in SO_{q_1} \to B_k^{-1} g_{\gamma}B_k\in SO_{q_k}.
\end{equation}
By the above isomorphism the lattice $\Gamma_k \subset SO_{q_k}$ is identified with $\Gamma^{\prime} \subset PSL_2(\mathbb{R})$, where
\begin{equation*}\Gamma^{\prime}:=\left\{\begin{bmatrix} a & b \\c &d   \end{bmatrix} \in PSL_2(\mathbb{R}):\begin{bmatrix}
a^2& k^{-1}c^2 & k^{-1}ac
\\ 
kb^2&d^2&bd
\\
2kab&2cd&ad+bc
\end{bmatrix}\in M_{3\times 3}(\mathbb{Z})  \right\}.
\end{equation*}
It is easy to check that $\Gamma^{\prime}$ contains  the congruence subgroup  $\Gamma_{0}(k):= \left\{\begin{bmatrix} a & b \\ c & d  \end{bmatrix} : a,b,c,d \in \mathbb{Z} \text{ and } k|c  \right\}.$ By Proposition~\ref{Ycusp} and~\ref{Ycusp1}, the cusps of $\Gamma_0(k)$ are parametrized (not uniquely) with $1/w$  for $1\leq w \leq k$ and its scaling matrix is
\(
  \sigma_{1/w}=\begin{bmatrix} 1 &0 \\ w &1 \end{bmatrix} \begin{bmatrix} \sqrt{k^{\prime\prime}} & 0 \\ 0& 1/\sqrt{k^{\prime\prime}}   \end{bmatrix},
\)
where $k^{\prime\prime}=\frac{k^{\prime}}{\gcd(k^{\prime},w)}$ and $k^{\prime}=\frac{k}{\gcd(k,w)}.$ Since $\vec{x}_0$ maps to $\frac{i}{\sqrt{k}}$,  we have 
\[
y_{1/w}(\frac{i}{\sqrt{k}})=\Im ( \sigma_{1/w}^{-1} (\frac{i}{\sqrt{k}}))=\frac{1}{k^{\prime\prime}}\frac{\sqrt{k}}{k+w^2} \geq \frac{1}{k^{2.5}}.
\]
This completes the proof of our theorem. 
 \end{proof}

Let $\theta(z,g):=\Omega \Theta(z,g).$ Since $\Omega$ is inside the center of enveloping algebra,   $\theta(z,g)$ remains a theta kernel.  The following theorem follows from the work of Kudla and Rallis~\cite[Proposition 5.3.1]{Kudla}; see also~\cite[Lemma 7.7]{weetech}.
\begin{prop}
$\theta(z,g)= O_{C}(|y_{\Gamma_0(4k)}(z)|^{-A})$ for every $A>0$ and  $g\in C,$ where $C\subset \Gamma_{k}\backslash SO_{q_k}$ is a fixed compact subset. Similarly  $\theta(z,g)=O_{C^{\prime}}(|y_{\Gamma_k}(g)|^{-A})$  for every $A>0$ and $g\in K,$ where $C^{\prime}\subset  \Gamma_0(4k) \backslash H$ is a fixed compact subset. Moreover, for every $z$
$$
\int_{\Gamma_{k}\backslash SO_{q_k}} \theta(z,g) d\mu(g)=0.
$$
\end{prop}
Let $\mathcal{E}^{\prime}_{4k}:=\{\mathfrak{a}^{\prime}: \mathfrak{a}^{\prime} \text{ ranges over all inequivalent cusps of }  \Gamma_0(4k) \backslash H\}$. For $z\in  \Gamma_0(4k) \backslash H$ and $s\in \mathbb{C}$, let $E_{\mathfrak{a}^{\prime}}(z,s)$ be the Eisenstein series of weight $1/2$ such that its constant Fourier coefficient at cusp $\mathfrak{b}^{\prime}$ is $\delta_{\mathfrak{a}^{\prime}\mathfrak{b}^{\prime}} y_{\mathfrak{b}^{\prime}}^s+ \varphi_{\mathfrak{a}^{\prime}\mathfrak{b}^{\prime}}(s)y_{\mathfrak{b}^{\prime}}^{1-s};$ see~\cite[Section 2]{Duke}. 

 \begin{prop}\label{specm}
 We have 
 \[
  \overline{\text{Er}_{\text{cts,low}}}= \frac{\sqrt{2}}{\pi^{1/4}} |m|^{3/4}\sum_{\mathfrak{a}^{\prime}\in \mathcal{E}^{\prime}_{4k}} \int_{|1/4+t^2|\leq T} \frac{1}{1/4+t^2} \rho_{ E_{\mathfrak{a}^{\prime}} (.,1/2+it/2),\infty}(m)\langle \theta*W,E_{\mathfrak{a}^{\prime} }(.,1/2+it/2) \rangle dt.
 \]
  \end{prop}
 \begin{proof}
 Since $W$ is compactly supported and $\theta$ is rapidly decreasing uniformly on compact sets, $\theta*W(z)$ is also rapidly decreasing on $\Gamma_0(4k) \backslash H.$   
 Since the Eisenstein series $E_{\mathfrak{a}}(.,1/2+it)$ has moderate growth on $\Gamma_{k}\backslash SO_{q_k}$ and  $\theta$ is rapidly decreasing on $\Gamma_{k}\backslash SO_{q_k},$
  $$\theta(z,.)*E_{\mathfrak{a}}(.,1/2+it):= \int_{\Gamma_K \backslash SO_{q_k}}\theta(z,g) \overline{E_{\mathfrak{a}}(g\vec{x}_0,1/2+it)} d\mu(g)$$ is absolutely convergent. By Siegel-Weil formula, we have 
\[
\theta(z,.)*E_{\mathfrak{a}}(.,1/2+it)= \sum_{\mathfrak{a}^{\prime}\in \mathcal{E}^{\prime}_{4k}} \alpha_{\mathfrak{a} \mathfrak{a}^{\prime}}(t) E_{\mathfrak{a}^{\prime}}(z,1/2+it/2),
\]
where $ \alpha_{\mathfrak{a} \mathfrak{a}^{\prime}}(t)$ is an analytic function for every pair $\mathfrak{a} \mathfrak{a}^{\prime}$ . Since  $\theta$ transfers the cusp forms on $\Gamma_{k}\backslash SO_{q_k}$ to the weight $1/2$ cusp forms on $ \Gamma_0(4k)\backslash H$ and  $\theta*W(z)$ is rapidly decreasing,  by the Plancherel theorem  
\[
\sum_{\mathfrak{a}\in \mathcal{E}_k } \overline{\langle W, E_{\mathfrak{a}}(.,1/2+it)\rangle}  \theta(z,.)*E_{\mathfrak{a}} (.,1/2+it) =  \sum_{\mathfrak{a}^{\prime}\in \mathcal{E}^{\prime}_{4k}}  \langle \theta*W,  E_{\mathfrak{a}^{\prime}}(.,1/2+it/2)\rangle  E_{\mathfrak{a}^{\prime}}(z,1/2+it/2)
\]
for every $t\in \mathbb{R}.$
Hence, the $m$-th Fourier coefficient of both sides are equal, and we obtain
\[
\sum_{\mathfrak{a}\in \mathcal{E}_k } \overline{\langle W,E_{\mathfrak{a}}(.,1/2+it) \rangle}  \rho_{\theta*E_{\mathfrak{a}} (.,1/2+it)} (m)= \sum_{\mathfrak{a}^{\prime}\in \mathcal{E}^{\prime}_{4k}}  \langle \theta*W,  E_{\mathfrak{a}^{\prime}}(.,1/2+it/2)\rangle   \rho_{ E_{\mathfrak{a}^{\prime}} (.,1/2+it/2),\infty}(m)
\]
for every $t\in \mathbb{R}.$
By a similar computation as  in Lemma~\ref{maasslem} and the identity  $\Omega E_{\mathfrak{a}}(\vec{v},s)=s(1-s)E_{\mathfrak{a}}(\vec{v},s)$, it follows that 
\[
   \frac{\sqrt{2}}{\pi^{1/4}} |m|^{3/4} \rho_{\theta*E_{\mathfrak{a}}(.,s),\infty}(m)= \overline{s(s-1)R(m, E_{\mathfrak{a}}(.,s))}. 
\]
Finally, we have
 \begin{equation*}
 \begin{split}
 \overline{ \text{Er}_{\text{cts,low}}}&= \sum_{\mathfrak{a}\in \mathcal{E}_k} \int_{|1/4+t^2|\leq T} \overline{\langle W, E_{\mathfrak{a}}(.,1/2+it) \rangle R(m, E_{\mathfrak{a}}(.,1/2+it))} dt
 \\
 &=  \frac{\sqrt{2}}{\pi^{1/4}} |m|^{3/4} \sum_{\mathfrak{a}\in \mathcal{E}_k} \int_{|1/4+t^2|\leq T}  \frac{1}{1/4+t^2} {\langle E_{\mathfrak{a}}(.,1/2+it), W \rangle   \rho_{\theta*E_{\mathfrak{a}}(.,1/2+it),\infty}(m)} dt
 \\
 &= \frac{\sqrt{2}}{\pi^{1/4}} |m|^{3/4}\sum_{\mathfrak{a}^{\prime}\in \mathcal{E}^{\prime}_{4k}} \int_{|1/4+t^2|\leq T} \frac{1}{1/4+t^2}{ \rho_{ E_{\mathfrak{a}^{\prime}} (.,1/2+it/2),\infty}(m)\langle \theta* W,E_{\mathfrak{a}^{\prime} }(.,1/2+it/2) \rangle} dt.
 \end{split}
 \end{equation*}
 \end{proof}
\begin{prop}\label{thetal2cts}
We have 
\[
\left|\langle \theta* W,E_{\mathfrak{a}^{\prime} }(.,1/2+it/2) \rangle\right| \ll k^{5}t^{2.25}  |\Gamma(1/4+it)|  \sum_{\mathfrak{a}\in\mathcal{E}_k}\left|\langle W,E_{\mathfrak{a} }(.,1/2+it) \rangle\right|.
\]
\end{prop}
\begin{proof} For simplicity, we give the proof for $\mathfrak{a}^{\prime}=\infty$. We briefly explain how to prove this proposition for every cusp $\mathfrak{a}^{\prime}\in \mathcal{E}^{\prime}_{4k}.$  By Proposition~\ref{Ycusp}, $\mathcal{E}^{\prime}_{4k}$ is parametrized (not uniquely) with $1/w$ for $1 \leq w \leq 4k.$   We use the scaling matrices $\sigma_{1/w}$ that are introduced in Proposition~\ref{Ycusp1} and the transformation properties of the theta series; see~\cite[equation 4.4]{Duke}, to reduce the problem to the cases  $\mathfrak{a}^{\prime}=\infty.$


We write $E(.,1/2+it/2)$ for $E_{\infty}(.,1/2+it/2).$ Suppose that $\Re(s)>1,$ then by unfolding $\int_{ \Gamma_0(4k)\backslash H}$ against the Eisenstein series and unfolding $\int_{\Gamma_{k}\backslash SO_{q_k}}$ against $W(g)$, we have 
\begin{equation*}
\begin{split}
\langle \theta* W,E(.,s) \rangle&= \int_{\Gamma_{k}\backslash SO_{q_k}} \int_{ \Gamma_0(4k)\backslash H} \theta(z,g)\bar{E}(z,s)W(g) \frac{dxdy}{y^2}d\mu(g)
\\
&=\int_{SO_{q_k}}\int_{0}^1 \int_{0}^{\infty} \theta(z,g)\bar{y^s}w(g)  \frac{dy}{y^2}dx d\mu(g)
\\
&=\int_{SO_{q_k}}\int_{0}^1 \int_{1}^{\infty} \left( y^{3/4}\sum_{\vec{h}\neq0   \in \mathbb{Z}^3} e\big(xq_k(\vec{h})\big)e\big(iy\vec{h}^{\intercal}g^{-\intercal}P_0g^{-1}\vec{h} \big)\right)\bar{y^s}\Omega w(g)  \frac{dy}{y^2}dx d\mu(g)
\\
&=\int_{SO_{q_k}}\int_{0}^{\infty}\left( y^{3/4}\sum_{\substack {\vec{h}\neq 0  \in \mathbb{Z}^3\\ 
q_k(\vec{h})=0 } } e\big(iy\vec{h}^{\intercal}g^{-\intercal}P_0g^{-1}\vec{h} \big)\right)\bar{y^s}\Omega w(g)  \frac{dy}{y^2}d\mu(g)
\\
&= \overline{\Gamma(s-1/4)} \int_{SO_{q_k}} \overline{ \sum_{\substack {\vec{h}\neq 0  \in \mathbb{Z}^3\\ q_k(\vec{h})=0 } }  \frac{1}{(\vec{h}^{\intercal}g^{-\intercal}P_0g^{-1}\vec{h})^{s-1/4}} }\Omega w(g)  d\mu(g).
\end{split}
\end{equation*}
%
Let 
\[
 \psi(g,s):=\sum_{\substack {\vec{h}\neq 0  \in \mathbb{Z}^3\\ q_k(\vec{h})=0 } }  \frac{1}{(\vec{h}^{\intercal}g^{-\intercal}P_0g^{-1}\vec{h})^{s-1/4}}.
 \]
It follows that $ \psi(g,s)$ is absolutely convergent for $\Re(s)>3/4.$  Note that $\Gamma_k$ acts on the set of projective  vectors $[\vec{h}]  \in P^2(\mathbb{Q}),$ where    $q_k(\vec{h})=0,$ and the inequivalent classes are in one to one correspondence with $\mathcal{E}_k.$ For every $\mathfrak{a}\in\mathcal{E}_k,$ pick a primitive representative $\vec{h}_{\mathfrak{a}}\neq 0\in \mathbb{Z}^3.$ We have 
\[
\psi(g,s)=2\zeta(2s-1/2) \sum_{\mathfrak{a}\in \mathcal{E}_k} \sum_{\gamma\in \Gamma_{k,\vec{h}_\mathfrak{a}} \backslash \Gamma_k } \frac{1}{(\vec{h}_{\mathfrak{a}}^{\intercal}(\gamma g)^{-\intercal}P_0(\gamma g)^{-1}\vec{h}_{\mathfrak{a}})^{s-1/4}},
\]
where $\Gamma_{k,\vec{h}_\mathfrak{a}} \subset \Gamma_{k}$ is the unipotent subgroup which fixes $\vec{h}_\mathfrak{a}.$
By Iwasawa decomposition, we write $\gamma g=n_{\mathfrak{a}}(\gamma g)a_{\mathfrak{a}}(\gamma g)k(\gamma g),$ where $k(\gamma g)\in K,$ $n_{\mathfrak{a}}(\gamma g)\in N_{\mathfrak{a}},$ and $a_{\mathfrak{a}}(\gamma g)$ is a symmetric matrix where $a_{\mathfrak{a}}(\gamma g)\vec{h}_{\mathfrak{a}}=t \vec{h}_{\mathfrak{a}}$ for some $t_{\mathfrak{a}}(\gamma g)\in \mathbb{R}^{+}.$ Hence, we have 
\begin{equation*}
\begin{split}
\frac{\psi(g,s)}{2\zeta(2s-1/2)}&= \sum_{\mathfrak{a}\in \mathcal{E}_k} \sum_{\gamma\in \Gamma_{k,\vec{h}_\mathfrak{a}} \backslash \Gamma_k } \frac{1}{(\vec{h}_{\mathfrak{a}}^{\intercal}(\gamma g)^{-\intercal}P_0(\gamma g)^{-1}\vec{h}_{\mathfrak{a}})^{s-1/4}}
\\
&= \sum_{\mathfrak{a}\in \mathcal{E}_k} \frac{1}{(\vec{h}_{\mathfrak{a}}^{\intercal}P_0\vec{h}_{\mathfrak{a}})^{s-1/4}}
 \sum_{\gamma\in \Gamma_{k,\vec{h}_\mathfrak{a}} \backslash \Gamma_k }t_{\mathfrak{a}}(\gamma g)^{2s-1/2}
\\
&= \sum_{\mathfrak{a}\in \mathcal{E}_k} \frac{1}{(\vec{h}_{\mathfrak{a}}^{\intercal}P_0\vec{h}_{\mathfrak{a}})^{s-1/4} y_{\mathfrak{a}}(\vec{x}_0)^{2s-1/2}}
 \sum_{\gamma\in \Gamma_{k,\vec{h}_\mathfrak{a}} \backslash \Gamma_k  }y_{\mathfrak{a}}(\gamma g\vec{x}_0)^{2s-1/2}
\\
&=\sum_{\mathfrak{a}\in \mathcal{E}_k}\frac{1}{(\vec{h}_{\mathfrak{a}}^{\intercal}P_0\vec{h}_{\mathfrak{a}})^{s-1/4} y_{\mathfrak{a}}(\vec{x}_0)^{2s-1/2}}  E_{\mathfrak{a}}(g,
2s-1/2).
\end{split}
\end{equation*}
Note that both sides of the above identity have analytic continuation to the whole complex plane. Hence 
\[
\psi(g,1/2+it)=  \sum_{\mathfrak{a}\in \mathcal{E}_k}\frac{2\zeta(1/2+2it)}{(\vec{h}_{\mathfrak{a}}^{\intercal}P_0\vec{h}_{\mathfrak{a}})^{1/4+it} y_{\mathfrak{a}}(\vec{x}_0)^{1/2+2it}}  E_{\mathfrak{a}}(g,
1/2+2it).
\]
By Lemma~\ref{heightlem}, integrality of $\vec{h}_{\mathfrak{a}}^{\intercal}P_0\vec{h}_{\mathfrak{a}}$ and convexity bound on the zeta function, for every $\mathfrak{a}\in  \mathcal{E}_k,$ we have 
\[
\left|\frac{2\zeta(1/2+2it)}{(\vec{h}_{\mathfrak{a}}^{\intercal}P_0\vec{h}_{\mathfrak{a}})^{1/4+it} y_{\mathfrak{a}}(\vec{x}_0)^{1/2+2it}}  \right|\ll t^{1/4}k^{1.25}.
\]
Therefore,
\[
|\langle \theta* W,E(.,s) \rangle| \ll  |\Gamma(1/4+it)| t^{1/4}k^{1.25} (1/4+4t^2)\sum_{\mathfrak{a}\in \mathcal{E}_k} |\langle W,E_{\mathfrak{a}}(.,1/2+it) \rangle|.
\]
This completes the proof of our Proposition. \end{proof}
Finally, we give an upper bound on the contribution of $\text{Er}_{\text{cts,low}}.$
\begin{prop}\label{Erctslow}
We have 
\(
|\text{Er}_{\text{cts,low}}|\ll  k^{6.5} T^{7/4}|m|^{-1/28+\epsilon}|X|^{1+\epsilon}.
\)
\end{prop}
\begin{proof}
By Proposition~\ref{specm}, we have 
\[
|\text{Er}_{\text{cts,low}}|\ll |m|^{3/4}\sum_{\mathfrak{a}^{\prime}\in \mathcal{E}^{\prime}_{4k}} \int_{|1/4+t^2|\leq T} \frac{1}{1/4+t^2} |\rho_{ E_{\mathfrak{a}^{\prime}} (.,1/2+it/2),\infty}(m)| |\langle \theta*W,E_{\mathfrak{a}^{\prime} }(.,1/2+it/2) \rangle| dt.
\]
By \cite[Theorem 5]{Duke}, we have 
\[
\left| \rho_{ E_{\mathfrak{a}^{\prime}} (.,1/2+it/2),\infty}(m)\right| \ll (1+|t|)^{3} \cosh (\pi t/2) |m|^{-2/7+\epsilon}.
\]
By Proposition~\ref{thetal2cts}, we have 
\[
\left|\langle \theta*W,E_{\mathfrak{a}^{\prime} }(.,1/2+it/2) \rangle\right| \ll k^{5}(1+|t|)^{2.25}  |\Gamma(1/4+it)|  \sum_{\mathfrak{a}\in\mathcal{E}_k}\left|\langle W,E_{\mathfrak{a} }(.,1/2+it) \rangle\right|.
\]
Proposition~\ref{Ycusp} implies that  $|\mathcal{E}^{\prime}_{4k}| \leq 4k.$ Hence, 
\begin{equation*}
\begin{split}
|\text{Er}_{\text{cts,low}}|&\ll |m|^{3/4}\sum_{\mathfrak{a}^{\prime}\in \mathcal{E}^{\prime}_{4k}} \int_{|1/4+t^2|\leq T} \frac{1}{1/4+t^2} |\rho_{ E_{\mathfrak{a}^{\prime}} (.,1/2+it/2),\infty}(m)|  |\langle \theta*W,E_{\mathfrak{a}^{\prime} }(.,1/2+it/2) \rangle| dt
\\
&\ll k^{6} |m|^{13/28+\epsilon} \sum_{\mathfrak{a}\in\mathcal{E}_k} \int_{|1/4+t^2|\leq T}(1+ |t|)^{3.25} \cosh (\pi t/2)   |\Gamma(1/4+it)| \left|\langle W,E_{\mathfrak{a} }(.,1/2+it) \rangle\right| dt. 
\end{split}
\end{equation*}
By Stirling's formula for the Gamma function, we have 
\[|\Gamma\left(1/4+it\right)|\sim\sqrt{2\pi}|t|^{-1/4}e^{-\pi|t|/2}.\]
Hence,
\[
|\text{Er}_{\text{cts,low}}|\ll k^{6} |m|^{13/28+\epsilon} \sum_{\mathfrak{a}\in\mathcal{E}_k} \int_{|1/4+t^2|\leq T} (1+|t|)^{3} \left|\langle W,E_{\mathfrak{a} }(.,1/2+it) \rangle\right| dt. 
\]
By Plancherel's theorem and  Lemma~\ref{l2w}, we have 
\[
\sum_{\mathfrak{a}\in\mathcal{E}_k}  \int_{|1/4+t^2|\leq T}|\langle W,E_{\mathfrak{a} }(.,1/2+it) \rangle|^2 \leq \int_{\Gamma_k \backslash V_{m,k}} | W|^2 d\mu \ll\frac{X^{2+\epsilon}}{d_1m}.
\]
Since $\Gamma_0(k)$ is a subgroup of $\Gamma_k$, $|\mathcal{E}_k| \leq k.$  By the Cauchy-Schwarz inequality, we have 
\begin{equation*}
\begin{split}
&\sum_{\mathfrak{a}\in\mathcal{E}_k} \int_{|1/4+t^2|\leq T} (1+|t|)^{3} \left|\langle W,E_{\mathfrak{a} }(.,1/2+it) \rangle\right| dt 
\\
& \left[\sum_{\mathfrak{a}\in\mathcal{E}_k} \int_{|1/4+t^2|\leq T} (1+|t|)^{6} dt \right]^{1/2} \left[\sum_{\mathfrak{a}\in\mathcal{E}_k} \int_{|1/4+t^2|\leq T} \left|\langle W,E_{\mathfrak{a} }(.,1/2+it) \rangle\right|^2 dt \right]^{1/2}
\\
&\ll  k^{1/2} T^{7/4} \frac{X^{1+\epsilon}}{\sqrt{d_1m}}.
\end{split}
\end{equation*}
Therefore,
\(
\text{Er}_{\text{cts,low}}\ll k^{6} |m|^{13/28+\epsilon} k^{1/2} T^{7/4} \frac{X^{1+\epsilon}}{\sqrt{d_1m}}\ll k^{6.5} T^{7/4}|m|^{-1/28+\epsilon}|X|^{1+\epsilon}.
\)
\end{proof}
\section{Bounding the $L^2$ norm of the  Siegel theta transfer}\label{thetatrans}
%

\subsection{The Mellin transform of the theta transfer}
We follow the same notations as in the previous sections. 
%
%
Let $f(z)$ be a weight 1/2 modular form on $\Gamma_0(4k)\backslash H$ with $L^2$ norm 1 and  eigenvalue $\lambda^{\prime}$. 
Recall  the Fourier expansion  of $f(z)$ at $\infty$  
$$f(z)=c_{f,\infty}(y)+\sum_{n\neq0}b_{f,\infty}(n)W_{1/4 \text{sgn}(n),it}(4\pi|n|y)e(nx), $$
where $1/4+t^2=\lambda^{\prime}$,  $c_{f,\infty}(y)$ is a linear combination of $y^{1/2+it}$ and $y^{1/2-it}$ and $W_{\beta,\mu}(y)$ is the Whittaker function normalized so that
$W_{\beta,\mu}(y)\approx e^{-y/2}y^{\beta}  \text{ as } y\to \infty.$
  For $g=\begin{bmatrix}a & b \\ c & d  \end{bmatrix}\in SL_2(\mathbb{R})$, we define
\begin{equation}
f_{g}(z):= \Big(\frac{cz+d}{|cz+d|}\Big)^{-1/2}f(gz),
\end{equation}
where for $z\neq 0$ and $\nu\in \mathbb{R}$ we define $z^{\nu}=|z|\exp (iv\arg(z)),$ where $\arg z \in (-\pi,\pi].$ 
Since $f$ is an eigenfunction of $\Delta_{1/2}$ with eigenvalue $\lambda^{\prime}$ and  invariant under $\Gamma_{0}(4k)$ with a multiplier of weight $1/2$ then $f_{g}$ is an eigenfunction of $\Delta_{1/2}$ with eigenvalue $\lambda^{\prime}$ and is invariant under $g^{-1}\Gamma_0(4k) g$ with a multiplier of weight $1/2$. 
Let  
\begin{equation}\label{phi}\varphi(g):=\int_{\Gamma_{0}(4k) \backslash H}  \Theta(x+iy,g)\overline{f(x+iy)}  \frac{dx dy}{y^{2}}.\end{equation}
Recall that $\Theta(z,g)$ is $\Gamma_k$ invariant from the left and $K$ invariant from the right in $g$ variable.
It follows from Theorem~\ref{dtrans} that  $\varphi$ is a Maass form of weight zero and eigenvalue $\lambda=4\lambda^{\prime}-3/4$ on $\Gamma_k\backslash V_{m,k}$. We consider the following torus  $\mathbb{G}_m$ inside $SO_{q_k}$
$$t\in \mathbb{G}_m \to g_t:=\begin{bmatrix} t & 0 &0 \\ 0& t^{-1} & 0 \\ 0 & 0& 1   \end{bmatrix}\in SO_{q_k}.$$
In the following lemma, we compute the Mellin-transfom of $\varphi$ along the above embedding of $\mathbb{G}_m$.  Let
\begin{equation}\label{Mel}
\Omega(s):= \int_{0}^{\infty} \varphi(g_t)t^{s} \frac{dt}{t},
\end{equation}

$$\theta(z):=y^{1/4}\sum_{h\in \mathbb{Z}} e\big((x+iy)h^2\big),$$
 and 
    \begin{equation}\label{element}
E(s,z):=\sum_{h_1,h_2}^{\quad \prime} \Big(\frac{y}{|h_1+4h_2zD|^2}\Big)^{s}, \end{equation}
where $\sum_{h_1,h_2}^{\quad \prime}$ is the sum over pairs of co-prime integers. 

\begin{lem}\label{seesaw}
We have
  \begin{equation}
 \Omega(s)= k^{s/2}2^{s-1}\Gamma(\frac{s+1}{2}) \pi^{-\frac{s+1}{2}}  \int_{\Gamma_{0}(4k) \backslash H}  \overline{f(z)} \theta(z) E(\frac{s+1}{2},z) \frac{dx dy}{y^{2}}.
 \end{equation}
 \end{lem}
\begin{proof}
We use the integral representation of $\varphi$ in equation \eqref{phi} and obtain:
\begin{equation}\label{mellin}
\begin{split}
\Omega(s)&=\int_{0}^{\infty} \Big(\int_{\Gamma_{0}(N) \backslash H} \overline{ f(x+iy)}\Theta(x+iy,g_t)\frac{dx dy}{y^{2}}\Big)t^{s}dt/t
\\
&=\int_{\Gamma_{0}(4k) \backslash H}  \overline{f(x+iy)} \Big( \int_{0}^{\infty}  \Theta(x+iy,g_t) t^{s}dt/t \Big) \frac{dx dy}{y^{2}}.
\end{split}
\end{equation}
Next, we split $\Theta(z,g_t)$ into product of two theta series.  By definition \eqref{thetaa}, we have
\begin{equation}\label{splittheta}
\begin{split}
\Theta(x+iy,g_t):&=y^{3/4}\sum_{h_1,h_2,h_3  \in \mathbb{Z}} e\big(x(h_3^2-4kh_1h_2)\big)e\big(iy(2kt^{-2}h_1^2+2kt^{2}h_2^2+h_3^2)\big)
\\
&=\Big(y^{1/4}\sum_{h\in \mathbb{Z}} e\big((x+iy)h^2\big)   \Big) \Big(y^{1/2} \sum_{h_1,h_2\in \mathbb{Z}}e\big((-4kxh_1h_2)e(iy(2kt^{-2}h_1^2+2kt^{2}h_2^2)\big)      \Big).
\end{split}
\end{equation}
We note that the first term in the above equation is the elementary theta series in one variable:
$$\theta(z):=y^{1/4}\sum_{h\in \mathbb{Z}} e\big((x+iy)h^2\big).$$
We denote the second term by $\theta_{2}(z,t):= \Big(y^{1/2} \sum_{h_1,h_2\in \mathbb{Z}}e\big((-4kxh_1h_2)e(iy(2kt^{-2}h_1^2+2kt^{2}h_2^2)\big)      \Big). $
By the symmetry between $h_1$ and $h_2$ we have 
$\theta_2(z,t)=\theta_2(z,t^{-1}).$
 By equation \eqref{splittheta}, the Siegel theta kernel $\Theta(z,g_t)$ splits  into the product of two theta series of dimensions $1$ and $2$:
\begin{equation}
\Theta(z,g_t):= \theta(z) \theta_{2}(z,t).
\end{equation}
Let
\begin{equation}\label{thetaeis}
M(s,z):=\int_{0}^{\infty}  \theta_2(x+iy,t) t^{s}dt/t,
\end{equation}
that is the Mellin transform of $\theta_2(z,t)$. By the definition of $\Omega(s)$ in \eqref{mellin}, we obtain
\begin{equation}
\Omega(s)= \int_{\Gamma_{0}(4k) \backslash H}  \overline{f(z)} \theta(z) M(s,z) \frac{dx dy}{y^{2}}.
\end{equation}
Next, we show that $M(s,z)$ is an Eisenstein series of weight zero and level $4k$. We show this by explicit computation. 
 Let 
\begin{equation}
\begin{split}
Q_{z,t}(h_1,h_2):&=8\pi k ix h_1h_2+4\pi k yt^2h_1^2 +4\pi k y t^{-2}h_2^2= 4\pi \Big( (\sqrt{ky}th_1 + \frac{ix\sqrt{k}}{t\sqrt{y}}h_2)^2+\frac{k|z|^2h_2^2}{yt^2}    \Big).
\end{split}
\end{equation}
Then,
$$\theta_2(z,t)=\theta_2(z,t^{-1})=y^{1/2}\sum_{h_1,h_2\in\mathbb{Z}}\exp(-Q_{z,t}(h_1,h_2)).$$
Next, we apply the Poisson summation formula on $h_1$ variable. Let $\hat{\exp}(\xi_1,h_2)$ be the Fourier transform of $\exp(-Q_{z}(h_1,h_2))$ in $h_1$ variable then: 
\begin{equation}
\hat{\exp}(\xi_1,h_2):=\int_{-\infty}^{\infty}\exp (-Q_{z,t}(u,h_2)-2\pi i u \xi_1) du.
\end{equation}
By applying the Poisson summation formula in $h_1$ variable, we obtain
\begin{equation}\label{Poisson}
y^{1/2}\sum_{h_1,h_2\in\mathbb{Z}}\exp(-Q_{z}(h_1,h_2))=y^{1/2}\sum_{\xi_1,h_2\in\mathbb{Z}}\hat{\exp}(\xi_1,h_2).
\end{equation}
 Next, we compute $\hat{\exp}(\xi_1,h_2)$:
 \begin{equation}
 \begin{split}
 \hat{\exp}(\xi_1,h_2)&=\int_{-\infty}^{\infty} \exp \Big(-4\pi\big((\sqrt{ky}tu+ \frac{ix\sqrt{k}}{t\sqrt{y}}h_2)^2+\frac{k|z|^2h_2^2}{yt^2}    \big)-2\pi i u \xi_1     \Big) du
 \\
 &= \frac{1}{2t\sqrt{ky}} \exp\big( -\frac{4\pi}{yt^2} \big| \sqrt{k}zh_2 +\frac{\xi_1}{4\sqrt{k}}\big|^2\big). 
 \end{split}
 \end{equation}
 We use the above formula and equation \eqref{Poisson} to obtain
 \begin{equation}
 \theta_2(z,t^{-1})=\frac{1}{2t\sqrt{k}}\sum_{h_1,h_2 \in \mathbb{Z}} \exp\big( -\frac{4\pi}{yt^2} \big| \frac{h_1}{4\sqrt{k}}+ \sqrt{k}zh_2 \big|^2\big). 
 \end{equation}
 Next, we use the above formula in order to simplify $M(s,z)$ that is defined in \eqref{thetaeis}.   We have
  \begin{equation}
 \begin{split}
 M(s,z)&= \int_{0}^{\infty}  \theta_2(z,t)t^{s}dt/t=\frac{1}{2\sqrt{k}}\int_{0}^{\infty} \sum_{h_1,h_2 \in \mathbb{Z}} \exp\big( -\frac{4\pi t^2}{y} \big| \frac{h_1}{4\sqrt{k}} +\sqrt{k}zh_2\big|^2\big) t^{s+1} dt/t.
 \end{split}
 \end{equation}
 Therefore,
 \begin{equation}
\begin{split}
\Omega(s)&= \int_{\Gamma_{0}(4k) \backslash H}  \overline{f(z)}\theta(z) M(s,z) \frac{dx dy}{y^{2}}
\\
&=\frac{1}{2\sqrt{k}} \int_{0}^{\infty} \int_{\Gamma_{0}(4k) \backslash H}  \overline{f(z)} \theta(z)\sum_{h_1,h_2 \in \mathbb{Z}} \exp\big( -\frac{4\pi t^2}{y} \big| \frac{h_1}{4\sqrt{k} +\sqrt{k}zh_2\big|^2\big) t^{s+1}} dt/t.
\end{split}
 \end{equation}
 Since $ \int_{\Gamma_{0}(4k) \backslash H}  f(z) \bar{\theta}(z) \frac{dx dy}{y^2}=0$, then
 $$\Omega(s)=\frac{1}{2\sqrt{k}} \int_{0}^{\infty} \int_{\Gamma_{0}(4k) \backslash H}  \overline{f(z)} \theta(z)\sum_{h_1,h_2 }^{\quad \prime} \exp\big( -\frac{4\pi t^2}{y} \big| \frac{h_1}{4\sqrt{k}} +\sqrt{k}zh_2\big|^2\big) t^{s+1} dt/t.$$
 where $\sum_{h_1,h_2 }^{\prime}$ is the sum over integers $h_1,h_2 \in \mathbb{Z}$ excluding  $h_1=h_2=0$. Next, we change the variable to 
 $
 \tau:= \frac{2t\sqrt{\pi }}{\sqrt{y}} |\frac{h_1}{4\sqrt{k}}+\sqrt{k}h_2 z |$. Then $t=\frac{\tau \sqrt{y} }{ 2\sqrt{\pi}\big| \frac{h_1}{4\sqrt{k}}+\sqrt{k}h_2z    \big| }$ and  $d\tau/\tau=dt/t.$ Therefore,
 \begin{equation}
 \begin{split}
\int_{0}^{\infty} \sum_{h_1,h_2 }^{\quad \prime} \exp\big( -\frac{4\pi t^2}{y} \big| \frac{h_1}{4\sqrt{k}} +\sqrt{k}zh_2\big|^2\big) t^{s+1} dt/t
&=\Big(\int_{0}^{\infty} \exp(-\tau^2)\tau^{s+1} d\tau/\tau \Big) \sum_{h_1,h_2}^{\quad \prime} \Big(\frac{\sqrt{y} }{ 2\sqrt{\pi}\big| \frac{h_1}{4\sqrt{k}}+\sqrt{k}h_2z    \big| }\Big)^{s+1}
 \\
 &=2^s k^{\frac{s+1}{2}} \pi^{-\frac{s+1}{2}}\Gamma(\frac{s+1}{2})\sum_{h_1,h_2}^{\quad \prime} \Big(\frac{y}{|h_1+4h_2zk|^2}\Big)^{\frac{s+1}{2}}.
 \end{split}
 \end{equation}
 We define
   \begin{equation}\label{element}
E(s,z):=\sum_{h_1,h_2}^{\quad \prime} \Big(\frac{y}{|h_1+4h_2zk|^2}\Big)^{s},  \end{equation}
 Therefore,
  \begin{equation*}
 \Omega(s)= k^{s/2}2^{s-1}\Gamma(\frac{s+1}{2}) \pi^{-\frac{s+1}{2}}  \int_{\Gamma_{0}(4k) \backslash H}  \bar{f}(z) \theta(z) E(\frac{s+1}{2},z) \frac{dx dy}{y^{2}}.
 \end{equation*}
 This completes the proof of the lemma. 
 \end{proof}
Let 
 \begin{equation}\label{Isdef}I(s):=\int_{\Gamma_{0}(4k) \backslash H} \overline{ f(z)} \theta(z) E(\frac{s+1}{2},z) \frac{dx dy}{y^{2}}.\end{equation}
 Hence, 
 $$ \Omega(s)= k^{s/2}2^{s-1}\Gamma(\frac{s+1}{2}) \pi^{-\frac{s+1}{2}}  I(s).
 $$
%
%
%
%
%
 Next, we give an explicit formula for $I(s)$ in term of the Fourier coefficients of $f.$ We begin by writing 
 $E(s,z)$
as a linear combination of Eisenstein series associated to the cusps of $\Gamma_{0}(4k)$. Then by unfolding method we write the  integral $I(s)$ as a Dirichlet series with coefficients associated to the Fourier coefficients of  $f(z)\bar{\theta}(z)$. First we parametrize the cusps of $\Gamma_0(4k)$. We cite   \cite[Proposition 3.1.]{Young}.
 \begin{prop}\cite[Proposition 3.1.]{Young}\label{Ycusp}
 Every  cusp of $\Gamma_0(N)$ is equivalent to one of the form $1/w$ with $1\leq w \leq N$. Two cusps of the form $1/w$ and $1/v$ with $1 \leq v, w \leq N$ are equivalent to each other if and only if 
 \begin{equation}
 (v,N)=(w,N),    \text{  and  } \frac{v}{(v,N)}\equiv \frac{w}{(w,N) }   (\text{  mod } \big((w,N),\frac{N}{(w,N)}   \big)  ).
 \end{equation}
 A cusp of the form $p/q$ is equivalent to one of the form $1/w$ with $w\equiv p^{\prime}q  (\text{ mod } N)$ where $p^{\prime}\equiv p ( \text{ mod } (q,N) )$ and $(p^{\prime},N)=1$. In particular, the cusp at $\infty$ is associated to $w=N$.
 \end{prop}
 For each cusp $\mathfrak{a}\in\mathbb{Q}\cup \{\infty \}$ of a finite covolume discrete subgroup $\Gamma$ of $SL_2(\mathbb{R})$, we call $\sigma_{\mathfrak{a}}\in SL_2(\mathbb{R})$ a scaling matrix for cusp $\mathfrak{a}$ if 
$
 \sigma_{\mathfrak{a}} \infty =\mathfrak{a}
$ and $
 \sigma_{\mathfrak{a}}^{-1} \Gamma_{\mathfrak{a}} \sigma_{\mathfrak{a}} =\Big\{\begin{bmatrix}1 & n \\ 0 & 1   \end{bmatrix}: n \in \mathbb{Z}   \Big\},
$
 where $\Gamma_{\mathfrak{a}}$ is the centralizer of the cusp $\mathfrak{a}$. Note that scaling matrices are not unique. If $\sigma_{\mathfrak{a}}$ is a scaling matrix for $\mathfrak{a}$ so does $\sigma_{\mathfrak{a}}\begin{bmatrix} 1 & \alpha \\ 0 &1 \end{bmatrix}$. We use \cite[Proposition 3.3.]{Young}\label{young1}, where the authors give a representative for scaling matrix $\sigma_{1/w}$ of each cusp $1/w$ of $\Gamma_{0}(N)$. 
 \begin{prop}\cite[Proposition 3.3.]{Young}\label{Ycusp1}
 Let $1/w$ be a cusp of $\Gamma=\Gamma_0(N)$, and set 
 \begin{equation}
 N=(N,w)N^{\prime}_{w},  \qquad w=(N,w)w^{\prime}=(N^{\prime}_w,w)w^{\prime\prime}, \qquad  N^{\prime}=(N^{\prime}_w,w)N^{\prime\prime}_w.
 \end{equation}
 The stabilizer of $1/w$ is given as
 \begin{equation}
 \Gamma_{1/w}=\Big\{\pm\begin{bmatrix}1-w^{\prime\prime}N^{\prime}t & N^{\prime\prime}t \\ -w^{\prime}w^{\prime\prime}Nt & 1+w^{\prime\prime}N^{\prime}t   \end{bmatrix}: t\in \mathbb{Z}   \Big\},
 \end{equation}
  and one may choose the scaling matrix as 
  \begin{equation}
  \sigma_{1/w}=\begin{bmatrix} 1 &0 \\ w &1 \end{bmatrix} \begin{bmatrix} \sqrt{N^{\prime\prime}} & 0 \\ 0& 1/\sqrt{N^{\prime\prime}}   \end{bmatrix}.
  \end{equation}
 \end{prop}
 For each cusp $1/w$ of $\Gamma_0(4k)$, we define the Eisenstein series  $E_{1/w,4k}(s,z)$:
 \begin{equation}
 E_{1/w,4k}(s,z):= \sum_{\gamma\in\Gamma_{1/w}\backslash \Gamma_{0}(4k)} im(\sigma_{1/w}^{-1} \gamma z)^s.
 \end{equation}
 By the spectral theory of $\Gamma_{0}(4k)\backslash H$, the continuous spectrum of the laplacian operator on $\Gamma_{0}(4k)\backslash H$ is spanned by the Eisenstein series associated to the cusps of $\Gamma_0(4k)$. 
In the following lemma, we write $E(s,z)$ that is defined in equation \eqref{element} as a linear combination of 
 $E_{1/w,4k}(s,z)$. 
 \begin{lem}\label{eisenexp}
 Let  $E(s,z)$ and $E_{1/w,4k}(s,z)$ be the Eisenstein series as above. Then
 \begin{equation}
 E(s,z)= \sum_{1/w } \phi_{1/w}(s)E_{1/w,4k}(s,z),
 \end{equation}
 where $\phi_{1/w}(s):= 2\zeta(2s)  \Big( \frac{N^{\prime\prime}_w}{N^{\prime 2}_w}\Big)^s$ with $N^{\prime}_w$ and $N^{\prime\prime}_w$  defined in Proposition \ref{young1}. 
   \end{lem}
 \begin{proof} We note that the Eisenstein series $E_{1/w,4k}(s,z)$ is zero asymptotically at every cusp  for $\Re (s)>1$  except the cusp $1/w,$ where
 $$\lim_{\text{Im}z \to \infty} E_{1/w,4k}(s,\sigma_{1/w} z)= y^s+o(1).$$ Hence, the asymptotic of $ E(s,z)$ at  cusp $1/w$ gives the coefficient of the associated Eisenstein series $E_{1/w,4k}(s,z)$ in the basis of $\left\{E_{1/w,4k}(s,z): w \in \text{cusps of } \Gamma_0(4k)\right\}$ for the continuous spectrum of $\Gamma_0(4k)$. Next, we give the asymptotic of $E(s,z)$ at cusp $1/w$. By definition \ref{element}, we have
 $$E(s,z)=\sum_{h_1,h_2 }^{\quad \prime} \frac{y^s}{|4kh_1z+h_2|^{2s}}.$$
We use the scaling matrix 
 \begin{equation*}
  \sigma_{1/w}=\begin{bmatrix} 1 &0 \\ w &1 \end{bmatrix} \begin{bmatrix} \sqrt{N^{\prime\prime}} & 0 \\ 0& 1/\sqrt{N^{\prime\prime}}   \end{bmatrix},
  \end{equation*}
  that is given in Proposition \ref{young1} in order to compute the asymptotic of $E(s,z)$ at cusp $1/w.$ We have
\begin{equation*}
\begin{split}
E(s,\sigma_{1/w} z)&=\sum_{h_1,h_2 \in \mathbb{Z}} \frac{\text{Im}(\sigma_{1/w} z)^s}{|4kh_1\sigma_{1/w} z+h_2|^{2s}}
\\
&=\sum_{h_1,h_2 \in \mathbb{Z}} \frac{N_w^{\prime\prime s} y^s}{ |wN_{w}^{\prime\prime}z+1|^{2s} |4kh_1\frac{N_w^{\prime\prime}z}{wN_w^{\prime\prime}z+1} +h_2|^{2s}}
\\
&=\sum_{h_1,h_2 \in \mathbb{Z}} \frac{N_w^{\prime\prime s}y^s}{|4kh_1N^{\prime\prime}_{w}z+h_2(wN_{w}^{\prime\prime}z+1) |^{2s}}
\\
&=\zeta(2s)\sum_{\gcd(h_1,h_2)=1} \frac{N_w^{\prime\prime s}y^s}{|4kh_1N^{\prime\prime}_{w}z+h_2(wN_{w}^{\prime\prime}z+1) |^{2s}}.
\end{split}
\end{equation*}
 We note that as $\text{Im}(z)  \to \infty$ then all the terms in the above sum goes to zero except $h_1$ and $h_2$ such that the coefficient of $z$ in the denominator is zero, that is 
 $$4kh_1N^{\prime\prime}_w+h_2wN^{\prime\prime}_w=0.$$
 Since $\gcd(h_1,h_2)=1$ then $h_2=\pm \frac{4k}{\gcd(w,4k)}=N^{\prime}_w$ by the notation of the Proposition \ref{young1}. Therefore,
\begin{equation}
\lim_{\text{Im}(z)\to \infty} E(s,\sigma_{1/w}z)=2\zeta(2s)\frac{N^{\prime\prime s}_w}{N^{\prime 2s}_w}.
\end{equation}
As a corollary,
\begin{equation}
 E(s,z)= \sum_{1/w \in \text{cusp of } \Gamma_{0}(4k)} 2\zeta(2s)  \Big( \frac{N^{\prime\prime}_w}{N^{\prime 2}_w}\Big)^s E_{1/w,4k}(s,z).
\end{equation}
 This completes the proof of our lemma.
 \end{proof}

 \subsubsection{Fourier expansion of the Jacobi function at every cusp of $\Gamma_0(4k)$:} 
 In this section we give the Fourier expansion of the classical Jacobi theta series at each cusp of $\Gamma_0(4k)$. We note that the Fourier expansion of the Jacobi theta series at $\infty$ is 
 \begin{equation}
 \theta(z):=y^{1/4}\sum_{n\in \mathbb{Z}} e(n^2z).
 \end{equation}
 $\theta(z)$ is a weight $1/2$ modular form  invariant by $\Gamma_0(4)$ that has $3$ inequivalent cusp $\infty$, $0$ and $1/2$. Hence, it suffices to give the Fourier expansion of $\theta(z)$ at $1/2$ and $0$. We use the the Following scaling matrices for $\Gamma_0(4).$ We let 
\begin{equation*}\tau_0:=\begin{bmatrix} 0 & -1/2 \\ 2 & 0 \end{bmatrix},
\qquad
\tau_{1/2}:=\begin{bmatrix}1 & -1/2 \\ 2 & 0    \end{bmatrix},
\end{equation*}
 where $\tau_0$ and $\tau_{1/2}$ are scaling matrices for cusps $0$ and $1/2$ of $\Gamma_0(4)$. The Fourier expansion of $\theta(z)$ at cusp $0$ is given by expanding $\theta_{\tau_0}$ that is
 $$\theta_{\tau_0}(z)= \big(\frac{z}{|z|}  \big)^{{-1/2}}\theta(-1/4z)$$
  at $\infty.$ We use the following formula from \cite[equation (2.4)]{katok}
 \begin{equation}\label{tet1}
\theta(z)_{\tau_0}=e^{-i\pi/4}\theta(z). 
\end{equation} 
Next, we give the Fourier expansion of $\theta(z)$ at cusp $1/2$. We have
\begin{equation*}
\begin{split}
\theta(\tau_{1/2} z)&=\text{Im}(\tau_{1/2} z)^{1/4}\sum_{n\in \mathbb{Z}} e(n^2(\tau_{1/2}z))
\\
&=\frac{y^{1/4}}{|2z|^{1/2}}\sum_{n\in \mathbb{Z}} e\big(n^2(1/2-1/(4z))\big)
\\
&=\frac{y^{1/4}}{|2z|^{1/2}}\sum_{n\in \mathbb{Z}} (-1)^n e\big(-n^2/(4z)\big)
\\
&=\frac{y^{1/4}}{|2z|^{1/2}}\Big(  2\sum_{n \text{  even}} e\big(-n^2/(4z)\big)-\sum_{n\in \mathbb{Z}}e(-n^2/4z)  \Big)
\\
&=\frac{y^{1/4}}{|2z|^{1/2}}\Big(  2\sum_{n \in \mathbb{Z}} e\big(-n^2/z\big)-\sum_{n\in \mathbb{Z}}e(-n^2/4z)  \Big)
\\
&=\sqrt{2}\theta(-1/z)-\theta(-1/4z).
\end{split}
\end{equation*}
We use the transformation formula of $\theta(z)$ under $\gamma_2:=\begin{bmatrix}0& -1 \\ 1 & 0   \end{bmatrix}$; see \cite[Page 202]{katok}
\begin{equation}\label{tet2}
\theta(-1/z)=i^{-1/2}\big(\frac{z}{|z|}\big)^{1/2}  \frac{\theta(z)+ \theta(z+1/2)  }{\sqrt{2}}.
\end{equation}
By equations \ref{tet1} and \ref{tet2}, we have
\begin{equation}\label{tau2}
\begin{split}
\theta(\tau_{1/2} z)&=e^{-\pi/4}\big(\frac{z}{|z|}\big)^{1/2} \Big(    \theta(z)+ \theta(z+1/2)  -\theta(z) \Big)=e^{-\pi/4}\big(\frac{z}{|z|}\big)^{1/2}  \theta(z+1/2).   
\end{split}
\end{equation}

We note that $\theta_{\sigma_{1/w}}$ is invariant under $\Gamma_{\infty}$. Hence, we have
\begin{equation}
\theta_{\sigma_{1/w}}(z):=y^{1/4}\sum_{n \in \mathbb{Z}} b_{\theta,1/w}(n) e(nz),
\end{equation}
where $b_{\theta,1/w}(n)$ is the $n$th Fourier coefficient of $\theta(z)$ at cusp $1/w$ associated to scaling matrices $\sigma_{1/w}$.
In the following lemma, we give the Fourier coefficients of $\theta(z)$.
\begin{lem}\label{boundtet}
Let $\theta(z)=y^{1/4}\sum_{n\in \mathbb{Z}} e(n^2z)$ and $\sigma_{1/w}$ be the scaling matrices introduced above. Then $\theta(z)$ has the following Fourier coefficients for each cusp $1/w$ of $\Gamma_0(4k)$. If $w\equiv 0 \text{ mod } 4$ then
\begin{equation}
\begin{split}
\theta_{\sigma_{1/w}}&=\theta(N^{\prime\prime}_{1/w}z),
\\
|b_{\theta,1/w}(n)|:&=\begin{cases}\big(N^{\prime\prime}_{1/w}\big)^{1/4}, \quad  &\text{ if }n=m^2N^{\prime\prime}_{1/w} \text{ for some } m \in \mathbb{Z},  \\ 0, \quad &\text{ Otherwise.}\end{cases}
\end{split}
\end{equation}
If $w\equiv \pm 1 \text{ mod } 4$ then $N^{\prime\prime}_{1/w}=4\alpha$ and
\begin{equation}
\begin{split}
\theta_{\sigma_{1/w}}(z)&= \theta(\alpha z\pm 1/4),
\\
|b_{\theta,1/w}(n)|:&=\begin{cases}\alpha^{1/4}, \quad  &\text{ if }n=m^2\alpha \text{ for some } m \in \mathbb{Z}, \\ 0, \quad &\text{ Otherwise.}\end{cases}
\end{split}
\end{equation}
Finally if $w\equiv 2 \text{  mod } 4$
\begin{equation}
\begin{split}
\theta_{1/w}(z)&=   \theta(N^{\prime\prime}_{1/w}z),
\\
|b_{\theta,1/w}(n)|:&=\begin{cases}\big(N^{\prime\prime}_{1/w}\big)^{1/4}, \quad  &\text{ if }n=m^2N^{\prime\prime}_{1/w} \text{ for some } m \in \mathbb{Z}, \\ 0, \quad &\text{ Otherwise.}\end{cases}
\end{split}
\end{equation}
\end{lem}
\begin{proof}
We note that $\theta(z)$ is invariant under $\Gamma_0(4)$ and $\Gamma_0(4)$ has 3 cusps $\{0,1/2, \infty\}$.  If $w\equiv 0 \text{ mod } 4$ then the cusp $1/w$ is equivalent to $\infty$ in $\Gamma_0(4)$ and the Fourier expansion of $\theta_{\sigma_{1/w}}$ is given by the following identity 
$$\theta_{\sigma_{1/w}}(z)=\theta(N^{\prime\prime}_{1/w}z).$$
If $w=4\alpha+2$ then $1/w$ is equivalent to $1/2$ in $\Gamma_0(4)$ and we have
\begin{equation}
\sigma_{1/w}=\begin{bmatrix} 1& 0 \\ 4\alpha & 1  \end{bmatrix} \tau_{1/2} \begin{bmatrix}1 & 1/2 \\ 0 & 1 \end{bmatrix} \begin{bmatrix}\sqrt{N^{\prime\prime}_{1/w}} & 0 \\ 0 & 1/\sqrt{N^{\prime\prime}_{1/w}}  \end{bmatrix}.
\end{equation}
By the above decomposition and equation \ref{tau2}, we have 
\begin{equation}
\theta_{1/w}(z)=   \theta(N^{\prime\prime}_{1/w}z).
\end{equation}
If $w=4\alpha+1$ then $1/w$ is equivalent to $0$ in $\Gamma_0(4)$ and we have
\begin{equation}
\sigma_{1/w}=\begin{bmatrix}1 & 1 \\ 4\alpha & 4\alpha+1    \end{bmatrix} \tau_0 \begin{bmatrix} 1 & 1/4 \\ 0 & 1   \end{bmatrix} \begin{bmatrix} \sqrt{N^{\prime\prime}/4} & 0 \\ 0 & 1/\sqrt{N^{\prime\prime}/4}  \end{bmatrix}.
\end{equation}
By the above decomposition and equation \ref{tet1}, we have
\begin{equation}
\theta_{1/w}(z)= \theta(N^{\prime\prime}z/4+1/4).
\end{equation}
Finally if $w=4\alpha+3$ then $1/w$ is equivalent to $0$ in $\Gamma_{0}(4)$ and we have
\begin{equation}
\sigma_{1/w}=\begin{bmatrix}-1 & 1 \\ -4(\alpha+1) & 4\alpha+3  \end{bmatrix} \tau_0 \begin{bmatrix} 1 & -1/4 \\ 0 & 1   \end{bmatrix} \begin{bmatrix} \sqrt{N^{\prime\prime}/4} & 0 \\ 0 & 1/\sqrt{N^{\prime\prime}/4}  \end{bmatrix}.
\end{equation}
By the above decomposition and equation \ref{tet1}, we have
\begin{equation}
\theta_{1/w}(z)= \theta(N^{\prime\prime}z/4-1/4).
\end{equation}
This completes the proof of our lemma. 
\end{proof}
Note that $f_{\sigma_{1/w}}$ is invariant under $\Gamma_{\infty}=\{\begin{bmatrix} 1 & n \\ 0 & 1  \end{bmatrix}: n \in \mathbb{Z}   \}.$ So, we can write the Fourier expansion of $f_{\sigma_{1/w}}$ at $\infty$ and obtain
\begin{equation}\label{fof}
f_{\sigma_{1/w}}:= \sum_{n\neq0}b_{f,1/w}(n)W_{1/4 \text{sgn}(n),ir}(4\pi|n|y)e(nx).
\end{equation}
Next, we apply Hardy's method in order to give the trivial bound on $b_{f,1/w}(n)$ the $n$th Fourier coefficient of $f$ at cusp $1/w$. This method was implemented by Matthes for real analytic cusp forms \cite[Page 157]{Matthes}.
\begin{lem}\label{bound1/2}
 Let $f$ be a weight $1/2$ modular form defined on $\Gamma_0(4k)$ with Laplacian eigenvalue $1/4+r^2$ and $|f|_2=1$. Then we have 
$$ |b_{f,1/w}(m)| \ll r^{\frac{(1-1/4\text { sgn}(m))}{2}} e^{\frac{\pi r}{2}} N^{\prime\prime 1/2}_{1/w} (1+O(|r|^{-1})),   $$
\end{lem}
\begin{proof}
 Let 
$\Lambda_{y_0}:=\{z=x+iy: |x|<1/2 \text{ and } y\geq y_0   \}.$
For each $z\in H$, we denote the number of elements of the orbit of $z$ by the discrete group $ \sigma_{1/w}^{-1} \Gamma_0(4k) \sigma_{1/w}$ that lies inside $\Lambda_{y_0}$ by
$
N(z,1/w,y_0).
$
 For each cusp $1/w$ of $\Gamma_0(4k)$, let 
\begin{equation}
c_{1/w}:=\min \left\{c >0 :\begin{bmatrix} * & * \\ c &*  \end{bmatrix}\in \sigma_{1/w}^{-1} \Gamma_0(4k) \sigma_{1/w}   \right\}.
\end{equation}
By definition of $\sigma_{1/w}$ in Proposition \ref{young1}, it is easy to check that $c_{1/w} \in 1/N^{\prime\prime}_{1/w}\mathbb{Z}$. Hence,
$
c_{1/w}\geq 1/N^{\prime\prime}_{1/w}.
$
By \cite[Lemma 2.10]{Iwaniec2}, we have the following upper bound on $N(z,1/w,y_0)$
\begin{equation}\label{Ny0}
\begin{split}
N(z,1/w,y_0) &\leq 1+ \frac{10}{c_{1/w} y_0}\leq 1+ \frac{10 N^{\prime\prime}_{1/w}}{y_0}.
\end{split}
\end{equation}
By inequality \ref{Ny0} and $|f|^2=1$, we have
\begin{equation}\label{inq1}
\begin{split}
 \int_{\Lambda_{y_0}} |f(\sigma_{1/w} z)|^2 d\mu(z)&=\int_{\sigma_{1/w}^{-1} \Gamma_0(4k) \sigma_{1/w} \backslash H} N(z,1/w,y_0)|f(\sigma_{1/w} z)|^2 d\mu(z)
 \leq \Big( 1+\frac{10N^{\prime\prime}_{1/w}}{y_0} \Big).
\end{split}
\end{equation}
Next,  for each $m\in \mathbb{Z}$, we give an upper bound on $|b_{f,1/w}(m)|$, the $m$th Fourier coefficient of $f$ at cusp $1/w$ defined in equation \ref{fof}
\begin{equation}\label{inq2}
\begin{split}
 \int_{\Lambda_{y_0}} |f(\sigma_{1/w} z)|^2 d\mu(z)&= \sum_{n\neq0}|b_{f,1/w}(n)|^2 \int_{y_0}^{\infty} |W_{1/4 \text{sgn}(n),ir}(4\pi|n|y)|^2dy/y^2
 \\
 &= \sum_{n\neq0}|b_{f,1/w}(n)|^2 4\pi|n| \int_{4\pi|n|y_0}^{\infty} |W_{1/4 \text{sgn}(n),ir}(u)|^2du/u^2
 \\
&\geq |b_{f,1/w}(m)|^2 4\pi|m|  \int_{4\pi|m|y_0}^{\infty} |W_{1/4 \text{sgn}(m),ir}(u)|^2du/u^2.
 \end{split}
\end{equation} 
 We take $y_0:=(4\pi |m|)^{-1}$ then by inequalities \eqref{inq1} and \eqref{inq2} we have
 \begin{equation}\label{inq}
  |b_{f,1/w}(m)|^2   \int_{1}^{2} |W_{1/4 \text{sgn}(m),ir}(u)|^2du/u^2 \ll N^{\prime\prime}_{1/w}.
 \end{equation}
For $t\to \infty$ and  bounded $y$, we have 
\begin{equation}\label{whitt}
W_{ \text{sgn}(m)1/4, ir } (y)= \Big(\frac{\Gamma(-2ir)}{\Gamma(1/2-\mu-\text{sgn}(m)1/4)}y^{1/2+ir} +\frac{\Gamma(2ir)}{\Gamma(1/2+2ir-\text{sgn}(m)1/4)}y^{1/2-ir}   \Big)(1+O(t^{-1})).
\end{equation} 
By Stirling formula, we have
\begin{equation}
\Gamma(x+iy)=\sqrt{2\pi}y^{x-1/2}e^{-\pi|y|/2} (1+O(|y|^{-1})),   \quad x \text{ bounded,}
\end{equation}
By using the above asymptotic formula, equation \eqref{whitt} and \eqref{inq}, we have
\begin{equation}
 |b_{f,1/w}(m)|^2 \ll r^{1-1/4\text { sgn}(m)} e^{\pi r} N^{\prime\prime}_{1/w} (1+O(|r|^{-1})), 
\end{equation}
with an absolute constant. This completes the proof of our lemma.
\end{proof}
Finally, we compute the integral $I(s)$ defined in equation \ref{Isdef}.  By Lemma \ref{eisenexp} and unfolding method we simplify the right hand side.
\begin{lem}\label{4.7}
We have 
\begin{equation}
I(s)=\psi(s)\sum_{n\geq 1} \frac{\rho(n)}{n^{s-1/2}}.
\end{equation}
where 
\begin{equation}
\rho(n):= \frac{1}{\sqrt{2}}\sum_{w \text{ odd}} \frac{N^{\prime\prime}}{N^{\prime}{}^{3/2}}\overline{b_{f,1/w} \big((\frac{2n}{N^{\prime}_{w}})^2N^{\prime\prime}_w/4\big)}+\sum_{w \text{ even }} \frac{N^{\prime\prime}}{N^{\prime}{}^{3/2}} \overline{b_{f,1/w}\big((\frac{n}{N^{\prime}_w})^2 N^{\prime\prime}   \big)},
\end{equation}
and 
\begin{equation}
\psi(s):=2\zeta(s+1)(4\pi)^{-(s/2-1/4)} \frac{\Gamma(s/2+1/4+ir)\Gamma(s/2+1/4-ir)}{\Gamma(\frac{s+1}{2})}. 
\end{equation}

\end{lem}
\begin{proof}

\begin{equation}
\begin{split}
I(s):&=\int_{\Gamma_{0}(4k) \backslash H}  \overline{f(z)} \theta(z) E(\frac{s+1}{2},z) d\mu(z)
\\
&=\int_{\Gamma_0(4k)\backslash H}  \overline{f(z)}\theta(z) \sum_{1/w \in \text{cusps}} \phi_{1/w}(\frac{s+1}{2}) E_{1/w,4k}(\frac{s+1}{2},z)d\mu(z)
\\
&= \sum_{1/w \in \text{cusps}} \phi_{1/w}(\frac{s+1}{2}) \int_{\Gamma_0(4k)\backslash H}  \overline{f(z)}\theta(z)   E_{1/w,4k}(\frac{s+1}{2},z)d\mu(z)
\\
&= \sum_{1/w \in \text{cusps}} \phi_{1/w}(\frac{s+1}{2}) \int_{\Gamma_0(4k)\backslash H}  \overline{f(z)}\theta(z) \sum_{\gamma\in \Gamma_{1/w} \backslash \Gamma_0(4k)} \text{Im}(\sigma_{1/w}^{-1} \gamma z)^{\frac{s+1}{2}} d\mu(z)
\\
&=\sum_{1/w \in \text{cusps}} \phi_{1/w}(\frac{s+1}{2}) \int_{\Gamma_{1/w}\backslash H}  \overline{f(z)}\theta(z)  \text{Im}(\sigma_{1/w}^{-1} z)^{\frac{s+1}{2}} d\mu(z)
\\
&=\sum_{1/w \in \text{cusps}} \phi_{1/w}(\frac{s+1}{2}) \int_{\Gamma_{\infty}\backslash H}  \overline{f(\sigma_{1/w} z)}\theta(\sigma_{1/w} z)  y^{\frac{s+1}{2}} d\mu(z)
\\
&=\sum_{1/w \in \text{cusps}} \phi_{1/w}(\frac{s+1}{2}) \int_{\Gamma_{\infty}\backslash H} \overline{f(z)}_{\sigma_{1/w} }\theta( z)_{\sigma_{1/w} }  y^{\frac{s+1}{2}} d\mu(z).
\end{split}
\end{equation}
By Lemma \ref{boundtet} and \ref{bound1/2}, we write $I(s)$ as a Dirichlet series 
\begin{equation}
\begin{split}
I(s)&=\sum_{1/w \in \text{cusps}} \phi_{1/w}(\frac{s+1}{2}) \sum_{n >0}\overline{b_{f,1/w}(n)} b_{\theta,1/w}(n) \int_{0}^{\infty} \overline{W_{1/4,ir}(4\pi|n|y)}\exp(-2\pi ny) y^{s/2-1/4} dy/y
\\
&= \sum_{1/w \in \text{cusps}} \phi_{1/w}(\frac{s+1}{2}) \sum_{n >0} \frac{\overline{b_{f,1/w}(n)} b_{\theta,1/w}(n)}{n^{s/2-1/4}} \int_{0}^{\infty}  \overline{W_{1/4,ir}(4\pi u)}\exp(-2\pi u) u^{s/2-1/4} du/u
\\
&= (4\pi)^{-(s/2-1/4)} \frac{\Gamma(s/2+1/4+ir)\Gamma(s/2+1/4-ir)}{\Gamma(\frac{s+1}{2})} \sum_{1/w \in \text{cusps}} \phi_{1/w}(\frac{s+1}{2}) \sum_{n >0} \frac{\overline{b_{f,1/w}(n)} b_{\theta,1/w}(n)}{n^{s/2-1/4}}
\\
&=\psi(s)\sum_{n\geq 1} \frac{\rho(n)}{n^{s-1/2}},
\end{split}
\end{equation}
This completes the proof of the lemma. 
\end{proof}
\begin{cor}
By Lemma~\ref{4.7} and Lemma~\ref{seesaw}, we have
\begin{equation}\label{omega11}
\begin{split}
\Omega(s)=\sqrt{2}\pi^{-s-1/4}\zeta(s+1) \Gamma(s/2+1/4+ir)\Gamma(s/2+1/4-ir)k^{s/2}\sum_{n\geq 1} \frac{\rho(n)}{n^{s-1/2}}.
\end{split}
\end{equation}
\end{cor}
\subsection{Bounding the $L^2$ norm of the theta transfer}
Recall that $f$ is a  weight 1/2 modular form on $\Gamma_{0}(4k)\backslash H$ with eigenvalue $1/4+r^2$ and $|f|_2=1,$ and
$$\varphi(g):=\int_{\Gamma_{0}(4k) \backslash H} \Theta(x+iy,g) \overline{f(x+iy)} \frac{dx dy}{y^{2}}.$$
In the following theorem, we give an upper bound on  the $L^2$ norm of $\varphi$. 
\begin{thm}\label{Rallisthm}
Let $f$, $\varphi$ and $r$ be as above. Then $\varphi$ can be realized as  a Maass form of weight 0  on $\Gamma_{0}(k)\backslash H$. Moreover 
\begin{equation}
|\varphi|_2\ll  \cosh(-\pi r/2) (kr)^{9},
\end{equation}
where the constant in $\ll$ is absolute. 
\end{thm}

\begin{proof}
 Recall that $\Theta(z,g)$ is $\Gamma$ invariant from the left and $K$ invariant from the right in $g$ variable.
By Theorem~\ref{dtrans},  $\varphi(g)$ is a Maass form of weight 0 on $\Gamma_k\backslash V_{m,k}$ by
$\varphi(\vec{v}):=\varphi(g_v),$
where $\vec{v}\in V_{m,k}$ and $g_v \in SO_{q_k}$ is an element such that $g_v \vec{x}_0 =\vec{v}$. Define the involution 
$\tau: SO_{q_k}\to SO_{q_k} $ by  
$$\tau(g)=\begin{bmatrix}1 &0 &0 \\ 0& 1& 0 \\0& 0 &-1   \end{bmatrix} g \begin{bmatrix}1 &0 &0 \\ 0& 1& 0 \\0& 0 &-1   \end{bmatrix}. $$
 By definition of theta series at \eqref{theta}, it is easy to check that $\Theta(z,g)=\Theta(z,\tau(g)).$
 As a result $\varphi(g)=\varphi(\tau(g))$ and this means that $\varphi$ is an even Maass form on $\Gamma_k\backslash V_{m,k}.$ Recall the isomorphism between $PSL_2(\mathbb{R})$ and $SO_{q_k}$ that we introduced in \eqref{isom}:
 \begin{equation}
\gamma \in PSL_2(\mathbb{R}) \to g_{\gamma}\in SO_{q_1} \to B_k^{-1} g_{\gamma}B_k\in SO_{q_k},
\end{equation}
where
 $B_k:=\begin{bmatrix}
k& 0 & 0
\\ 
0&1&0
\\
0&0&1
\end{bmatrix}.
$ Recall that we introduced  an isomorphism between $\Gamma_k\backslash V_{m,k}$ and $\Gamma^{\prime}\backslash H$ in \eqref{eqisom},  where $\Gamma_0(k)\subset\Gamma^{\prime},$ by:
 \begin{equation}
 \vec{a}:=\begin{bmatrix} a_1\\a_2\\a_3\end{bmatrix}\in V_{m,k} \to z_{\vec{a}}=\frac{-a_3+i\sqrt{|m|}}{2ka_1} \in H.
 \end{equation}

 As a result, we define the even Maass form $u(z)$ with the Laplacian eigenvalue $1/4+(2r)^2$ on the congruence curve $\Gamma_0(k)\backslash H$ by:
\begin{equation}   
u(z_{\vec{a}}):=\varphi(\vec{a}).
\end{equation}
  Next, we relate the coefficients of  $\Omega(s)$ defined in \eqref{Mel} to the Fourier coefficients of $u(z)$ at the cups $\infty$ of $\Gamma_0(k)$. Recall that
$$
\Omega(s):= \int_{0}^{\infty} \varphi(g_t)t^{s} \frac{dt}{t},
$$
where
$g_t=\begin{bmatrix} t & 0 &0 \\ 0& t^{-1} & 0 \\ 0 & 0& 1   \end{bmatrix}\in G.$
By equation \eqref{eqisom}, $z_{\vec{x}_0}= i/\sqrt{k}$. Moreover, by isomorphism \eqref{isom} $$g_t=\begin{bmatrix} t & 0 &0 \\ 0& t^{-1} & 0 \\ 0 & 0& 1   \end{bmatrix} \to \begin{bmatrix}\sqrt{t} & 0 \\ 0 & \sqrt{t}^{-1} \end{bmatrix}\in SL_2(\mathbb{R}). $$
Hence,  $\varphi(g_t)=u(it/\sqrt{k})$ and as a result
$$\Omega(s)=\int_{t=0}^{\infty} u(it/\sqrt{k})t^{s} \frac{dt}{t}.$$
$u(z)$ is an even  Maass form with eigenvalue $1/4+(2r)^2$ on $\Gamma_0(k)$, we write the Fourier expansion of $u$ at $\infty$ and obtain
\begin{equation}
u(x+iy)=2\sum_{n= 1}^{\infty} a_{u}(n)n^{-1/2}\cos(2\pi nx)W_{0,2ir}(4\pi ny),
\end{equation}
where $W_{0,2ir}(y)$ is the usual Whittaker function which is normalized so that
$W_{\beta,\mu}(y)\approx e^{-y/2}y^{\beta}  \text{ as }  y \to \infty.$
By using the above expansion, we have
\begin{equation}\label{omega2}
\begin{split}
\Omega(s)&=2\int_{t=0}^{\infty} \sum_{n=1}^{\infty} a_{u}(n)n^{-1/2}W_{2ir}(4\pi n t/\sqrt{k})t^{s} \frac{dt}{t}
\\
&= 2k^{s/2}\pi^{-s}\sum_{n= 0}^{\infty} \frac{a_{u}(n)}{n^{s+1/2}} \int_{t=0}^{\infty} W_{2ir}(4t)t^{s}\frac{dt}{t}
\\&=k^{s/2}\pi^{-s-1/2}\Gamma(\frac{s+1/2+2ir}{2})\Gamma(\frac{s+1/2-2ir}{2}) \sum_{n=1}^{\infty} \frac{a_{u}(n)}{n^{s+1/2}}.
\end{split}
\end{equation}
where we used 
\begin{equation}
\int_{0}^{\infty} W_{2ir}(4u)u^s \frac{du}{u}=\frac{\pi^{-1/2}}{2}\Gamma(\frac{s+1/2+2ir}{2})\Gamma(\frac{s+1/2-2ir}{2}),
\end{equation}
from \cite{Grad}. 
\\
\\
By  the equations \eqref{omega11} and \eqref{omega2}, we obtain 
\begin{equation}
\begin{split}
&\sqrt{2}\pi^{-s-1/4}\zeta(s+1) \Gamma(s/2+1/4+ir)\Gamma(s/2+1/4-ir)k^{s/2}\sum_{n\geq 1} \frac{\rho(n)}{n^{s-1/2}}
\\
&=k^{s/2}\pi^{-s-1/2}\Gamma(\frac{s+1/2+2ir}{2})\Gamma(\frac{s+1/2-2ir}{2}) \sum_{n=1}^{\infty} \frac{a_{u}(n)}{n^{s+1/2}}.
\end{split}
\end{equation}
Hence,
\begin{equation}
\sqrt{2}\pi^{1/4}\zeta(s+1)\sum_{n\geq 1} \frac{\rho(n)}{n^{s-1/2}}=\sum_{n=1}^{\infty}\frac{a_{u}(n)}{n^{s+1/2}}.
\end{equation}
Therefore,
\begin{equation}
\begin{split}
a_{u}(n)&=n^{1/2}\sqrt{2}\pi^{1/4} \sum_{lm=n}l^{-1}m^{1/2}\rho(m),
\end{split}
\end{equation}
where 
$$
\rho(m):= \frac{1}{\sqrt{2}}\sum_{w \text{ odd}} \frac{N^{\prime\prime}}{N^{\prime}{}^{3/2}}\overline{b_{f,1/w} \big((\frac{2m}{N^{\prime}_{w}})^2N^{\prime\prime}_w/4\big)}+\sum_{w \text{ even }} \frac{N^{\prime\prime}}{N^{\prime}{}^{3/2}} \overline{b_{f,1/w}\big((\frac{m}{N^{\prime}_w})^2 N^{\prime\prime}   \big)}.$$
By Lemma \ref{bound1/2} and Proposition~\ref{young1},  we have 
\begin{equation}
\begin{split}
\rho(m) &\leq \frac{1}{\sqrt{2}}\sum_{w \text{ odd}} \frac{N^{\prime\prime}}{N^{\prime}{}^{3/2}}|b_{f,1/w} \big((\frac{2m}{N^{\prime}_{w}})^2N^{\prime\prime}_w/4\big)|+\sum_{w \text{ even }} \frac{N^{\prime\prime}}{N^{\prime}{}^{3/2}} |b_{f,1/w}\big((\frac{m}{N^{\prime}_w})^2 N^{\prime\prime}   \big)|
\\
&\ll r^{5/8}e^{\pi r/2} \sum_{w \text{ Cusp of } \Gamma_0(4k)} \frac{N^{\prime\prime 3/2}}{N^{\prime}{}^{3/2}}  \ll  r^{5/8}e^{\pi r/2} k^{\epsilon}.
\end{split}
\end{equation}
Therefore,
\begin{equation}\label{lastinq}
\begin{split}
|a_{u}(n)| & \ll n^{1/2} \sum_{lm=n}l^{-1}m^{1/2}|\rho(m)|
 \ll n^{1+\epsilon}\max_{ 1\leq m\leq n}|\rho(m)|
 \ll n^{1+\epsilon} k^{\epsilon} r^{5/8}e^{\pi r/2}.
\end{split}
\end{equation}
Recall that $u$ is a Maass form of weight 0 on the congruence group $\Gamma_0(k).$ We use \cite[(8.17)]{Iwaniec2}, and obtain 
\begin{equation}
\sum_{|n| \leq X } |\nu_{u}(n)|^2=8[SL_2(\mathbb{Z}):\Gamma_0(k)]^{-1}X|\varphi|_2^2+O(rX^{7/8}|\varphi|_2^2).
\end{equation}
where $\nu_{u}(n)=\big(\frac{4\pi}{\cosh 2\pi r }\big)^{1/2}a_{u}(n).$ We have 
$$[SL_2(\mathbb{Z}):\Gamma_0(k)]=k\prod_{p|k}(1+1/p) \leq k\log(k).$$
Let $(kr)^{8+\epsilon}<X$ then the main term $8[SL_2(\mathbb{Z}):\Gamma_0(k)]^{-1}X|\varphi|_2^2$ dominates the error term $O(rX^{7/8}|\varphi|_2^2)$ and we obtain 
\begin{equation}\label{l2phi}
|\varphi|_2^2 \ll \frac{k^{1+\epsilon}}{X}\sum_{|n|\leq X} |\nu_{u}(n)|^2.
\end{equation}
By inequality \eqref{lastinq}, we have
\begin{equation}
\begin{split}
|\nu_{u}(n)|^2&=\big(\frac{4\pi}{\cosh 2\pi r }\big)|a_{u}(n)|^2 \ll
 \cosh(-\pi r) n^{2+\epsilon} k^{\epsilon} r^{5/4}.
\end{split}
\end{equation}
 We apply  the above inequality in \eqref{l2phi} and obtain
\begin{equation}\label{finq}
\begin{split}
|\varphi|_2^2 &\ll \cosh(-\pi r) k^{1+\epsilon} r^{5/4} \frac{1}{X}\sum_{1\leq n \leq X} n^{2+\epsilon}
\ll  \cosh(-\pi r) k^{1+\epsilon} r^{5/4} X^{2+\epsilon}.
\end{split}
\end{equation}
By choosing $X=(kr)^{8+\epsilon}$, we deduce that
$$
|\varphi|_2^2 \ll \cosh(-\pi r) k^{17+\epsilon} r^{18}.
$$
This completes the proof of our Theorem.  
\end{proof}
\noindent \textit{Acknowledgment.} 
I would like to thank Professors Heath-Brown, Radziwill and Soundararajan for several insightful and inspiring conversations during the Spring~2017 at MSRI. In fact, Theorem~\ref{positive} is inspired by the ideas that were  developed in my discussions with Professor Heath-Brown.  Professors Radziwill and Soundararajan kindly outlined the proof of Lemma~\ref{selower}.   Furthermore, I would like to thank Professor Rainer Schulze-Pillot for his comments regarding the Siegel mass formula.  I am also very thankful to Professors Peter Sarnak, Simon Marshall,   and Asif Ali Zaman   for their comments and encouragement.  This material is based upon work supported by the National Science Foundation under Grant No.DMS-1902185 and  Grant No. DMS-1440140 while the author  was in residence at the Mathematical Sciences Research Institute in Berkeley, California, during the Spring 2017 semester.
\bibliographystyle{alpha}
\bibliography{leastprime}

\end{document}